\documentclass[11pt,letterpaper]{amsart}
\usepackage{comment}
\usepackage[dvipsnames]{xcolor}
\usepackage{xcolor}
\usepackage[title]{appendix}
\usepackage{todonotes}
\usepackage{pst-all}
\usepackage{pgfplots}
\usepackage{pst-plot}
\usetikzlibrary{calc}

\usepackage{xcolor}
\usepackage{tikz}
\usepackage{tkz-euclide}
\usepackage[shortcuts]{extdash}
\usepackage{bbm}
\usepackage{hyperref}
\hypersetup{
    colorlinks=true,              
    breaklinks=true,              
    urlcolor= black,              
    linkcolor= black,              
    bookmarksopen=false,
    filecolor=black,
    citecolor=blue,
    linkbordercolor=blue}

\usepackage[utf8]{inputenc}
\usepackage[T1]{fontenc}
\usepackage{textcomp}

\usepackage{chngcntr}

\usepackage{mathrsfs}
\usepackage{nicefrac}

\usepackage{amsmath,amsthm,amssymb, mathtools}

\usepackage{bbm}

\usepackage{graphicx}
\usepackage{faktor}
\newcommand{\A}{\mathcal{A}}
\newcommand{\U}{\mathcal{U}}

\newcommand{\R}{\mathbb{R}}
\newcommand{\ver}{\mathcal{Z}}

\newcommand{\C}{\mathbb{C}}

\newcommand{\Q}{\mathbb{Q}}
\newcommand{\N}{\mathbb{N}}

\newcommand{\B}{\mathcal{B}}
\newcommand{\Z}{\mathbb{Z}}

\newcommand{\GL}{\normalfont \text{GL}}

\newcommand{\ka}
{\kappa}

\newcommand{\al}
{\alpha}

\newcommand{\be}
{\beta}

\newcommand{\vv}{\normalfont \textbf{v}}

\newcommand{\SL}{\normalfont \text{SL}}

\newcommand{\SO}{\normalfont \text{SO}}

\newcommand{\SP}{\normalfont \text{SP}}

\newcommand{\Ad}{\normalfont \text{Ad}}
\newcommand{\zcl}{\normalfont \text{zcl}}

\newcommand{\G}{\mathcal{G}}

\newcommand{\e}{\epsilon}

\newcommand{\de}{\delta}
\newcommand{\De}{\Delta}
\newcommand{\la}{\lambda}

\newcommand{\La}{\Lambda}

\newcommand{\Ga}{\Gamma}
\newcommand{\ga}{\gamma}
\newcommand{\Lat}{\mathcal{L}}

\newcommand{\cA}{\mathcal{A}}
\newcommand{\omin}{\mathscr{S}}
\newcommand{\pbomin}{\mathscr{S}_{\mathrm{ pbd}}} 

\newcommand{\V}{\mathcal{V}}
\newcommand{\St}{\mathscr{S}}

\newcommand{\Span}{\normalfont \text{Span}}
\newcommand{\diag}{\normalfont \text{diag}}

\usepackage{cleveref}
\newcommand{\tr}{\normalfont \text{tr}}
\newcommand{\ad}{\normalfont \text{ad}}
\newcommand{\abs}[1]{\lvert #1 \rvert}

\DeclarePairedDelimiter\norm{\lVert}{\rVert}%

\newcommand{\F}{\mathscr F}
\newcommand{\cf}{\underline{c}}

\newtheorem{theorem}{Theorem}[section]

\newtheorem{lemma}[theorem]{Lemma}

\newtheorem{proposition}[theorem]{Proposition}

\newtheorem{corollary}[theorem]{Corollary}
\newtheorem*{example*}{Example}

\theoremstyle{definition}

\newtheorem{definition}[theorem]{Definition}

\theoremstyle{remark}

\newtheorem{remark}[theorem]{Remark}

\numberwithin{equation}{section}

\usepackage{setspace}
\usepackage{tikz-cd}
\usepackage[
url=false,
doi=false,
isbn=false,
style=alphabetic,
sorting=nty,
maxnames=99,
maxalphanames=99
]{biblatex}
\renewbibmacro{in:}{}
\bibliography{ref1.bib}

\DeclareFieldFormat[article]{citetitle}{\mkbibemph{#1}}
\DeclareFieldFormat[article]{title}{\mkbibemph{#1}}
\DeclareFieldFormat[article]{journaltitle}{#1}

\usepackage{enumerate}

\pgfplotsset{compat=1.18}
\begin{document}
\date{}

\title[Equidistribution of o-minimal curves]{Equidistribution of polynomially bounded o-minimal curves in homogeneous spaces}
\author{Michael Bersudsky, Nimish A. Shah, and Hao Xing}
\address{The Ohio State University, Columbus, OH 43210, USA}
\email{bersudsky87@gmail.com, shah@math.osu.edu, and thinkhowxh@gmail.com}

\begin{abstract}
We extend Ratner's theorem on equidistribution of individual orbits of unipotent flows on finite volume homogeneous spaces of Lie groups to trajectories of non-contracting curves definable in a polynomially bounded o-minimal structure.

To be precise, let $\varphi:[0,\infty)\to \SL(n,\R)$ be a continuous curve whose coordinate functions are definable in a polynomially bounded o-minimal structure; for example, rational functions. Suppose that $\varphi$ is non-contracting; that is, for any linearly independent vectors $v_1,\ldots,v_k$ in $\R^n$, $\varphi(t)\cdot (v_1\wedge\cdots\wedge v_k)\not\to0$ as $t\to\infty$. Then, there exists a unique closed connected subgroup $H_\varphi\subseteq \SL(n,\R)$   of smallest dimension such that $H_\varphi$ is generated by unipotent one-parameter subgroups and such that $\varphi(t)H_\varphi\to g_0H_\varphi$ in $\SL(n,\R)/H_\varphi$ as $t\to\infty$, for some $g_0\in \SL(n,\R)$. 

Let $\G\subseteq \SL(n,\R)$ be a closed subgroup and let $\Gamma\subseteq \G$ be a lattice. Suppose that $\varphi([0,\infty))\subseteq \G$. Then $H_\varphi\subseteq \G$, and for any $x\in \G/\Gamma$, the trajectory $\{\varphi(t)x:t\in [0,T]\}$ is equidistributed with respect to the measure $g_0\mu_{Lx}$ as $T\to\infty$, where $L\subseteq \G$ is a closed subgroup such that $\overline{H_\varphi x}=Lx$ and $Lx$ admits a unique $L$-invariant probability measure, denoted by $\mu_{Lx}$. 

A crucial new ingredient is the proof that for any finite-dimensional representation $V$ of $\SL(n,\R)$, there exist $T_0>0$, $C>0$, and $\alpha>0$ such that for any $v\in V$, the map $t\mapsto \norm{\varphi(t)v}$ is $(C,\alpha)$-good on $[T_0,\infty)$.
\end{abstract}

\maketitle
\tableofcontents

\section{Introduction}
\newcommand{\X}{X}
Let $\G$ be a Lie group, let $\Ga\subseteq \G$ be a lattice and consider  $\X:=\G/\Gamma$. For a closed subgroup $H\subseteq \G$ and an $x\in X$, the orbit $Hx$ is called \textit{periodic} if the orbit $Hx$ is closed and supports a unique $H$-invariant probability measure. We say that a probability measure $\mu$ on $\X$ is  \textit{homogeneous} if $\mu$  is supported on a periodic orbit $Hx$ and is the $H$-invariant probability on that orbit.  
Marina Ratner, in proving Raghunathan's conjectures stated in \cite{Dani81min_sets_of_hor}, showed that any Borel probability measure $\mu$ which is invariant and ergodic under the action of a $\Ad_{\G}$-unipotent one-parameter subgroup is homogeneous, see \cite{Ratner91a}. 
Using this measure classification and the non-divergence property of unipotent orbits due to Margulis and Dani \cite{Margulis71,Dani_1986}, Ratner~\cite{Ratner91duke} proved the following equidistribution result for unipotent flows: Let $U:=\{u(t):t\in\R\}$ be a one-parameter $\Ad_{\G}$-unipotent subgroup of $\G$ and let $x\in X$. Then there exists a closed subgroup $H$ of $\G$ containing $U$ such that $Hx$ is a periodic orbit, and for any $f\in C_c(\X)$, $\lim_{T\to\infty}\frac{1}{T}\int_0^Tf(u(t)x)dt=\int_\X f\,d\mu_{Hx}$, where $\mu_{Hx}$ is the homogeneous probability measure on $Hx$. 

More generally, for a continuous curve $\varphi:[0,\infty)\to \G$, $x\in X$, and $T>0$, consider the probability measure $\mu_{T,\varphi,x}$ on $X$ defined by:  
\begin{equation}\label{eq:main definition of measures averaging on curve}
    \mu_{T,\varphi,x}(f)=\frac{1}{T}\int_0^Tf(\varphi(t)x)\,dt,~ \forall f\in C_c(\X).
\end{equation}
In \cite{Shah1994LimitDO}, it was shown that the equidistribution of unipotent flows due to Ratner generalizes to polynomial curves. Namely, if $\G\subseteq\SL(n,\R)$ and $\varphi:\R\to \G$ is a map whose coordinate functions are polynonmials, then the measures $\mu_{T,\varphi,x}$ converge to  $\varphi(0)\mu_{Hx}$ as $T\to\infty$, where  $H\subseteq \G$ is the smallest closed subgroup such that $\{\varphi(0)^{-1}\varphi(t):t\in \R\}\subseteq H$ and the orbit $Hx$ is periodic. 

More recently, Peterzil and Starchenko \cite{Peterzil2018OminimalFO} studied such questions in the following setting: Let $\G$ be a closed subgroup of the group of the upper-triangular unipotent matrices in $\SL(n,\R)$ and let $\Gamma\subseteq \G$ be a (cocompact) lattice. They showed that if $S\subseteq \G$ is definable in an o-minimal structure, then there exists a definable set $S'\subseteq \G$ whose image in $\G/\Ga$ is the closure of the image of $S$ in $\G/\Ga$. Moreover, they proved the following \cite[Theorem 1.6]{Peterzil2018OminimalFO}: Let $\varphi:[0,\infty)\to \G$ be a curve definable in a polynomially bounded o-minimal structure, and let $x\in \G/\Gamma$ such that $\{\varphi(t)x:t\geq 0\}$ is dense in $G/\Gamma$. Then, $\mu_{T,\varphi,x}$ converges to the $\G$-invariant probability on $\G/\Gamma$  as $T\to\infty$.

Our goal in this paper is to generalize the above equidistribution results for curves definable in polynomially bounded o-minimal structures taking values in linear Lie groups, under the assumption that an additional non-contraction condition is satisfied. 
\subsection{O-minimal curves} We now provide an overview of the notion of an o-minimal structure. In this paper,  o-minimal structures refer to o-minimal structures on the real field. We have provided a more detailed summary of the required definitions, results, and references in Appendix~\ref{sec:Basic-notions-in o-minim}. 

The notion of o-minimality was introduced in \cite{generalization_tarski_seind_dries} as a generalization of semi-algebraic geometry. The semi-algebraic sets in $\R^d$ are formed by taking finite unions and intersections of algebraic sets $\{x\in \R^d: P(x)=0\}$ and of sets of the form $\{x\in\R^d:P(x)>0\}$, where $P$ is a polynomial in $d$ variables. These semi-algebraic sets form the smallest o-minimal structure, and our results are already new in the semi-algebraic setup.

More generally, an o-minimal structure is composed of sets in $\R^d$ for each $d\in \N$ such that several axioms are satisfied. The collection of semi-algebraic sets satisfy those axioms and form the smallest o-minimal structure. In a general  o-minimal structure, the collection of sets are referred to as the {\em definable sets\/}, and a function is called a {\em definable function\/} if its graph is a definable subset.  The definable sets and functions share many desirable properties, see~\Cref{sec:Basic-notions-in o-minim}. Particularly, the collection of definable functions is both rich and flexible: for example, in any o-minimal structure, all rational functions are definable; these are functions of the form $f(t) = \frac{P(t)}{Q(t)}$, on the domain where $Q(t) \neq 0$, and where $P$ and $Q$ are polynomials.
 Moreover, if a definable function is injective, its inverse is also definable, and the composition of definable functions remains definable. 

An o-minimal structure is called \emph{polynomially bounded}, if for every definable function $f:\R \to \R$, there exists $r\in \R$ such that $f(x)=O(x^r)$, as $x\to\infty$. The semi-algebraic o-minimal structure is polynomially bounded. Larger polynomially bounded o-minimal structures include as definable functions the analytic functions restricted to compact sets, and the power functions $t\mapsto t^r$, for $r\in \R$; see \cite{generalization_tarski_seind_dries,Miller_Expansions_of_real_field}. 

We note the following useful fact.
\begin{proposition}\emph{\cite{miller_exp_hard_avoid}}\emph{:} If  $f:[0,\infty)\to\R$ is definable in a polynomially bounded o-minimal structure, then either $f$ is eventually constantly zero, or there exists $r\in \R$ such that $\lim_{t\to\infty}f(t)/t^r=c\neq0$. In the latter case, we define \begin{equation}
        \deg(f):=r.\label{eq:def of degree}
    \end{equation}
\end{proposition}

\begin{remark}
    Our equidistribution results will be stated only for polynomially bounded o-minimal structures. As noted in \cite{Peterzil2018OminimalFO}, this setup is optimal in the following sense. Notice that the measures on the circle defined by 
    $$\frac{1}{T}\int_0^T f(\log(t+1)+\Z)\,dt,~f\in C(\R/\Z),$$
    do not converge as $T\to\infty$.  By \cite{miller_exp_hard_avoid}, if an o-minimal structure is not polynomially bounded, then 
    $x\mapsto e^x$ is definable, and as a consequence, $x\mapsto\log(x)$ will also be definable. 
\end{remark}


We will denote by $\mathbf{G}$ an affine algebraic group defined over $\R$, and in what follows, we let $G:=\mathbf{G}(\R)$ be the real points of $\mathbf{G}$.
\begin{definition}
    Let $\R[G]$ be the coordinate ring of real polynomial functions on $G$. We say that a curve $\varphi:[0,\infty)\to G$  is \textit{definable} in an o-minimal structure, if the map $t\mapsto f(\varphi(t))$ on $[0,\infty)$ is definable, for all $f\in \R[G]$.
\end{definition}
We observe that if $V$ is a finite-dimensional rational representation, and  $\varphi:[0,\infty)\to G$ is definable in $\mathscr S$, then the curve $t\mapsto\varphi(t)\cdot v$ is definable in $\mathscr S$ for all $v\in V$. So, if $G$ is embedded in $\SL(n,\R)\subseteq \mathrm{M}(n,\R)$, and
\begin{equation}
	\varphi(t)=[\varphi_{i,j}(t)]_{1\le i,j\leq n},\ \forall t\in [0,\infty),
\end{equation}
then $\varphi$ is definable in $\omin$ if and only if each $\varphi_{i,j}:[0,\infty)\to\R$ is definable in $\omin$. 

\subsubsection{The non-contraction property } 
We now introduce the non-contraction property, which narrows down the polynomially bounded o-minimal curves to a class of curves whose orbits on rational linear representations of $G$ have the $(C,\alpha)$-good growth property introduced by Kleinbock and Margulis~\cite{Klein_Marg_flows_and_dioph}. 

\begin{definition}\label{def: non-contracting curves}
    We say that a definable curve $\varphi:[0,\infty)\to G$ is \textit{non-contracting} for $G$ if for all rational finite-dimensional representations $V$ of $G$ defined over $\R$, it holds that $\lim_{t\to\infty}\varphi(t)\cdot v\neq 0$ for all $v\in V\setminus \{0\}$.
\end{definition}

We note that for any definable curve $\varphi:[0,\infty)\to \GL(n,\R)$ and any $v\in \R^n$, $\lim_{t\to\infty} \varphi(t)\cdot v$ exists in $V\cup\{\infty\}$ by the monotonicity theorem~\ref{thm:monotonicity thm}.

\begin{remark}
Suppose that $G$ is a unipotent group and $\varphi:[0,\infty)\to G$ is a definable curve. Then $\varphi$ is non-contracting. This is so because, for any algebraic action of an algebraic unipotent group on a finite-dimensional vector space, every orbit of a nonzero vector is Zariski closed, and in particular, the closure of the orbit does not contain the origin. See \cite[Theorem~12.1]{Birks71}.
\end{remark}

We now give a practical criterion that allows one to verify the non-contraction property for curves in $\SL(n,\R)$, see also Proposition \ref
{prop:non contracting is intrinsic when radical is unipotent}. Consider the exterior representation of $\SL(n,\R)$ on $\bigwedge^k \R^n$ defined by $$g\cdot(v_1\wedge v_2\wedge\cdots\wedge v_k):=gv_1\wedge gv_2\wedge\cdots\wedge gv_k.$$

\begin{proposition}\label{reduction prop}

Suppose that $\varphi:[0,\infty)\to \SL(n,\R)$ is a continuous curve definable in an o-minimal structure. Then $\varphi$ is  non-contracting if and only if  for all $1\leq k\leq n-1$ and linearly independent vectors $v_1,..., v_k\in \R^n$, it holds that 
\begin{equation} \label{eq:non-contract-decoposible}
\lim_{t\to\infty}\varphi(t)\cdot (v_1\wedge \cdots \wedge v_k)\neq 0.
\end{equation}
\end{proposition}



\subsection{The main equidistribution result}
To state our main result, we need to introduce an important subgroup associated with our curve.  Let $\varphi:[0,\infty)\to G$ be a non-contracting, continuous curve definable in an o-minimal structure. In \Cref{thm:main result on the hull and correcting curve}, we will show that there exists a smallest closed connected subgroup $H_\varphi\subseteq  G$ \textit{generated by unipotent one-parameter subgroups} such that the image of $\varphi([0,\infty))$ in $G/H_\varphi$ is bounded.
The subgroup $H_\varphi$ will be called the \textit{hull} of $\varphi$.
We also establish the existence of a bounded definable curve $\be:[0,\infty)\to G$, which will be called a \textit{correcting curve}, such that $\be(t)\varphi(t)\in H_\varphi$ for all $t$.
Since $\be$ is  bounded and definable,  $\lim_{t \to \infty} \be(t) = \be_\infty \in G$ by the monotonicity theorem, see  \Cref{thm:monotonicity thm}.

Now suppose that $\varphi$ takes its values in a connected unimodular Lie subgroup $\G\subseteq G$ and the o-minimal structure is polynomially bounded. Then $H_\varphi\subseteq \G$,
see Proposition~\ref{prop:hull in lie group}, and hence $\beta$ takes its values in $\G$.  Suppose that $\Gamma$ is a lattice in $\G$. 
Then, by Ratner's theorem \cite{Ratner91a}, since $H_\varphi$ is generated by one-parameter unipotent subgroups, for any $x_0\in \G/\Ga$ there exists a closed connected subgroup $L\subseteq \G$ such that $Lx_0$ is a periodic orbit, and $\overline{H_\varphi x_0}=Lx_0$. 

\begin{theorem}\label{thm: main equidistribution theorem}
  Let  $\G\subseteq  G$ be a connected Lie subgroup and  let $\Ga\subseteq \G$ be a lattice. Let \emph{$\varphi:[0,\infty)\to G$} be a non-contracting, continuous curve definable in a polynomially bounded o-minimal structure such that $\varphi\left([0,\infty)\right)\subseteq \G$. Let $x_0\in \G/\Ga$. 
   Then, 
$$\lim_{T\to\infty}\mu_{T,\varphi,x_0}=\be_\infty^{-1}\mu_{L x_0},$$
  in the weak-$\ast$ topology, where $Lx_0=\overline{H_\varphi x_0}$ is periodic, and $\mu_{Lx_0}$ is the $L$-invariant probability measure on $Lx_0$. 
  
\end{theorem}
As mentioned above, Theorem \ref{thm: main equidistribution theorem} generalizes Shah's result on polynomial curves \cite{Shah1994LimitDO}, and generalizes the result of Peterzil and Starchenko on dense curves in nilmanifolds \cite{Peterzil2018OminimalFO}.
\begin{remark} \label{rem:contraction} The non-contraction assumption in Theorem \ref{thm: main equidistribution theorem} is essential. Indeed, if $\varphi$ is a contracting curve in $\SL(n,\R)$, then,  there exists $x_0 \in \SL(n,\R)/\SL(n,\Z)$ such that $\lim_{t \to \infty} \varphi(t)x_0 = \infty$. To see this, we note that if $\varphi$ is contracting, then by Proposition \ref{reduction prop}, there exist linearly independent vectors $v_1, \dots, v_k \in \R^n$ such that $\varphi(t) \cdot (v_1 \wedge \cdots \wedge v_k) \to 0$ as $t \to \infty$. Let $g_0 \in \SL(n,\R)$ satisfy $v_1 \wedge \cdots \wedge v_k = g_0 \cdot (e_1 \wedge \cdots \wedge e_k)$. Since the orbit $\SL(n,\Z) \cdot (e_1 \wedge \cdots \wedge e_k)$ is discrete and does not contain $0$, given a compact set $C \subseteq \SL(n,\R)$, it holds that  $\varphi(t)g_0 \notin C\SL(n,\Z)$ for all large $t$. \end{remark}

In the next result, proved in Section~\ref{sec:7}, we give an example of a definable curve in $\SL(n+1,\R)$ whose 
hull is $\SL(n+1,\R)$. As a consequence, every trajectory of such a curve in the space of unimodular lattice in $\Z^{n+1}$ gets equidistributed (Corollary~\ref{cor:curve}).

\begin{proposition}\label{prop:main example for equidistribution}
Let $f_0,...,f_n$, and $h_1,\ldots, h_n$ be continuous real valued functions on $[0,\infty)$ definable in a polynomially bounded o-minimal structure. For $t\geq 0$, define
{\small 
\begin{equation}\label{eq:main example for equidistributing curve}
\varphi(t):=\begin{bmatrix}
    f_0(t)&f_1(t)&\cdots&f_n(t)\\
       &h_1(t)^{-1}&&\\
       &&\ddots&\\
       &&&h_n(t)^{-1}
\end{bmatrix}
\end{equation}
}
Suppose that the functions $h_i$'s and $f_i$'s defining $\varphi$ satisfy the following conditions:
\begin{enumerate}
    \item $f_0=h_1\cdots h_n$, $\deg f_0=n$, and $\deg h_1\geq \cdots \geq \deg h_n>0$. \label{itm:ex1}
    \item $\deg(f_0+g)\geq n$ for any $g$ in the linear span of $f_1,\ldots, f_n$. \label{itm:ex2}
    \item For any $i\in \{1,\ldots,n\}$, $\deg(f_i+g)> n-i$  for any $g$ in the linear span of $f_{i+1},\ldots,f_n$. \label{itm:ex3}
\end{enumerate}  
Then the  hull of $\varphi$ is $H_\varphi=\SL(n+1,\R)$.
\end{proposition}

For example, let $ f_0(t)=t^n$. Suppose $f_1,\dots,f_n$ are rational functions such that $\deg f_1>\cdots>\deg f_n$, and $\deg f_i\neq n$ and $\deg f_i>n-i$ for all $1\leq i\leq n$. Let $h_i(t)=t^{r_i}$ for some $r_1\geq \cdots\geq r_n>0$ such that $\sum_{i=1}^n r_i=n$. Then, these functions satisfy the conditions of Proposition~\ref{prop:main example for equidistribution}.

Using Theorem~\ref{thm: main equidistribution theorem}, we obtain the following:
\begin{corollary}\label{cor:curve}
Let $\G$ be a Lie group and $\Gamma$ be a lattice in $\G$. Suppose that  $\rho:\SL(n+1,\R)\to \G$ is a continuous homomorphism. Let $\varphi:[0,\infty)\to \SL(n+1,\R)$ be a curve satisfying the conditions of Proposition~\ref{prop:main example for equidistribution}. Then, for any $x\in \G/\Gamma$, and any bounded continuous function $f$ on $\G/\Gamma$, 
$$\lim_{T\to\infty}\frac{1}{T}\int_0^Tf(\varphi(t)x_0)dt=\mu_{X}(f),$$ where $\mu_X$ is the homogeneous probability measure on the homogeneous space $\overline{\rho(\SL(n+1,\R))x}$.
\end{corollary}

\subsection{Growth of o-minimal functions}\label{sec:intro c alpha good}

We now describe the $(C,\al)$-good property, introduced by Kleibock and Margulis~\cite{Klein_Marg_flows_and_dioph}, concerning the growth of certain families of functions definable in polynomially bounded o-minimal structures. This property is well known for polynomials~\cite[4.1~Lemma]{Dani1993LimitDO},  but is new in the o-minimal setting. It is crucial for the techniques used in this paper.  

\begin{definition}
     Let $C$ and $\al$ be positive constants. We say that a function $f:[0,\infty)\to \R$ is $(C,\alpha)$-\textit{good} on $[T_0,\infty)\subseteq [0,\infty)$, if for every $\e>0$, and for every bounded interval $I\subseteq[T_0,\infty)$, it holds that:
\begin{equation} \label{eq:C-alpha-good}
|\{t\in I : |f(t)|\leq \e\}|\leq C \left(\frac{\e}{\|f\|_I}\right)^\al\cdot |I|.
\end{equation}
\end{definition}

For the following statement, we say that a subset $\F$ of a finite-dimensional vector space $\V$ over $\R$ is a \textit{definable cone}, if $\F$ is a cone in the vector space, and $\F$ is a definable set when identifying $\V\cong \R^n$ by choosing a basis.


\begin{theorem}\label{thm:C_alpha_goodness}
\emph{($(C,\alpha)$-good property)}. Let $\V$ be a finite-dimensional real vector space spanned by functions $f:[0,\infty)\to\R$, each definable in the same polynomially bounded o-minimal structure. Suppose that $\F\subseteq \V$ is a closed definable cone such that for all $f\in\F\setminus \{0\}$ it holds that $$\lim_{t\to\infty}f(t)\neq0.$$Then, there exist $C>0$, $\alpha>0$ and $T_0\geq1$ such that for any $f\in\F\setminus \{0\}$, {\em $f$ is $(C,\alpha)$-good on $[T_0,\infty)$}.
\end{theorem}

We obtain the following statement as a direct corollary of Theorem \ref{thm:C_alpha_goodness}.

\begin{proposition}\label{prop:c alpha good ppty of curves and action in reps} 
Let $\varphi:[0,\infty)\to G$ be a non-contracting curve definable in a polynomially bounded o-minimal structure. Suppose that $V$ is a finite-dimensional rational representation of $G$. Fix a norm on $\|\cdot\|$ on $V$.
Then, there exist $T_0>0$, $C>0$, and $\alpha>0$  such that for all $v\in V\setminus \{0\}$ the function $\Theta_v(t):=\|\varphi(t)\cdot v\|$ is $(C,\al)$-good in $[T_0,\infty)$.  
\end{proposition}
Proposition \ref{prop:c alpha good ppty of curves and action in reps} is obtained from Theorem \ref{thm:C_alpha_goodness} as follows. By choosing a basis for the $G$-representation $V$, we identify  $V\cong \R^m$. Let $\psi:[0,\infty)\to \GL(m,\R)$ represent the image of $\varphi$ in the representation $V$.  Let $\|\cdot\|$ be the Euclidean norm on $\R^m$. Then, the vector space of functions $$\V:=\Span_\R\{t\mapsto\|\psi(t)\vv\|^2, \ t\geq 0:\vv \in \R^m\}$$is finite dimensional, and $\F:=\{t\mapsto\|\psi(t)\vv\|^2, \ t\geq 0:\vv \in \R^m\}$ is a closed definable cone. 

\subsection{Proof ideas and structure of the paper} 

Our main theorem will be proved by the following strategy employed in many works on equidistribution in homogeneous spaces, for example \cite{Dani1993LimitDO,Mozes1995OnTS,Shah1994LimitDO}:
\begin{enumerate}
     \item\label{enu:workflow non escape of mass} Proving that there is no escape of mass (Section~\ref{sec:non escape of mass}) -- that is, any limiting measure $\mu$ of the probability measures $\mu_{T,\varphi,x}$ is a probability measure. 
    \item\label{enu:workflow invariance} Showing that any limiting measure $\mu$ of the measures $\mu_{T,\varphi,x}$ is invariant under a group generated by nontrivial unipotent one-parameter subgroups (Section~\ref{sec:unip invariance}). 
    \item\label{enu:workflow ergodic decomposition} Studying the ergodic decomposition of a limiting measure $\mu$ (Section \ref{sec:linearization}). By Ratner's theorem~\cite{Ratner91a}, every ergodic component appearing in the ergodic decomposition of $\mu$ is homogeneous. The theorem is proved once it is verified that $\mu$ is equal to exactly one uniquely determined ergodic component. For this step, we will use the linearization technique of \cite{Dani1993LimitDO} to prove avoidance of singular sets that contain ergodic components of various types. 
\end{enumerate}
 
To establish unipotent invariance, we will use an idea similar to the one used for polynomial curves \cite{Shah1994LimitDO}. This idea generalizes to the o-minimal setting via the notion of the Peterzil-Steinhorn group (see Section \ref{sec:unip invariance}), which was also used in \cite{Peterzil2018OminimalFO}.

The $(C,\alpha)$-good property will be used to prove the non-escape of mass, which is avoidance of the point at infinity, and to apply the linearization technique.

\subsection{Notation and conventions} \label{subsec:convention}
We will denote the Zariski closure of a subset $A$ of $\R^N$ by $\zcl(A)$. For a linear Lie group $H$, we denote by $H_u$ the subgroup of $H$ generated by all the one-parameter unipotent subgroups contained in $H$. We note that if $H\subseteq \GL(n,\R)$, then $H_u$ equals the connected component of the identity in $\zcl(H_u)$. 


\subsubsection*{Acknowledgement.}
We thank Chris Miller for helpful discussions. We thank anonymous referees for their careful reading and valuable suggestions which  helped us improve the readability of this article.
 
\section{On the non-contracting property and the hull\label{sec:consequences and properties of the non-contracting property}}
This section is divided into two main parts. In the first part, we prove Proposition \ref{reduction prop}, which gives a convenient criterion for verifying the non-contraction property. The second part of the section concerns the existence of the hull and the correcting curve, and discusses their behavior under homomorphisms.

\subsection{The criterion for the non-contracting property}\label{sec:crit for the non-contract}
Our main goal here is to prove Proposition \ref{reduction prop}, namely to reduce the verification of the non-contracting property for any vector in any representation to decomposable vectors in the exterior of the standard representation of $\SL(n,\R)$. Our proof follows from the observations in \cite[Section 2]{shah2023equidistribution}, which uses Kempf's numerical criteria on geometric invariant theory \cite{Kem78}. 

We recall the following preliminaries. Let ${S}$ denote the full diagonal subgroup of $\SL(n,\R)$. The group of the unipotent upper triangular matrices in $\SL(n,\R)$, denoted by $N$, is a maximal unipotent subgroup of $\SL(n,\R)$ normalized by ${S}$.
Let $X^\ast({S})$ denote the group of algebraic homomorphisms from ${S}$ to $\R^\ast$ defined over $\R$.
Then $X^\ast({S})$ is a free abelian group with $(n-1)$ generators, and we treat it as an additive group. For each $i\in\{1,\ldots,n-1\}$, let $\alpha_i\in X^\ast({S})$ be defined by $\alpha_i(\diag(t_1,\ldots,t_n))=t_i/t_{i+1}$.
Then $\Delta=\{\alpha_i:1\leq i\leq n-1\}$ is the set of simple roots on ${S}$ corresponding to the choice of the maximal unipotent subgroup $N$. For each $1\leq i\leq n-1$, let $\mu_i\in X^\ast({S})$ be defined by $\mu_i(\diag(t_1,\ldots,t_n))=t_1\cdots t_i$.
Then $\{\mu_i:1\leq i\leq n-1\}$ is the set of fundamental characters (weights) of ${S}$ with respect to our choice of simple roots. For each $1\leq i\leq n-1$, let $W_i$ denote the $i$-th exterior of the standard representation of $\SL(n,\R)$ on $\R^n$, called a \textit{fundamental representation} of $\SL(n,\R)$. So $W_i=\wedge^i \R^n$, and for any $v_1\wedge\cdots\wedge v_i\in W_i$ and $g\in \SL(n,\R)$, $g\cdot(v_1\wedge \cdots \wedge v_i)=gv_1\wedge \cdots \wedge gv_i$. Let $w_i=e_1\wedge \cdots \wedge e_i$. Then $w_i$ is fixed by $N$, and ${S}$ acts on the line $\R w_i$ via the fundamental character $\mu_i$.

A \textit{dominant} integral character (weight) is a non-negative integral combination of fundamental characters. Suppose $\chi=m_1\mu_1+\cdots+m_{n-1}\mu_{n-1}$, where $m_i\in\Z_{\geq 0}$. For each $i$, let $W_i^{\otimes m_i}$ denote the tensor product of $m_i$-copies of $W_i$, and $w_i^{m_i}:=w_i\otimes \cdots\otimes w_i\in W_i^{\otimes m_i}$. Consider the tensor product representation of $\SL(n,\R)$ on $W_\chi=W_1^{\otimes m_1}\otimes \cdots \otimes W_{n-1}^{\otimes m_{n-1}}$, and let $w_\chi=w_1^{m_1}\otimes \cdots \otimes w_{n-1}^{m_{n-1}}$. Then $w_{\chi}$ is fixed by $N$, and ${S}$ acts on the line $\R w_{\chi}$ via the character $\chi$.  We endow $W_\chi$ with a $\SO(n,\R)$-invariant norm.

\begin{proof}[Proof of Proposition \ref{reduction prop}] Let $V$ be a finite dimensional representation of $G:=\SL(n,\R)$ endowed with some norm. Suppose $v\in V$ is a  nonzero vector such that $\varphi(t)\cdot v\to 0$ as $t\to \infty$. Then $\overline{G\cdot v}\ni 0$, so $v$ is a $G$-\textit{unstable} vector in $V$. So, by \cite[Remark~2.3]{shah2023equidistribution}, there exist $g_0\in G$, a dominant integral character $\chi\in X^\ast({S})$, and constants $\beta>0$ and $C>0$ such that for any $g\in G$, 
\begin{equation} \label{eq:reduction}
    \norm{gg_0\cdot w_\chi}\leq C\norm{g\cdot v}^\beta.
\end{equation}

Therefore, $\lim_{t\to \infty} \varphi(t)g_0\cdot w_\chi=0$. Let $A={S}^0$, the component of the identity in ${S}$. Let $K=\SO(n,\R)$. Then by Iwasawa decomposition, $G=KAN$. For each $t$, we express $\varphi(t)g_0=k_ta_tn_t$, where $k_t\in K$, $a_t\in A$ and $n_t\in N$. Since the norm on $W_\chi$ is $K$-invariant and $N$ fixes $w_\chi$,
\[
\norm{\varphi(t)g_0\cdot w_\chi}=\norm{a_t \cdot w_\chi}=\chi(a_t)\norm{w_\chi}.
\]
By \eqref{eq:reduction}, $\lim_{t\to\infty} \varphi(t)g_0\cdot w_\chi=0$. Therefore, $\lim_{t\to\infty} \chi(a_t)=0$. Now 
\[
\chi(a_t)=\mu_1(a_t)^{m_1}\cdots \mu_{n-1}(a_t)^{m_{n-1}}.
\] 
Since $m_i\geq 0$, after passing to a subsequence, we can pick $i\in\{1,\ldots,n-1\}$ and a sequence $t_l\to\infty$ such that $\lim_{l\to\infty} \mu_i(a_{t_l})=0$. Therefore, since $N$ fixes $w_i$, and the norm is $K$-invariant,
\[
\norm{\varphi(t_l)g_0\cdot w_i}=\norm{{a_{t_l}\cdot w_i}}=\mu_i(a_{t_l})\norm{w_i}\to 0\text{, as $l\to\infty$.}
\]
Since the coordinate functions of $t\to\varphi(t)g_0\cdot w_i$ are definable in an o-minimal structure, we get $\lim_{t\to\infty}\varphi(t)g_0\cdot w_i=0$. Let $v_j=g_0e_j$ for each $j$. Since $w_i=e_1\wedge \cdots \wedge e_i$, $\varphi(t)\cdot (g_0\cdot w_i)=\varphi(t)\cdot (v_1\wedge \cdots \wedge v_i)\to 0$ as $t\to\infty$. This contradicts the assumption that the curve is non-contracting.
\end{proof}

We record the following statement. It will be used in the proof of Lemma~\ref{lem:hull is semisimple if curve expands}, which is in turn needed for proving Proposition \ref{prop:main example for equidistribution}. 

\begin{lemma}\label{lem:expanding implies orbit closed}
    Let $\varphi:[0,\infty)\to \SL(n,\R)$ is a map such that \begin{equation}\label{eq:div}
     \lim_{t\to\infty}\varphi(t)\cdot (v_1\wedge \dots \wedge v_k)=\infty,
 \end{equation}for all $1\leq k\leq n-1$ and linearly independent vectors $v_1,..., v_k\in \R^n$. Then for any finite-dimensional rational representation $V$ and any  $v\in V$, if $\{\varphi(t)\cdot v: t\geq 0\}$ is bounded in $V$, then $\SL(n,\R)\cdot v$ is Zariski closed.
\end{lemma}
\begin{proof}
    Suppose that $V$ is a finite-dimensional rational representation of $G=\SL(n,\R)$ and let $v\in V$ such that $G\cdot  v$ is not Zariski closed. So there exists a finite-dimensional rational representation $W$ of $G$ and a $G$-equivariant polynomial map $P:V\to W$ such that $P(v)\neq 0$ and $\overline{G\cdot P(v)}\ni 0$, see \Cref{lem:kempf lemma adjusted}. Therefore, 
    by \cite[Remark~2.3]{shah2023equidistribution}, there exist a $g_0\in G$, a nontrivial dominant integral character $\chi\in X^\ast({S})$, and constants $\beta>0$ and $C>0$ such that for any $g\in G$,
    \[
    \norm{gg_0\cdot w_\chi}\leq C\norm{g\cdot P(v)}^\beta=C\norm{P(g\cdot v)}^\beta,
    \]
    cf.~\cite[Corollary~2.5]{shah2023equidistribution}. 
    Since $P$ is continuous, given any $R>0$, there exists $C_R>0$ such that for any $g\in G$,
\begin{align} \label{eq:reduction in expanding case}
    \text{if $\|g\cdot v\|\leq R$, then } \norm{gg_0\cdot w_\chi}\leq C_R\norm{g\cdot v}^\beta.
\end{align}
Now suppose that $\varphi:[0,\infty)\to\SL(n,\R)$ is a map satisfying the diverging property \eqref{eq:div} and assume for contradiction that $\{\varphi(t)\cdot v:t\geq 0\}$ is bounded in $V$. Then, by \eqref{eq:reduction in expanding case}, we get that $\{\norm{\varphi(t)g_0\cdot w_\chi}:t>0\}$ is bounded. As in the  proof of Proposition \ref{reduction prop}, by Iwasawa decomposition, for all $t>0$, $\varphi(t)g_0=k_ta_tn_t$, where $k_t\in K$, $a_t\in A$ and $n_t\in N$, and
\[
\norm{\varphi(t)g_0\cdot w_\chi}=\norm{a_t\cdot w_\chi}=\chi(a_t)\norm{w_\chi}.
\]
We claim that $\lim_{t\to\infty}\chi(a_t)=\infty,$ which is a contradiction. Indeed, we have
\[
\chi(a_t)=\mu_1(a_t)^{m_1}\cdots \mu_{n-1}(a_t)^{m_{n-1}},
\]
for some $(m_1,\ldots,m_{n-1})\in\Z_{\geq 0}$ and $m_i>0$ for some $i$. Now
\[
\norm{\varphi(t)g_0\cdot w_i}=\norm{{a_{t}\cdot w_i}}=\mu_i(a_{t})\norm{w_i},
\]
where $w_i=e_1\wedge\cdots \wedge e_i$ and $1\leq i\leq n-1$. And by \eqref{eq:div}, we have $\lim_{t\to\infty}\norm{\varphi(t)g_0\cdot w_i}=\infty$. So, $\lim_{t\to\infty} \mu_i(a_t)\to\infty$.
\end{proof}

\subsection{On the hull and the correcting curve\label{sec:hull and correcting curve}} In what follows,  $G:=\mathbf{G}(\R)$ is the real points of an algebraic group $\mathbf{G}$ over $\R$. To define the hull and obtain its useful properties, we will need the following notion.
 \begin{definition}\label{def:observable gps}
   We say that a Zariski closed subgroup $H$ is \textit{observable} in $G$ if there is a finite-dimensional rational representation $V$ over $\R$ and  $v\in V$ such that $H=\{g\in G:g\cdot v=v\}$.
\end{definition}
The notion of observable groups was introduced \cite{Mostow_et_al_observable_groups}. More precisely, \cite{Mostow_et_al_observable_groups} defines observable groups as the algebraic subgroups of algebraic groups for which every finite-dimensional rational representation extends to a finite-dimensional rational representation of the ambient group. The above definition turns out to be equivalent to the extension property \cite[Theorem 8]{Mostow_et_al_observable_groups} (see \cite[Theorem 9]{Observable_groups_over_fields} for this statement for non-algebraically closed fields). The following known result will be useful, see e.g. \cite[Corollary 2.8]{Observable_group_grosshans}.
\begin{lemma}\label{lem:generated by uni implies observ} 
     Let $H$ be an algebraic subgroup of an algebraic group $G$ over $\R$. If the radical of $H$ is unipotent, then $H$ is an observable subgroup of $G$. 
\end{lemma}

 \begin{remark} \label{rem:genbyuni-observable}
 Suppose that $L$ is a closed subgroup of $\SL(n,\R)$ generated by unipotent one-parameter subgroups. Then $L=\zcl(L)^0$ and the radical of $\zcl(L)$ is unipotent (see \cite[Lemma 2.9]{Shah1991}). Hence, by \Cref{lem:generated by uni implies observ}, $\zcl(L)$ is observable in $\SL(n,\R)$. 
\end{remark}

In what follows, we say that a curve $\varphi([0,\infty))\subseteq G$ is \textit{bounded} modulo a closed subgroup $H\subseteq G$ if there exists a compact subset $K\subseteq G$ such that $\varphi([0,\infty))\subseteq KH$. 

We fix an o-minimal structure $\omin$ for the rest of the following two subsections. We stress that the assumption of polynomial boundedness on $\omin$ is unnecessary here. 

Here is our main theorem in this section.

\begin{theorem}\label{thm:main result on the hull and correcting curve}
  Let $\varphi:[0,\infty)\to G$ be a continuous,  non-contracting curve definable in  $\omin$.
   Then, there exists a unique closed connected \emph{(}in the Hausdorff topology\emph{)} subgroup  $H_\varphi\subseteq G$  of the smallest dimension such that $H_\varphi$ is generated by one-parameter unipotent subgroup and such that $\varphi$ is bounded modulo $H_\varphi$. Moreover, the following properties hold: 
   
 \begin{enumerate}
  \item\label{enu:hull is minimal-Lie}The subgroup $H_\varphi$ is minimal in the following sense: if $H'$ is an observable group such that $\varphi$ is bounded modulo $H'$, then $H_\varphi\subseteq H'$. Moreover, if $V$ is a rational representation, and $v\in V$ is such that $\{\varphi(t)\cdot v:t\geq 0\}$ is bounded in $V$, then $\varphi$ is bounded modulo the isotropy group of $v$ and $H_\varphi\cdot v=\{v\}$. 
  \item \label{enu:curve contained in hull-Lie} There is a bounded, continuous  curve $\be: [0,\infty)\to G$ definable in $\mathscr S$ such that $\be(t)\varphi(t)\in H_\varphi$ for all $t\geq 0$.
  \item \label{enu:fixing property after correction-Lie} If $V$ is a rational finite-dimensional representation of $G$ over $\R$,  then for any $v\in V$, either $H_\varphi\cdot v=\{v\}$ or $\lim_{t\to\infty}\be(t)\varphi(t)\cdot v=\infty$.
\end{enumerate}  
\end{theorem}

In view of the above statement, we can now define the hull.

\begin{definition}\label{def:hull and correcting curve}
    Let $\varphi:[0,\infty)\to G$ be a continuous non-contracting curve definable in  $\mathscr S$. We define the \textit{hull} $H_\varphi$ to be the smallest connected closed subgroup generated by unipotent elements such that $\varphi$ is bounded modulo $H_\varphi$, and we call a curve $\be:[0,\infty)\to G$ a \textit{correcting curve} if $\be(t)\varphi(t)\in H_\varphi$ for all $t\geq0$.
\end{definition}

\begin{remark}\label{rem:hull is identity compo of zcl}
In view of \Cref{rem:genbyuni-observable}, $H_\varphi=\zcl(H_\varphi)^0$ and $\zcl(H_\varphi)$ is the observable subgroup of smallest dimension modulo which $\varphi$ is bounded.
\end{remark} 

\begin{remark} \label{rem:hull of corrected curve is hull}
   We note that if $\varphi$ is as in the above theorem and $\be$ is a definable correcting curve, then the hull of $\psi(t):=\be(t)\varphi(t)$ is the same as the hull of $\varphi$, and the correcting curve of $\psi$ can be chosen to be the constant identity matrix. This follows from the uniqueness of the hull.

\end{remark}

\begin{remark}
    Suppose that $\varphi:[0,\infty)\to \SL(n,\R)$ is a definable curve whose entry functions $\varphi_{i,j}$ generate an algebra which consists of functions that are either constant or diverging to $\infty$, e.g., when $t\mapsto \varphi_{i,j}(t)$ are polynomials for all $i,j$. Then,  by Theorem \ref{thm:main result on the hull and correcting curve},  $\varphi(0)^{-1}\varphi([0,\infty))\subseteq H_\varphi$. A correcting curve is given by the constant curve $\be(t)= \varphi(0)^{-1}$ for all $t\geq 0$.
\end{remark}

Before proving Theorem \ref{thm:main result on the hull and correcting curve}, we establish that the hull is intrinsically defined.
\begin{lemma}\label{lem:image of hull under homo}
 Let $G_1$ and $G_2$ be the real points of algebraic groups over $\R$, and let $f:G_1\to G_2$ be a rational homomorphism. Suppose that $\varphi:[0,\infty)\to G_1$ is a non-contracting curve definable in $\omin$. Then $f(H_\varphi)=H_{f\circ\varphi}$.
\end{lemma}

\begin{proof}
   We start by showing that $H_{f\circ\varphi}\subseteq f(H_{\varphi})$. Let $\be:[0,\infty)\to G_1$ be a correcting curve of $\varphi$. Since $H_\varphi$ is generated by unipotents, we get that $f(H_\varphi)$ is generated by unipotents. Moreover, the curve $f\circ\varphi$ is bounded modulo $f(H_\varphi)$. In fact, $t\mapsto f(\be(t))$ is bounded, and for all $t\geq 0$: $$f(\beta(t)\varphi(t))=f(\be(t))f(\varphi(t))\in f(H_\varphi).$$ Thus, by minimality of the hull, we conclude that $H_{f\circ\varphi}\subseteq f(H_{\varphi})$. 
   
   We proceed to show the other inclusion. Since $\zcl(H_{f\circ\varphi})$ is observable, we choose a rational representation $\rho:G_2\to\GL(V)$ and a $v\in V$ such that $$\zcl(H_{f\circ\varphi})=\{g\in G_2:\rho(g)v=v\}.$$Since $f\circ \varphi$ is bounded modulo $H_{f\circ\varphi}$,  the trajectory, $$\rho((f\circ\varphi)(t))v=(\rho\circ f)(\varphi(t))v,~t\geq 0,$$is bounded in $V$. Now, $\eta:=\rho\circ f:G_1\to \GL(V)$ is a representation of $G_1$ such that $t\mapsto\eta(\varphi(t))v,~t\geq 0$ is bounded. By Theorem \ref{thm:main result on the hull and correcting curve}\eqref{enu:hull is minimal-Lie}, we obtain that $\rho(f(H_\varphi))v=\eta(H_\varphi)v=\{v\}$. That is, $f(H_\varphi)$ is contained in the isotropy group of $v$ which is equal to $\zcl(H_{f\circ\varphi})$. Since $f(H_\varphi)$ is generated by unipotent one-parameter subgroups, 
   \[
   f(H_\varphi)\subseteq \zcl(H_{f\circ\varphi})^0=H_{f\circ\varphi}.
   \]
\end{proof}

\subsection{Proving Theorem \ref{thm:main result on the hull and correcting curve}}
We begin with a ``curve correcting'' lemma.
\begin{lemma}\label{lem:curve bounded modification}Consider an algebraic action of $G$ on the $\R$ points of an affine or projective variety $\ver$ over $\R$.  Let $\varphi:[0,\infty)\to G$ be a continuous curve definable in  $\omin$. Suppose that $x\in \ver$ and $g_0\in G$ are such that
$$\lim_{t\to\infty}\varphi(t)\cdot x =g_0\cdot x.$$

Then, there exists a continuous $\omin$-definable curve $\delta:[0,\infty)\to G$ such that 
\[
\lim_{t\to\infty}\de(t)=e \text{ and }\delta(t)\varphi(t)\cdot x= g_0\cdot x,\ \forall t\geq 0.
\]
\end{lemma}
\begin{proof}
  Identify $G$  as a closed subgroup of $\SL(n,\R)$.  Consider the following definable set:\begin{align*}
    \mathcal{D}&:=\{(t,g):t>0, \ g\in G, \|g-I_n\|< 1\text{ and }g_0^{-1}\varphi(t)\cdot x=g\cdot x\}.
\end{align*}
Here $\|\cdot\|$ is a definable norm (e.g., the sum of squares norm). Since the orbit map is an open map (maps open sets in $G$ to open sets in $G\cdot x$ with respect to the induced topology from $\ver$, see \cite[Corollary 3.9]{Platonov_Rapinchuk_Rapinchuk_2023}), and since 
$$\lim_{t\to\infty}g_0^{-1}\varphi(t)\cdot x=x,$$ 
we get that for all $t$ large enough, say $t\geq T_0$, there exists $g\in G$ such that $(t,g)\in\mathcal{D}$. Using the choice function theorem, see Theorem \ref{thm:choice function}, there is a definable curve $\tilde{\delta}(t):[T_0,\infty)\to G$ such that $(t,\tilde \de(t))\in\mathcal D$ for all $t\geq T_0$. In particular, 
$$\tilde{\delta}(t)\cdot  x=g_0^{-1}\varphi(t)\cdot x, \,\forall t\geq T_0,$$ 
and $\tilde\de$ is bounded. Since the curve $\tilde \de (t)$ is bounded, by the monotonicity theorem, see \Cref{thm:monotonicity thm}, we have that $\lim_{t\to\infty}\tilde\de(t)=\delta_0$. Note that $\delta_0\cdot x=x$. Let $\delta(t):=g_0(\tilde{\delta}(t)\de_0^{-1})^{-1}g_0^{-1}$ for all $t\geq T_0$. Then, 
$$\de(t)\varphi(t)\cdot x=g_0\cdot x,\forall t\geq T_0.$$ 
By the monotonicity theorem, $\de(t)$ is continuous in $[T'_0,\infty)$, for some $T'_0\geq T_0$. We now extend $\de$ to the interval $[0,T'_0]$ by putting 
$$\de(t):=\de(T'_0)\varphi(T'_0)\varphi(t)^{-1},\,\forall t\leq T'_0.$$

Then $\de(t)\varphi(t)\cdot x=g_0\cdot x,\forall t\geq 0$. Since $\varphi$ is continuous and definable, $\de:[0,\infty)\to G$ is continuous, definable and bounded.
\end{proof}

\begin{lemma}\label{lem:correction of bounded trajectory in reps}
  Suppose that $\varphi:[0,\infty)\to G$ is a non-contracting, continuous curve, definable in $\mathscr S$. Suppose that $V$ is a finite-dimensional rational representation of $G$ and $v\in V\setminus \{0\}$ is such that $\varphi(t)\cdot v$ is bounded. Then, there exists $g_0\in G$ such that 
  $$\lim_{t\to\infty}\varphi(t)\cdot v=g_0\cdot v.$$ 
  In particular, there exists a bounded, continuous curve $\be:[0,\infty)\to G$ definable in $\mathscr S$, such that $\be(t)\varphi(t)\cdot v=v$ for all $t\geq 0$.
\end{lemma}
\begin{proof}
   Since $\varphi:[0,\infty)\to G$ is definable, and since $V$ is a rational representation,  $\varphi(t)\cdot v$ is a bounded definable curve. By the monotonicity theorem (see Theorem \ref{thm:monotonicity thm}), $\lim_{t\to\infty}\varphi(t)\cdot v=v'\in V$. We claim that $v'\in G\cdot v$. Denote $S:=\zcl(G\cdot v)\setminus G\cdot v$, and  suppose for contradiction that $v'\in S$. Then, by Lemma \ref{lem:kempf lemma adjusted}, which is a slight modification of \cite[Lemma 1.1(b)]{Kem78}, there exists a rational representation $W$ of $G$ and  $G$-equivariant polynomial map $ P:V\to W$ such that $P(v)\neq 0$ and $P(S)=0$. In particular,  $P(v')=0$. Then 
   $$\lim_{t\to\infty}\varphi(t)\cdot P(v)=P(\lim_{t\to\infty}\varphi(t)\cdot v)=P(v')=0,$$
   which is a contradiction, as $\varphi$ is non-contracting.
   Thus, we conclude that there exists $g_0\in G$ such that $$\lim_{t\to\infty}\varphi(t)\cdot v=g_0\cdot v.$$

The existence of the bounded curve $\be$ is obtained by Lemma \ref{lem:curve bounded modification}.
\end{proof}

\begin{lemma}\label{lem:intersection of observable groups for bounding curve is again such}
  Suppose that $\varphi:[0,\infty)\to G$ is a non-contracting, continuous curve, definable in $\omin$. Let $H_1$ and $H_2$ be observable groups of $G$ such that the image of $\varphi$ is bounded in $G/H_i$ for $i\in\{1,2\}$. Then, the image of $\varphi$ is bounded in $G/(H_1\cap H_2)$.
\end{lemma}
\begin{proof}
    For $i\in\{1,2\}$, let $V_i$ be a finite-dimensional rational representation such that $H_i=\{g\in G:g\cdot v_i=v_i\}$ where  $v_i\in V_i$. Consider the direct-sum representation $V_1\oplus V_2$, and notice that $$H_1\cap H_2=\{g\in G:g\cdot(v_1,v_2)=(v_1,v_2)\}.$$Since the image of $\varphi$ is bounded in $G/H_i$,  $\varphi(t)\cdot v_i$ is bounded for all $i\in\{1,2\}$. In particular, $\varphi(t)\cdot (v_1,v_2)$ is bounded in $V_1\oplus V_2$. Then, by Lemma \ref{lem:correction of bounded trajectory in reps}, there is a bounded continuous curve $\de:[0,\infty)\to G$ such that $\de(t)\varphi(t)\cdot (v_1,v_2)=(v_1,v_2)$ for all $t\geq 0$. Namely, $\de(t)\varphi(t)\in H_1\cap H_2$ for all $t\geq 0$, and therefore the image of $\varphi$ is bounded in $G/(H_1\cap H_2)$. 
\end{proof}
\begin{corollary}\label{cor:existance of minimal observable}
  Suppose that $\varphi:[0,\infty)\to G$ is a non-contracting, continuous curve definable in $\omin$. Then there exists a unique observable subgroup $\tilde H_{\varphi}\subseteq G$ of minimal dimension such  that $\varphi$ is bounded in $G/ \tilde H_{\varphi}$. Moreover, if $H$ is an observable subgroup such that $\varphi$ is bounded in $G/H$, then $ \tilde H_\varphi\subseteq H$.
\end{corollary}
\begin{proof}
  Let $H_0$ be a Zariski connected observable group of the smallest dimension such that $\varphi$ is bounded in $G/H_0$. If $H$ is any connected observable group such that $\varphi$ is bounded in $G/H$, then by Lemma \ref{lem:intersection of observable groups for bounding curve is again such}, we get that $\varphi$ is bounded in $G/(H_0\cap H)$. If $H_0\not\subseteq H$, then $H\cap H_0$ is a proper subgroup of dimension strictly smaller than $\dim H_0$, which leads to a contradiction.
\end{proof}
\begin{corollary}\label{cor:proving the fixing property}
    Suppose that $\varphi:[0,\infty)\to G$  is a non-contracting, continuous curve, definable in $\omin$. Let $\tilde H_\varphi\subseteq G$ be the minimal observable group such that $\varphi$ is bounded modulo $\tilde H_\varphi$. Denote $H_\varphi:=\tilde H_\varphi^0$, the identity component of $\tilde H_\varphi$. 
    
    Then, there exists a bounded, continuous definable curve $\be:[T_0,\infty)\to G$ such that $\be(t)\varphi(t)\in H_\varphi$ for all $t\geq 0$. Moreover, if $V$ is a finite-dimensional rational representation of $G$ and $v\in V$ is such that $\varphi(t)\cdot v$ is bounded, then $\be(t)\varphi(t)\cdot v=v$ for all $t\geq T_0$. 
\end{corollary}
\begin{proof}
    By Lemma \ref{lem:correction of bounded trajectory in reps}, we obtain a continuous definable curve $\tilde \be:[0,\infty)\to G$ such that $\tilde\be(t)\varphi(t)\in \tilde H_\varphi$ for all $t\geq 0$. Since $\varphi$ is continuous, it is contained in a connected component of $\tilde H_\varphi$. By choosing a suitable $h_0\in \tilde H_\varphi$ we get that for $\be(t):=h_0\tilde \be(t)$ it holds $\be(t)\varphi(t)\in \tilde H^0_\varphi=H_\varphi$, for all $t\geq 0$. Now, if $V$ is a finite-dimensional rational representation such that $\varphi(t)\cdot v$ is bounded, then by Lemma \ref{lem:correction of bounded trajectory in reps}, we conclude that $\varphi$ is bounded modulo the observable group $H:=\{g\in G:g\cdot v=v\}$. Hence, by minimality of $\tilde H_\varphi$, we get that $\be(t)\varphi(t)\in H_\varphi\subseteq H$.
\end{proof}

The last property in  Theorem \ref{thm:main result on the hull and correcting curve} remaining to prove is that  the identity component of the   observable group of smallest dimension is generated by one-parameter unipotent subgroups. This is given by the following.

\begin{lemma} \label{lemma:genbyuni}
   Let $\varphi:[0,\infty)\to G$ be a $\omin$-definable, non-contracting, continuous curve, and let $\tilde H_\varphi$ be the smallest observable subgroup such that $\varphi$ is bounded modulo $\tilde H_\varphi$. Then $H_\varphi=\tilde H_\varphi^0$ is generated by one-parameter unipotent subgroups.
\end{lemma}

\begin{proof}
    Let $\psi(t):=\be(t)\varphi(t)$ for all $t\geq 0$, where $\be$ is a correcting curve of $\varphi$. Let $L\subseteq \tilde H_\varphi$ be the Zariski closure of the subgroup generated by all one-parameter unipotent subgroups of $\tilde H_\varphi$. Then, $L$ is observable in $G$, see \Cref{rem:genbyuni-observable}. We will now show that the image of $\varphi$ in $ \tilde H_\varphi/L$ is bounded, which implies that $ \tilde H_\varphi=L$, and proves the claim, see Remark~\ref{rem:hull is identity compo of zcl}.
    
    Since $L$ is a normal subgroup of $\tilde H_\varphi$, we get that $\tilde H_\varphi/L$ is an affine algebraic group. Moreover, since $L$ contains the unipotent radical of $ \tilde H_\varphi$, we get that  $\tilde H_\varphi/L$ is  reductive. Then, there is an almost direct product decomposition $\tilde H_\varphi/L=ZD$, where $Z$ is the center, which is a torus, and $D$ is the derived group, which is a normal semi-simple subgroup. Notice that each simple factor of $D$ must be compact since a non-compact simple group is generated by unipotent one-parameter subgroups. Namely, $D$ is compact. Let $q:\tilde H_\varphi\to  \tilde H_\varphi/L$ be the natural quotient map. To finish the proof, we show below that for every character $\chi: \tilde H_\varphi/L\to\R^\times$  it holds that $\chi\circ q(\psi(t))=1$ for all $t\geq 0$. This will complete the proof since then $q\circ\psi([0,\infty))\subseteq Z_a D$, where $Z_a\subseteq Z$ is the maximal anisotropic compact torus subgroup of $Z$. 
   
    Let $\chi: \tilde H_\varphi/L\to\R^\times$ be a character, and consider the representation of $ \tilde H_\varphi$ on $\R^2$ defined by 
    \begin{equation} \label{eq:repR2}
     h\cdot(x,y):=\left(\chi(q(h))x,\chi(q(h))^{-1}y\right),\,\forall (x,y)\in\R^2,\,\forall h\in \tilde  H_\varphi.
    \end{equation}
    Since $ \tilde H_\varphi\subseteq G$ is observable,  the above representation  extends to a representation of $G$, see \cite[Theorem 9]{Observable_groups_over_fields}. Now, pick $x_0,y_0\in\R\setminus\{0\}$. Because $\psi$ is non-contracting, we get that $\psi(t)\cdot (x_0,0)$ and $\psi(t)\cdot (0,y_0)$ are bounded away from $0$. In view of \eqref{eq:repR2}, we deduce that $\psi(t)\cdot(x_0,y_0)$ is bounded. So, by Corollary \ref{cor:proving the fixing property}, $\psi(t)\cdot (x_0,y_0)=(x_0,y_0)$. Thus,  $\chi\circ q(\psi(t))=1$ for all $t\geq 0$.
\end{proof}

\subsection{The non-contracting property under homomorphisms}\label{sec:the non-contraction under homomo}
In this subsection, we show that the homomorphic image of a definable non\-/contracting curve is non\-/contracting, Lemma \ref{lem:image of non contracting is non contracting}, and also we show that if a definable curve in a homomorphic image is non\-/contracting, then it has a definable non\-/contracting lift, Lemma \ref{lem:non-contracting under homomorphisms}. Finally, we establish that the non\-/contracting property is intrinsic for curves in groups whose radical is unipotent,  \Cref{prop:non contracting is intrinsic when radical is unipotent}.

\begin{lemma} \label{lem:image of non contracting is non contracting}
Let $G_1$ and $G_2$ be the real points of affine algebraic groups over $\R$, and let $f:G_1\to G_2$ be a homomorphism of algebraic groups. Let $\varphi:[0,\infty)\to G_1$ be a continuous,  non-contracting curve in $G_1$ definable in $\omin$. Then $f\circ \varphi$ is a continuous,  non-contracting curve in $G_2$ definable in $\omin$.   
\end{lemma}

\begin{proof}
    Let $\rho:G_2\to \GL(V)$ be a rational representation of $G_2$. Let $v\in V$. Then $\rho\circ f:G_1\to \GL(V)$ is a rational representation of $G_1$. Let $v\in V$. If $\rho(f\circ\varphi(t))v\to0$ as $t\to\infty$, then $(\rho\circ f)(\varphi(t))v\to0$ as $t\to\infty$. Since $\varphi$ is non-contracting, we get $v=0$.
\end{proof}

\begin{lemma}\label{lem:non contraction modulo normal unipotent subgroup}
    Suppose that $U\subseteq G$ is a normal, unipotent algebraic subgroup. Let $q:G\to G/U$ be the natural map. Then, a $\omin$-definable curve $\varphi:[0,\infty)\to G$ is non-contracting in $G$ if and only if $q\circ \varphi$ is non-contracting in $G/U$.
\end{lemma}
\begin{proof}
    By Lemma \ref{lem:image of non contracting is non contracting}, it suffices to prove that if $q\circ\varphi$ is non-contracting, then $\varphi$ is non-contracting. Suppose for contradiction that $q\circ\varphi$ is non-contracting and $\varphi$ is contracting. Let $V$ be a rational representation of $G$ of the smallest dimension such that there exists a $v\in V\setminus \{0\}$ with $\lim_{t\to\infty} \varphi(t)\cdot v=0$. Since $U$ is a unipotent group, the subspace of $U$-fixed points $V^U:=\{x\in V:u\cdot x=x,\forall u\in U\}$ is of positive dimension. Since $U$ is a normal subgroup, $V^U$ is $G$-invariant. Then $G$ acts on $V/V^U$, and the natural quotient map $\pi:V\to V/V^U$ is $G$-equivariant. Now, $$\lim_{t\to\infty}\varphi(t)\cdot \pi(v)=\lim_{t\to\infty}\pi(\varphi(t)\cdot v)=0.$$Since $\dim(V/V^U)<\dim(V)$, we get that $\pi(v)=0$. Namely, $v\in V^U$, and again, since $V$ is of smallest dimension with a $\varphi(t)$-contracting nonzero vector, we get $V=V^U$. Since $U$ acts trivially on $V$, the action of $G$ on $V$ factors through the action of $G/U$. Therefore, $\varphi(t)\cdot v=(q\circ\varphi(t))\cdot v$. Since $\varphi(t)\cdot v\to 0$ as $t\to\infty$, and $q\circ\varphi$ is non-contracting, we get a contradiction.
\end{proof}
\begin{lemma}\label{lem:existence of definable lifts}
Let $G_1$ and $G_2$ be the real points of algebraic groups over $\R$. Suppose that $f:G_1\to G_2$ is a homomorphism of algebraic groups, and $\varphi:[0,\infty)\to f(G_1)$ is a curve definable in $\mathscr S$. Then there exists a \emph{definable lift} $\psi:[0,\infty)\to G_1$ of $\varphi$. Namely, $\psi$ is definable in $\mathscr S$, and $f\circ \psi =\varphi$. 
\end{lemma}
\begin{proof}
    Consider the definable set $$\mathcal D:=\{(t,g):g\in G_1,~t\geq 0,~f(g)=\varphi(t)\}.$$By the choice function theorem (see Theorem \ref{thm:choice function}), there exists a definable curve $\psi:[0,\infty)\to G_1$ such that $(t,\psi(t))\in \mathcal D$ for all $t\geq 0$.
\end{proof}
\begin{lemma}\label{lem:lifts for central finite covers of ss groups}
    Let $G_1$ and $G_2$ be the $\R$-points of semi-simple connected algebraic groups over $\R$, and let $f:G_1\to G_2$ be a surjective homomorphism with finite kernel. Let $\psi:[0,\infty)\to G_1$ be a definable curve. Suppose that $f\circ \psi$ is non-contracting for $G_2$, then $\psi$ is non-contracting for $G_1$. 
\end{lemma}
\begin{proof}
  Let $V$ be a finite-dimensional irreducible rational representation of $G_1$. By Schur's lemma, the center of $G_1$ acts via a character on $V\otimes \C$. Then, $\ker f$, which is central in $G_1$, acts trivially on the $n$-fold tensor product $W=\otimes^n(V\otimes \C)$, where $n$ is the order of $\ker f$. Therefore, the action of $G_1$ on $W$ factors through an action of $G_1/(\ker f)\cong G_2$ on $W$. Let $v\in V$. If $\psi(t)\cdot v\to 0$, then $f\circ\psi(t)\cdot  (v\otimes\cdots \otimes v)\to 0$ as $t\to\infty$. But, since $f\circ\psi$ is non-contracting, we get that $v=0$.
  \end{proof}
  
\begin{lemma}\label{lem:non-contracting under homomorphisms}
    Let $G_1$ and $G_2$ be the real points of affine algebraic groups over $\R$, and let $f:G_1\to G_2$ be a surjective homomorphism of algebraic groups. Let $\varphi:[0,\infty)\to G_2$ be a non-contracting curve in $G_2$ definable in $\omin$. Then, $\varphi$ has a $\omin$-definable lift $\psi:[0,\infty)\to G_1$ which is non-contracting in $G_1$. 
\end{lemma}

\begin{proof} In this proof, all subgroups considered are the open subgroups of $\R$-points of algebraic groups defined over $\R$.

    Since the hull is generated by one-parameter unipotent subgroups, by modifying $\varphi$ with a bounded correcting curve and replacing $G_2$ by $(G_2)_u$, we may assume without loss of generality that the radical of $G_2$ is unipotent, and we denote it by $U_2$. Let $\bar G_2=G_2/U_2$ and $q_2:G_2\to \bar G_2$ be the quotient homomorphism. Let $U_1$ be the unipotent radical of $G_1$. Then, there exists a reductive subgroup $M_1$ of $G_1$ such that $G_1=M_1U_1$ and $M_1\cap U_1=\{e\}$, see \cite[Page 11]{RA72}. Since $f$ is surjective, $f(U_1)$ is a unipotent normal subgroup of $G_2$, and hence $f(U_1)\subseteq U_2$. In fact, $f(U_1)=U_2$. This can be proved as follows. $f^{-1}(U_2)=(M_1\cap f^{-1}(U_2))U_1$. Since $f^{-1}(U_2)$ is a normal subgroup of $G_1$,  $M_1\cap f^{-1}(U_2)$ is reductive. So $f(M_1\cap f^{-1}(U_2))$ being a reductive subgroup of $U_2$, it is trivial. Since $f$ is surjective, we get $f(U_1)=U_2$. Therefore, $q_2\circ f:M_1\to \bar G_2$ is surjective. Since $\bar G_2$ is semisimple, $q_2\circ f([M_1,M_1])=\bar G_2$. Let $L_2=[M_1,M_1]\cap\ker(q_2\circ f)$. Since $L_2$ is a normal subgroup of the semisimple group $[M_1,M_1]$, there exists a semisimple normal subgroup $L_1$ of $[M_1,M_1]$ such that $[M_1,M_1]=L_1L_2$, where $L_1\cap L_2$ is a finite central subgroup of $L_1$. Therefore, $q_2\circ f: L_1\to \bar G_2$ is a surjective homomorphism whose kernel is $L_1\cap L_2$. Also, since $\bar G_2$ is generated by unipotent one-parameter subgroups, it has no compact simple normal subgroup. Therefore, $L_1$ also does not have a compact normal subgroup, and $L_1$ is generated by unipotent one-parameter subgroups.  
    
    Let $\tilde G_1=L_1U_1$. Since $q_2\circ f(L_1)=G_2/U_2$ and $f(U_1)=U_2$, we have $f:\tilde G_1\to G_2$ is a surjective map. Therefore, by Lemma~\ref{lem:existence of definable lifts}, there exists a definable $\psi:[0,\infty)\to \tilde G_1$ such that $f\circ \psi=\varphi$. Let $q_1:\tilde G_1\to \tilde G_1/U_1$ be the natural quotient map. Since $q_1:L_1\to \tilde G_1/U_1$ is surjective (it is an isomorphism), let $\tilde\psi:[0,\infty)\to L_1$ be a definable curve such that $q_1\circ\tilde\psi=q_1\circ\psi$. Since $U_1\subseteq\ker (q_2\circ f)$, $q_2\circ f$ factors through $q_1$. Therefore, we have $(q_2\circ f)\circ\tilde\psi=(q_2\circ f)\circ\psi=q_2\circ\varphi$. Now $q_2\circ f:L_1\to \bar G_2$ is surjective with finite central kernel, and $q_2\circ \varphi$ is definable and non-contracting in $\bar G_2$. Therefore, by Lemma~\ref{lem:lifts for central finite covers of ss groups}, $\tilde\psi$ is non-contracting in $L_1$. Therefore, $q_1\circ\psi=q_1\circ\tilde\psi$ is non-contracting in $\tilde G_1/U_1$. Therefore, by Lemma~\ref{lem:non contraction modulo normal unipotent subgroup}, $\psi$ is non-contracting in $\tilde G_1$. 
    Finally, by Lemma \ref{lem:image of non contracting is non contracting} we get that $\psi$ is non-contracting in $G_1$.
\end{proof}

The following observation extends Lemma~\ref{lem:lifts for central finite covers of ss groups}. 


\begin{proposition}\label{prop:non contracting is intrinsic when radical is unipotent}
    Suppose that $G_1$ is a real algebraic group generated by unipotent one-parameter subgroups. Let $f:G_1\to G_2$ be a homomorphism of real algebraic groups with finite kernel. Then a definable curve $\varphi:[0,\infty)\to G_1$ is non-contracting in $G_1$ if and only if $f\circ\varphi$ is non-contracting in $G_2$.
\end{proposition}

\begin{proof}
    If $\varphi$ is non-contracting in $G_1$, then by Lemma~\ref{lem:image of non contracting is non contracting}, it follows that $f\circ\varphi$ is non-contracting in $G_2$. 
    
    Now suppose that $f\circ\varphi$ is non-contracting in $G_2$. Since $G_1$ is generated by unipotent one-parameter subgroups, $f(G_1)$ is a subgroup of $G_2$ generated by unipotent one-parameter subgroups. Therefore, the solvable radical of $f(G_1)$ is unipotent. Hence, $f(G_1)$ is an observable subgroup of $G_2$, see~\cite[Corollary~2.8]{Observable_group_grosshans}. Therefore, given any finite-dimensional rational representation of $V$ of $f(G_1)$, $V$ is a subrepresentation of a finite-dimensional representation $W$ of $G_2$ restricted to $f(G_1)$. Let $v\in V$. Suppose that $f\circ\varphi(t)\cdot v\to 0$. Since $v\in V\subseteq W$ and $f\circ\varphi$ is non-contracting in $G_2$, we conclude that $v=0$. This proves that $f\circ\varphi$ is non-contracting for $f(G_1)$. Therefore, replacing $G_2$ by $f(G_1)$, we may assume that $f$ is surjective. Therefore, by Lemma~\ref{lem:non-contracting under homomorphisms} there exists a lift $\psi:[0,\infty)\to G_1$ of $f\circ\varphi$ which is definable and non-contracting in $G_1$. By the monotonicty theorem (see Theorem \ref{thm:monotonicity thm}), there is a $T_0\geq 0$ such that $\varphi$ and $\psi$ are continuous on $[T_0,\infty)$.  Let $z=\varphi(T_0)\psi(T_0)^{-1}\in\ker f$. So, $t\mapsto z\psi(t)$ and $t\mapsto \varphi(t)$ are two lifts of $f\circ\varphi$ from $[T_0,\infty)$ to $G_1$ from the point $\varphi(T_0)\in G_1$. Since $\ker f$ is finite, $f$ is a covering map, we get
    $\varphi(t)=z\psi(t)$ for all $t\in [0,\infty)$. Since $\psi$ is non-contracting, so is $\varphi$.
\end{proof}

\section{\texorpdfstring{$(C,\al)$}{(C,α)}-good property} 
\label{sec:c alpha good property}

In our setting, a property more fundamental than the $(C,\alpha)$-good property (Theorem \ref{thm:C_alpha_goodness}) is the following.

\begin{definition}
	Let $\de\in (0,1]$. A family $\tilde \F$ of real functions defined on $[T_0,\infty)$ is called \textit{$\de$-good} if there exists a constant $M(\de)$ depending only on $\de$ such that
	\begin{equation}
		\frac{\|f\|_{I}}{\|f\|_{I_{\de}}} \le M(\de), 
	\end{equation}
	for all $f\in \tilde\F$ and all bounded sub-intervals $I_{\de}\subseteq I\subseteq [T_0,\infty]$ satisfying $|I_{\de}|=\de |I|$.
\end{definition}

Such inequalities are well-known in the literature as Remez-type inequalities; see \cite{Remez_ineq}. 

We will prove the following result, from which Theorem \ref{thm:C_alpha_goodness} will be deduced. For the rest of this section, we fix a \emph{polynomially bounded} o-minimal structure $\pbomin$. 
\begin{theorem}\label{thm:delta goodness for linear combinations of o-minimal} 
Let $\V $ be a finite-dimensional vector space of real functions defined on $[0,\infty)$ definable in $\pbomin$. Suppose that $\F\subseteq \V$ is a closed definable cone such that
\[
\lim_{t\to\infty}f(t)\neq 0, \ \forall f\in \F\setminus \{0\}.
\]
Then, there exists $T_0>0$ such that the family of functions $\tilde\F$ of the nonzero functions in $\F$ restricted to $[T_0,\infty)$ is $\delta$-good, for all $\de\in(0,1]$.
\end{theorem}
\begin{remark}
The $T_0$ in \Cref{thm:delta goodness for linear combinations of o-minimal} can be chosen as a $T_0$, which is obtained by the outcome of \Cref{lem:Wronskian lemma}.
\end{remark} 

We now proceed to prove Theorem  \ref{thm:C_alpha_goodness} by assuming Theorem \ref{thm:delta goodness for linear combinations of o-minimal}, and the rest of the section will be dedicated to proving Theorem \ref{thm:delta goodness for linear combinations of o-minimal}.
\subsection{Proving \texorpdfstring{Theorem~\ref{thm:C_alpha_goodness}}{Theorem 1.14} via \texorpdfstring{$\delta$}{δ}-goodness}

We first note the following corollary of Theorem \ref{thm:delta goodness for linear combinations of o-minimal}.
\begin{corollary}\label{cor: the exponent for delta goodness}
Let $\V $ be a finite-dimensional vector space of real functions on $[0,\infty)$ which are definable in $\pbomin$. Suppose that $\F\subseteq \V$ is a closed definable cone such that for all $f\in \F\setminus \{0\}$, $$\lim_{t\to\infty}f(t)\neq 0.$$Let $T_0\geq1$ such that the outcome of Theorem \ref{thm:delta goodness for linear combinations of o-minimal} holds. Then, there exist $\mathrm{m},r>0$ such that\emph{:} \begin{equation}
    \|f\|_I\leq \mathrm{m} x^r \|f\|_{I'},
\end{equation}
for all $x\geq 1$,  $f\in\F\setminus \{0\}$, and bounded intervals $I'\subseteq I\subseteq [T_0,\infty)$ such that $\frac{|I|}{|I'|}\leq x$.
\end{corollary}
\begin{proof}
    Consider the definable set (see \Cref{lemma:sup-definable}):
    \begin{equation*}
        S:=\left\{(x,y):\forall I'\subseteq I\subseteq [T_0,\infty),~\frac{|I|}{|I'|}\leq x,~\frac{\|f\|_I}{\|f\|_{I'}}\leq y,~\forall 0\neq f\in \F \right\}.
    \end{equation*}
    By Theorem \ref{thm:delta goodness for linear combinations of o-minimal}, the projection of $S$ to the first coordinate includes $[1,\infty).$ Then, by the definable  choice function theorem (see Theorem \ref{thm:choice function}), there exists a function $\phi:[1,\infty)\to \R^2$ definable in the same polynomially bounded o-minimal structure, such that $$(x,\phi(x))\in S,~\forall x\geq 1.$$ 
    Thus, for all $x>1$,  $f\in\F\setminus \{0\}$, and intervals $I'\subseteq I\subseteq [T_0,\infty)$ such that $\frac{|I|}{|I'|}\leq x$, it holds that $$\|f\|_I\leq \phi(x)\|f\|_{I'}.$$ 
    Since $\phi$ is polynomially bounded, there exist $r>0$ and $T_1\geq 1$ such  that $\frac{\phi(x)}{x^r}\leq c_0$ for all $x\geq T_1$, for some $c_0>0$. Finally, since $\frac{\phi(x)}{x^r}$ is bounded in $[1,T_1]$, the result follows. 
\end{proof}
 Our proof below is inspired by 
\cite[proof of Proposition 3.2]{Klein_Marg_flows_and_dioph}.
\begin{proof}[Proof of Theorem \ref{thm:C_alpha_goodness} assuming Theorem \ref{thm:delta goodness for linear combinations of o-minimal}]
    We recall that the number of connected components of the family of definable sub-level sets 
    $$\{t\in I : |f(t)|\leq \e\},~\e>0,~I\subseteq [T_0,\infty),~f\in\F\setminus \{0\},
$$
is bounded uniformly, say by $K$, see Theorem~\ref{thm:uniform bound on fibers}. 
Now fix  $I\subseteq [T_0,\infty),~\e>0$ and an $f\in\F\setminus \{0\}$. Let $$I'\subseteq \{t\in I : |f(t)|\leq \e\}$$be an interval of maximum length. 
Then 
\[
L\leq  K|I'|\text{, where } L:=|\{t\in I : |f(t)|\leq \e\}|.
\]
Therefore, ${|I|}/{|I'|}\leq {|I|}/{(L/K)}$. By \Cref{cor: the exponent for delta goodness}, we get$$\|f\|_I\leq \mathrm{m}\left(\frac{|I|}{L/K}\right)^r\|f\|_{I'}\leq \mathrm{m}\left(\frac{|I|}{L/K}\right)^r\epsilon.$$
Therefore, 
$$|\{t\in I : |f(t)|\leq \e\}|=L\leq \mathrm{m}^{\frac{1}{r}}K \left(\frac{\epsilon}{\|f\|_I}\right)^{\frac{1}{r}}|I|.$$
Thus, $f$ is $(\mathrm{m}^{\frac{1}{r}}K,1/r)$-good on $[T_0,\infty)$, see \eqref{eq:C-alpha-good}.
\end{proof}

\subsection{Proving Theorem \ref{thm:delta goodness for linear combinations of o-minimal}}
  The space of polynomials of bounded degrees is the prototypical example of $\de$-good functions. We record the following well-known result, as it will play a role in the proof of Theorem \ref{thm:delta goodness for linear combinations of o-minimal}. 

\begin{proposition}[{\cite[Lemma~4.1]{Dani1993LimitDO}}]\label{prop:delta goodness of polynomials}
    Fix $n\in \N$. Then, for every $\de\in(0,1]$, nested intervals $I_\de \subseteq I$ with $|I_{\de}|=\de |I|>0$, and every nonzero polynomial $c_0+c_1x+\cdots+c_nx^n$, it holds that\emph{:}
	\begin{equation}
		\frac{\|c_0+c_1x+\cdots+c_nx^n\|_I}{\|c_0+c_1x+\cdots+c_nx^n\|_{I_{\de}}} \le (n+1)\frac{n^n}{\delta^n} .
	\end{equation} 
\end{proposition}


We will be using the following convenient fact. Recall by the monotonicity theorem (Theorem \ref{thm:monotonicity thm}) that the derivative of a definable function $f(t)$ on an unbounded interval exists for all large $t$.
\begin{lemma}\emph{\cite[Proposition 3.1]{Miller94}}\label{lem:der asymp} Let $f:[0,\infty)\to \R$ be a function definable in  $\pbomin$ such that $f(t)\sim t^r$ as $t\to\infty$, where $r\neq0$. Then, $f'(t)\sim rt^{r-1}, \text{ as } t\to\infty$. 
\end{lemma}

\begin{lemma}\label{lem:linear independence} Let $F_i:[0,\infty)\to\R$ be functions definable in $\pbomin$, for $i=0,1,\ldots,N$, and assume that none of the $F_i$ are eventually zero. Suppose that the degrees $\deg(F_0),\ldots,\deg(F_N)$ are all distinct.
Then, there exists a $T_0>0$ such that for any interval $I\subseteq [T_0,\infty)$ of positive length, it holds that $$\|c_0F_0+c_1F_1+...+c_NF_N\|_I>0$$ for all $(c_0,c_1,\ldots,c_N)\neq0$.
\end{lemma}

\begin{proof}
By o-minimality (Theorem~\ref{thm:uniform bound on fibers}), the definable set $F_0^{-1}(0)$ has finitely many connected components. Since $F_0$ is not eventually  zero, for some $T_0>0$, $F_0(t)\neq 0$ for all  $t\geq T_0$. Thus for all $I\subseteq [T_0,\infty)$ we have$$\left\|c_0F_0+c_1F_1+\cdots+c_NF_N\right\|_I>0\iff\left\|c_0+c_1\frac{F_1}{F_0}+\cdots+c_N\frac{F_N}{F_0}\right\|_I>0. $$
Since the degrees of $F_0,F_1,...,F_N$ are distinct,  the degrees of $$1,F_1/F_0,...,F_N/F_0,$$ are distinct. Thus, we are reduced to proving the statement under the assumption that $F_0(t)=1$ for all $t\geq T_0$.


Pick $T_1\geq T_0$ such that the functions $F_1,\ldots,F_N$ are differentiable in $[T_1,\infty)$, see Theorem \ref{thm:monotonicity thm}. 
By Lemma \ref{lem:der asymp}, $\deg(F_i')=\deg(F_i)-1$ and $F_i'$ is not eventually zero for $1\leq i\leq N$. So, the degrees of $F_1',F_2',\cdots,F_N'$ are all distinct. Now, for any interval $I\subseteq[T_1,\infty)$ of positive length, 
\[
\text{if }\|c_0+c_1F_1+...+c_NF_N\|_I=0\text{, then }
\|c_1F_1'+...+c_NF_N'\|_I=0.
\]
In view of the above observations, the result is straightforward to deduce by induction. 
\end{proof}
Next, for definable functions $F_0(t),F_1(t),...,F_N(t)$, we consider the \emph{ Wronskian matrix}  $W(F_0,...,F_N)(t)$, 
which is defined to be the $(N+1)\times (N+1)$ matrix whose $k$-th row is $(F^{(k)}_0(t),F_1^{(k)}(t),\ldots,F_N^{(k)}(t))$ for  $0\leq k\leq N$. The Wronskian  is well defined for all $t$ large enough due to Theorem \ref{thm:monotonicity thm}.

\begin{lemma}\label{lem:Wronskian lemma}
Let the notation be as in \Cref{lem:linear independence}. Then, there exists a $T_0>0$  such that the functions $F_0,...,F_N$ are $(N+1)$-time continuously differentiable and $W(F_0,...,F_N)(t)$ is non-singular for all $t\geq T_0$. 
\end{lemma}

\begin{proof}
Since $\det(W(F_0,...,F_{N})(t))$ is a definable function, either $$|\det(W(F_0,...,F_N)(t))|>0,$$for all large $t$ or $\det(W(F_0,...,F_N)(t))=0$ for all large $t$. We will prove the result by induction. Since $F_0$ is not eventually zero, the result follows for $N=0$. Now suppose $N\in\N$ and $T_0>0$ such that $\det(W(F_0,...,F_{N-1})(t))\neq 0$ for all $t\geq T_0$. Now, suppose $\det(W(F_0,...,F_N)(t))=0$ for all $t\geq T_1\geq T_0$. Then, by a classical result of B\^{o}cher \cite[Theorem~II]{Bochner_certain_cases_Wronksian_implies_dep},  $F_0,\ldots,F_N$ are linearly dependent on $[T_1,\infty)$. This contradicts \Cref{lem:linear independence}.
\end{proof}

\begin{lemma}\label{lem:convenient basis for F}
    Let $\V$ be a real finite-dimensional vector space spanned by functions definable in $\pbomin$, and suppose that the zero function is the only function in $\V$ which is eventually constantly zero. Then, there exists a basis $\{F_1,...,F_N\}$ of $\V$ such that: \begin{equation}\label{eq:basis of ordered degress}
        \deg(F_1)<\cdots<\deg(F_N).
    \end{equation}
\end{lemma}
\begin{proof}
    We argue by induction. The claim is trivial for  $N=1$. Now choose an arbitrary basis $\{H_1,...,H_N\}$ of $\V$, and suppose without loss of generality that $\deg(H_i)\leq \deg(H_N) $ for all $i$. Moreover, we claim that there is no loss of generality in assuming that $\deg(H_i)<\deg(H_N)$, for all $i<N$. Indeed, for each $i<N$ such that $\deg(H_i)=\deg(H_N)$, we can find $c_i\neq 0$ such that $\deg(H_i-c_iH_N)<\deg(H_N)$. By replacing $H_i$ with $H_i-c_iH_N$, we again get a basis for $\V$. Note that the vector space $\mathcal U:=\Span_{\R}\{H_1,...,H_{N-1}\}$ is $(N-1)$-dimensional, and $\deg(f)<\deg(H_N)$ for all $f\in \mathcal U$. By induction, there is a basis $\{F_1,...,F_{N-1}\}$ for $\U$ such that $\deg(F_i)<\deg(F_{i+1})$ for all $i$. Then, $\{F_1,...,F_{N-1},F_N\}$ with $F_N:=H_N$ is the required basis for $\V$.
\end{proof}
\begin{remark}\label{rem:trivial subspace of eventually zero constant functions}
     If $\V$ is a finite-dimensional vector spanned by definable functions, then the subset of functions which are eventually constantly zero forms a finite-dimensional subspace $\V_0\subseteq \V$. Then, there exists a uniform $T_0>0$ such that for all $f\in \V_0$, it holds that
  $f(t)=0$, for all $t\geq T_0$. Thus, by restricting the functions in $\V$ to $[T_0,\infty)$, we may assume $\V_0=\{0\}$.
\end{remark}

\begin{proof}[Proof of Theorem \ref{thm:delta goodness for linear combinations of o-minimal}] 
Let $\V$ be a finite-dimensional vector space of functions $f:[0,\infty)\to\R$ definable in a polynomially bounded o-minimal structure. 
By Remark \ref{rem:trivial subspace of eventually zero constant functions}, there is no loss in generality in assuming that only the zero function is eventually constantly zero. Using Lemma \ref{lem:convenient basis for F}, we may further assume, by possibly enlarging $\V$, that $\V$ has a basis of the form $\{F_1^-,...,F_n^-,F_0^+,F_1^+,...,F_m^+\}$ such that
$$\deg(F_1^-)<\cdots<\deg(F_n^-)<0\leq\deg(F^+_0)<\cdots<\deg(F_m^+),$$
where $m,n\in\N$. Let $N=m+n+1$. We will denote
\begin{equation}\label{eq:exponents def}
  r_i=\deg(F^+_i), \quad -\ka_j:=\deg(F^-_j),  
\end{equation}where $r_i,\ka_j\geq0$. In addition, we assume for convenience that $$\lim_{t\to\infty} F^+_i (t)/t^{r_i}=1 \text{ and } \lim_{t\to\infty}F^-_j (t)/t^{-\ka_j}=1.$$
For any $\cf:=(c_1^-,\ldots,c_n^-,c_0^+,\ldots,c^+_m)\in\R^{N}$, we denote
\begin{gather*}
     f^-_{\cf}:=c^-_1F^-_1+\cdots+c^-_nF^-_n, \ 
     f^+_{\cf}:=c^+_0F_0^+ +\cdots+c^+_mF^+_m \text{, and }  f_{\cf}=f^-_{\cf}+f^+_{\cf}.
\end{gather*}

Let $\F\subseteq\V$ be a closed definable cone such that
\[
\lim_{t\to\infty}f(t)\neq 0,\,\forall f\in\F\setminus\{0\}.
\]

Let $\tilde \F:=\{\cf\in\R^{N}:f_{\cf}\in\F\}\cong \F$. Then $\tilde \F$ is a closed cone in $\R^{N}$. 

Fix a $T_0\geq 1$ for which the outcome of Lemma \ref{lem:Wronskian lemma} holds, and suppose for contradiction that there exists a $\de\in(0,1)$ such that $\F\setminus \{0\}$ is not $\delta$-good on $[T_0,\infty)$. Consider the following definable subset (see \Cref{lemma:sup-definable}):
\begin{equation} \label{eq:definableset}
\mathcal{A}:=\left\{(s,\cf,a,l,\al): 
\begin{array}{c}
   s\geq 1, ~ \cf\in\tilde\F\setminus\{0\},~a>T_0,~l>0, \\
   a\le \al\leq a+l-\delta l,~ \frac{\|f_{\cf}\|_{[a,a+l]}}{\|f_{\cf}\|_{[\al,\al+\de l]}}>s\\
  \end{array} 
  \right\}.    
\end{equation}
By the assumption for contradiction, the projection of $\mathcal{A}$ to the first coordinate equals $[1,\infty)$. 
By the choice function theorem (\Cref{thm:choice function}), there exists a polynomially bounded definable curve $$\phi(s):=(\cf(s),a(s),l(s),\al(s))\text{ for $s\geq 1$},$$such that $(s,\phi(s))\in \mathcal A,~\forall s\geq1.$  In particular, there exists an $C_0\neq 0$ and an $\ka\in\R$ such that
\begin{equation}\label{relation between a and T}
\frac{a(s)}{l(s)} \sim C_0 s^{\kappa},\text{ as $s\to\infty$}.
\end{equation}

The remainder of the proof is organized by dividing the argument into cases based on the sign of \( \kappa \) and on whether \( l(s) \) and \( a(s) \) are bounded or unbounded. In each case, we derive a contradiction. To guide the reader, we begin with a brief overview. 

If \( \kappa \leq 0 \), this implies that the length of the interval \( I(s) = [a(s), a(s) + l(s)] \) grows faster than its starting point. The case where \( l(s) \) is bounded is easier to handle, since in that case \( I(s) \) remains a bounded interval. When \( l(s) \) is unbounded, we exploit the fact that polynomially bounded definable functions \( f(t) \) are asymptotically equivalent to functions of the form \( ct^r \). 

On the other hand, if \( \kappa > 0 \), the intervals \( I(s) \) are short relative to their left endpoints. In this case, we can apply Taylor approximation, and use the \( \delta \)-goodness of polynomials (see Proposition~\ref{prop:delta goodness of polynomials}) to reach a contradiction.

\vspace{2mm}
\noindent Before proceeding to the case analysis, we record the following observation, which will be useful throughout the proof. Let \( \|\cdot\|_1 \) denote the \( L^1 \)-norm on \( \R^N \), and define
\begin{equation}
    \label{eq:hatcf}
    \hat{\cf}(s) := \frac{\cf(s)}{\|\cf(s)\|_1},\ s \geq 1.
\end{equation}
The function \( s \mapsto \hat{\cf}(s) \) is a bounded definable function with \( \|\hat{\cf}(s)\|_1 = 1 \). It follows that
\[
\lim_{s \to \infty} \hat{\cf}(s) = \mathbf{v}
\]
for some $\mathbf{v} \in \R^N $ satisfying $ \|\mathbf{v}\|_1 = 1.$ Since  $\cf(s) \in  \tilde\F \setminus \{0\}$  and $ \tilde\F$  is a closed cone, we conclude that \( \mathbf{v} \in \tilde\F \setminus \{0\} \).

\vspace{3mm}

\noindent \textbf{Case 1:  $\kappa \le 0$}.  There are two sub-cases: $\{l(s):s\geq 1\}$ is bounded or not.
\vspace{1mm}

\noindent \emph{\textbf{Case 1.1.} $l(s)$ is bounded in $s$}. Then $a(s)$ is also bounded, and by o-minimality $l(s)$ and $a(s)$ converge as $s\to\infty$. We  observe that  $$\lim_{s\to\infty}l(s)\neq0.$$
Because, since  $a(s)\geq T_0\geq 1,\forall s\geq 1$, if  $\lim_{s\to\infty}l(s)=0$, then $\lim_{s\to\infty}\frac{a(s)}{l(s)}=\infty$, which contradicts the assumption that $\kappa\leq0$.  We also note that the end-points of the nested intervals $$[\al(s),\al(s)+\de l(s)]\subseteq[a(s),a(s)+l(s)],$$ converge to the end-points of intervals  $I_\delta\subseteq I\subseteq[T_0,\infty)$ of positive length,  where  $\frac{|I_{\de}|}{|I|}=\de$.  Hence  $\|f_{\hat{\cf}(s)}\|_{[a(s),a(s)+l(s)]}$ is uniformly bounded in $s$. Now, 
\begin{align*}
\frac{\|f_{\hat{\cf}(s)}\|_{[a(s),a(s)+l(s)]}}{\|f_{\hat{\cf}(s)}\|_{[\al(s),\al(s)+\de l(s)]}}
&=\frac{\|f_{\cf(s)}/\norm{\cf(s)}_1\|_{[a(s),a(s)+l(s)]}}{\|f_{\cf(s)}/\norm{\cf(s)}_1\|_{[\al(s),\al(s)+\de l(s)]}}\\
&=\frac{\|f_{\cf(s)}\|_{[a(s),a(s)+l(s)]}}{\|f_{\cf(s)}\|_{[\al(s),\al(s)+\de l(s)]}}>s,
\end{align*}
so that
$$\frac{\|f_{\hat{\cf}(s)}\|_{[a(s),a(s)+l(s)]}}{s}\geq\|f_{\hat{\cf}(s)}\|_{[\al(s),\al(s)+\de l(s)]}.$$
By taking   $s\to\infty$,  we get $$0=\|f_{\textbf{v}}\|_{I_\de}.$$This is a contradiction to Lemma \ref{lem:linear independence}.

\vspace{1mm}

\noindent \emph{\textbf{Case 1.2}. $l(s)$ is unbounded in $s$}. In this case, we will achieve a contradiction by proving that:
$$\lim_{s\to\infty}\hat \cf(s)=(v_1^-,...,v_n^-,0,...,0),$$where
$(v_1^-,...,v_n^-)\neq 0.$ This leads to a contradiction, because $$f_{(v_1^-,...,v_n^-,0,...,0)}\in \F\setminus \{0\} \text{ and } \lim_{t\to\infty} f_{(v_1^-,...,v_n^-,0,...,0)}(t)=0,$$ is contrary to our assumption on $\F$.

\begin{remark} \label{rem:nondecay}
This is the only place in the proof of Theorem~\ref{thm:C_alpha_goodness} where we use the non-decaying function assumption for the closed cone $\F$, and it highlights the necessity of the non-contracting assumption for obtaining the $C,\alpha$-good property.
\end{remark}
\noindent We now proceed with our proof. Since $l(s)$ is an unbounded definable function, we have that  $l(s)\to\infty$. For $s\geq 1$, we define
\begin{align}\label{eq:def of Cpm and C}
    C^+(s):=\sum_{i=0}^m |c^+_i(s)|l(s)^{r_i},\,
   C^-(s):=\sum_{j=1}^n |c^-_j(s)| \text{, and } 
   C(s):=C^+(s)+C^-(s).
\end{align}
Since $s\mapsto \frac{(c^+_0(s)l(s)^{r_0},...,c^+_m(s)l(s)^{r_m})}{C(s)} $ is definable and bounded, we get that:
\begin{equation} \label{eq:w+}
    (w^+_0,...,w^+_m):=\lim_{s\to\infty} \frac{(c^+_0(s)l(s)^{r_0},...,c^+_m(s)l(s)^{r_m})}{C(s)} \in [0,1]^{m+1}.
\end{equation}

\subsection*{Claim.} There exists an interval $J\subset[0,1]$ of positive length such that $$w_0^+t^{r_0}+\cdots+w_m^+t^{r_m}=0,~\forall t\in J.$$ Here, $r_i$ are as in \eqref{eq:exponents def}.  In particular, $(w^+_0,...,w^+_m)=0$.

\begin{proof}[Proof of Claim]
We consider
\begin{equation*}
    \varphi_{{\cf}(s)}(t):=\frac{f_{\cf(s)}(a(s)+tl(s))}{C(s)},\ \forall t\in[0,1].
\end{equation*}
Recall that, by our choice function for the definable set \eqref{eq:definableset}, 
\begin{equation}\label{eq:ratios of varphis}
s<\frac{\|f_{\cf(s)}\|_{[a(s),a(s)+l(s)]}}{\|f_{\cf(s)}\|_{[\al(s),\al(s)+\de l(s)]}}= \frac{\|\varphi_{{\cf}(s)}\|_{[0,1]}}{\|\varphi_{{\cf}(s)}\|_{I_{\de}(s)}}, 
\end{equation}
where $I_{\de}(s):=\left[\frac{\al(s)-a(s)}{l(s)},\frac{\al(s)-a(s)+\de l(s)}{l(s)}\right]\subseteq [0,1]$ having length $|I_\de(s)|=\de$ for all $s\geq 1.$ Since $\frac{a(s)}{l(s)}$ is a bounded definable function, we have $$\lim_{s\to\infty}\frac{a(s)}{l(s)}=x_0\in [0,\infty).$$
Also, since the end-points of $I_\de(s)$ are bounded, they converge to end-points of a sub-interval $I_\de$ of $[0,1]$ of length $\de$. We will prove that: \begin{equation} \label{eq:phiunif:bounded}
\limsup_{s\to\infty} \norm{\varphi_{{\cf}(s)}}_{I_\de(s)}=0,
\end{equation}
and that for any $t_0\in (0,1)$:
\begin{equation} \label{eq:phiunif:conv}
\lim_{s\to\infty} \sup_{t\in[t_0,1]} \abs{\varphi_{\cf(s)}(t)- w^+_0\left(x_0 +t\right)^{r_0}+\cdots+w^+_m\left(x_0+t\right)^{r_m}}=0.
\end{equation} 
We note that the claim readily follows from \eqref{eq:phiunif:bounded} and \eqref{eq:phiunif:conv}. 
To prove \eqref{eq:phiunif:bounded}, it is sufficient to show that $$\sup_{s\to\infty}\|\varphi_{\cf(s)}\|_{[0,1]}<\infty,$$
see \eqref{eq:ratios of varphis}. For each $0\leq i\leq m $, there exists a $\nu_i>0$ such that for all $t\in [0,1]$, we have as $l\to\infty$,
\begin{align}
 \frac{F^+_i(a+tl)}{l^{r_i}}&=
\frac{\left(a+lt\right)^{r_i}}{l^{r_i}}+O\left(\frac{\left(a+lt\right)^{r_i-\nu_i}}{l^{r_i}}\right)\nonumber\\
&=\left(\frac{a}{l} +t\right)^{r_i}+O\left(l^{-\nu_i}\left(\frac{a}{l} +t\right)^{r_i-\nu_i}\right)\label{eq:homogenisity of Fi+}.
\end{align}
We can pick $S_0\geq1$ such that $l(s)\geq 1$ for all $s\geq S_0$ and that the following two statements hold: 
    \begin{enumerate}
        \item \emph{$\sup_{t\in[0,1]}\left(\frac{a(s)}{l(s)} +t\right)^{r_i}$ is uniformly bounded in $s\geq S_0$ and $0\leq i\leq m$.}
        
 Indeed, because $r_i\geq0$, $\lim_{s\to\infty}\frac{a(s)}{l(s)}=x_0$, and $t\in [0,1]$.
    
    \item \label{enu:error term bdd}\emph{$\sup_{t\in[0,1]}l(s)^{-\nu_i}\left(\frac{a(s)}{l(s)} +t\right)^{r_i-\nu_i}$ is  uniformly bounded in $s\geq S_0$ and $0\leq i\leq m$.} 

This is clear if $r_i-\nu_i\geq0$. Now suppose that $r_i-\nu_i<0$. Then, for all $t\in[0,1]$, 
    \begin{align*}
       l(s)^{-\nu_i}\left(\frac{a(s)}{l(s)} +t\right)^{r_i-\nu_i}\leq& l(s)^{-\nu_i}\left(\frac{a(s)}{l(s)} \right)^{r_i-\nu_i}\\
       =&\frac{a(s)^{r_i-\nu_i}}{l(s)^{r_i}}=  O(1),
    \end{align*}
    since $\frac{a(s)}{l(s)}=O(s^\ka)$ with $\ka\leq0$, and since $a(s)\geq T_0\geq 1$.
     \end{enumerate}
     
By \eqref{enu:error term bdd},  as $s\to\infty$, uniformly in $t\in[0,1]$,
\begin{align} \label{eq:phiunif+}
    f^+_{\cf(s)}(a(s)+tl(s))
    =\sum_{i=0}^m (c_i^+(s) l(s)^{r_i}) \left(\frac{a(s)}{l(s)} +t\right)^{r_i}+O(C^+(s)). 
\end{align}
In addition, we observe that as $s\to\infty$:
\begin{equation} \label{eq:phiunif-:bound}
\sup_{t\in[0,1]}\left\vert f_{\cf(s)}^-(a(s)+tl(s))\right\vert=O(C^-(s)).
\end{equation}
Here, $C^+,C^-$ are given by \eqref{eq:def of Cpm and C}. Now:
\begin{align*}
    |\varphi_{{\cf}(s)}(t)|&=\frac{|f_{\cf(s)}(a(s)+tl(s))|}{C(s)}\\
    &=\frac{|f^+_{\cf(s)}(a(s)+tl(s))+f_{\cf(s)}^-(a(s)+tl(s))|}{C(s)}=O(1),~\text{as $s\to\infty$,}
\end{align*}
uniformly in $t\in[0,1]$. This proves \eqref{eq:phiunif:bounded}. 

Next, we  prove \eqref{eq:phiunif:conv}. We have $\lim_{s\to\infty} l(s)=\infty$, $a(s)/l(s)>0$, and $\lim_{s\to\infty} a(s)/l(s)=x_0$. 
Fix $t_0\in (0,1)$. Then $t_0\leq \abs{a(s)/l(s)+t}\leq x_0+2$ for all $t\in [t_0,1]$, for all large $s$.  So, for any $0\leq i\leq m$, by \eqref{eq:homogenisity of Fi+}, as $s\to\infty$, 
   \[
   \sup_{t\in [t_0,1]} \left\vert \frac{F^+_i(a(s)+tl(s))}{l(s)^{r_i}}-(x_0+t)^{r_i}\right\vert = O(l(s)^{-\nu_i}).
   \] 
Let $\nu=\min\{\nu_0,\ldots,\nu_m\}>0$. Then, as $s\to\infty$, uniformly in $t\in[t_0,1]$,
\begin{align}\label{eq:fplus after t0} 
    f^+_{\cf(s)}(a(s)+tl(s))
    =\sum_{i=0}^m (c_i^+(s) l(s)^{r_i}) (x_0+t)^{r_i}+O(C^+(s)l(s)^{-\nu}).
\end{align}
Also,
\begin{equation} \label{eq:phiunif-:conv}
\sup_{t\in [t_0,1]} \left\vert f_{\cf(s)}^-(a(s)+tl(s))\right\vert=O(C^-(s)l(s)^{-\ka_n}),
\end{equation}
where $\ka_n$ is as in \eqref{eq:exponents def}. By \eqref{eq:fplus after t0} and \eqref{eq:phiunif-:conv}, for $t\in [t_0,1]$, we have
\begin{align*}
    |&\varphi_{\cf(s)}(t)-(w^+_0(x_0+t)^{r_0}+\cdots+w^+_m(x_0+t)^{r_m})|\\
    &\leq\sum_{i=0}^m \left|\frac{c_i^+(s) l(s)^{r_i}}{C(s)}-w_i^+\right|(x_0+t)^{r_i}+O(l(s)^{-\ka_n}+l(s)^{-\nu}),
\end{align*}
and the last expression decays to $0$ uniformly for $t\in [t_0,1]$, as $s\to\infty$.
\end{proof}

Using our claim and by recalling \eqref{eq:w+}, we get that
\begin{align*}
  &\lim_{s\to\infty}\frac{(c_1^-(s),...,c_n^-(s),c^+_0(s)l(s)^{r_0},...,c^+_m(s)l(s)^{r_m})}{C(s)}\\&\lim_{s\to\infty}\frac{(c_1^-(s),...,c_n^-(s),c^+_0(s)l(s)^{r_0},...,c^+_m(s)l(s)^{r_m})}{\|(c^-_1(s),...,c^-_n(s),c^+_0(s)l(s)^{r_0},...,c_m^+(s)l(s)^{r_m})\|_1}\\
  &=(w^-_1,...,w^-_n,w^+_0,...,w^+_m)=(w^-_1,...,w^-_n,0,...,0).
\end{align*}
 Therefore, 
 \begin{equation} \label{eq:w-}
    \|(w^-_1,...,w^-_n)\|_1=1. 
 \end{equation}
We now again consider $\hat\cf(s)$, as in  \eqref{eq:hatcf}. We have
 \begin{align*}
 \textbf{v}=(v^-_1,...,v^-_n,v^+_0,...,v^+_m)&:=\lim_{s\to\infty} \hat\cf(s)\\
 &=\lim_{s\to\infty}
 \frac{(c_1^-(s),...,c_n^-(s),c^+_0(s),...,c^+_m(s))}
 {\|(c^-_1(s),...,c^-_n(s),c^+_0(s),...,c^+_m(s))\|_1},
 \end{align*}
where $\|\textbf{v}\|_1=1$ and $\textbf{v}\in \tilde\F\setminus \{0\}$.  
By \eqref{eq:w-}, and since $l(s)\to\infty$ and $r_i\geq 0$:
\begin{align*}
 1&=\|(w_1^-,...,w_n^-)\|_1
 \\&=\lim_{s\to\infty}\frac{\|(c_1^-(s),...,c_n^-(s))\|_1}{\|(c^-_1(s),...,c^-_n(s),c^+_0(s)l(s)^{r_0},...,c^+_m(s)l(s)^{r_m})\|_1}\\
 &\leq\lim_{s\to\infty}\frac{\|(c_1^-(s),...,c_n^-(s))\|_1}{\|(c^-_1(s),...,c^-_n(s),c^+_0(s),...,c^+_m(s))\|_1}\\
 &=\|(v_1^-,...,v_n^-)\|_1\leq 1.   
\end{align*}
Therefore, $\|(v_1^-,...,v_n^-)\|_1=1$. As a consequence, we get $v^+_0=\cdots=v^+_m=0.$

\vspace{3mm}

\noindent\textbf{Case 2:  $\kappa > 0$.} 
Here, we consider two cases corresponding to $a(s)$ being bounded or $a(s)\to\infty$.
\vspace{1mm}

\noindent\emph{\textbf{Case 2.1.} $a(s)$ is bounded.} We have that $\lim_{s\to\infty}a(s)=x_0\geq T_0$. Since $\lim_{s\to\infty}\frac{a(s)}{l(s)}=\infty$, it follows that  $\lim_{s\to\infty}l(s)=0$. We again consider $\hat\cf(s)$, as in  \eqref{eq:hatcf}, and recall
$\lim_{s\to\infty}\hat{\cf}(s)=\textbf{v}$ with $\textbf{v}\in\tilde\F$ such that $\|\textbf{v}\|_1=1.$

By our choice of $T_0$, the basis functions $F^-_i,F^+_j$ are $(N+1)$-times continuously differentiable, where $N:=n+m+1$. 
For $t\geq T_0$, let $W(t)$ denote the Wronskian matrix $W(F_{-n}(t),\ldots,F_m(t))$, whose $k$-th row is $(F_{-n}^{(k)}(t),\ldots,F_{m}^{(k)}(t))$, for $k=0,\dots,N$. 
The Taylor polynomial of order $N$ of the function $h\mapsto f_{\hat{\cf(s)}}(a(s)+h)$ centered at $h=0$ is expressed by:
\begin{equation} \label{eq:Q-coeff}
Q_{N,a(s)}(h)=\hat{\cf}(s)W(a(s))^\top\cdot(h^0,\ldots, h^N),
\end{equation}
where $\cdot$ denotes the dot product on $\R^{N+1}$. By Taylor's theorem (mean value form for the remainder), for all $h\in[0,l(s)]$, there exists a $\xi\in a(s)+[0,l(s)]$ such that
\begin{equation*}
f_{\hat{\cf}(s)}(a(s)+h)-Q_{N,a(s)}(h)=\frac{f_{\hat{\cf}(s)}^{(N+1)}(\xi)}{(N+1)!}h^{N+1}.
\end{equation*} 
Since $a(s)$, $\hat{\underline{c}}(s)$, and $l(s)$ are bounded, we obtain that
$|f_{\hat{\cf}(s)}^{(N+1)}(\xi)|$ is uniformly bounded. 

Denote $I(s):=[0,l(s)]$, and let $I_\de(s):=[0,\de l(s)]$. Thus, we conclude:
\begin{equation*}
    \|f_{\hat{\cf}(s)}\|_{a(s)+I(s)}=\|Q_{N,a(s)}\|_{I(s)}+O(l(s)^{N+1}), 
\end{equation*}
and similarly,
\begin{equation}\label{eq:f_c on I_de with Taylor}
    \|f_{\hat{\cf}(s)}\|_{a(s)+I_{\de}(s)}=\|Q_{N,a(s)}\|_{I_\de(s)}+O(l(s)^{N+1}).
\end{equation}
Using Proposition \ref{prop:delta goodness of polynomials},
\begin{align*}
\|f_{\hat{\cf}(s)}\|_{a(s)+I(s)}
= & \|Q_{N,a(s)}\|_{I(s)}+O(l(s)^{N+1}) \\
\le & M_{N}(\de)\|Q_{N,a(s)}\|_{I_\de(s)}+O(l(s)^{N+1}) \\
\le & M_N(\de)\|f_{\hat{\cf}(s)}\|_{a(s)+I_{\de}(s)}+O(l(s)^{N+1})\\
= & M_N(\de)\|f_{\hat{\cf}(s)}\|_{a(s)+I_{\de}(s)}\left[1+O\left(\frac{l(s)^{N+1}}{\|f_{\hat{\cf}(s)}\|_{a(s)+I_{\de}(s)}}\right)\right].\nonumber 
\end{align*}
Thus, \begin{equation*}
    \frac{\|f_{{\cf}(s)}\|_{a(s)+I(s)}}{\|f_{{\cf}(s)}\|_{a(s)+I_\de(s)}}\ll1+O\left(\frac{l(s)^{N+1}}{\|f_{\hat{\cf}(s)}\|_{a(s)+I_{\de}(s)}}\right).
\end{equation*}
We will now show that
\begin{equation}\label{eq:f_c decays slower that T^n+1}
\lim_{s\to\infty}\frac{l(s)^{N+1}}{\|Q_{N,a(s)}\|_{I_{\de}(s)}}=0.
\end{equation} 
Before proceeding with the proof, note that by combining  \eqref{eq:f_c decays slower that T^n+1} with \eqref{eq:f_c on I_de with Taylor}, we obtain that  $\frac{l(s)^{N+1}}{\|f_{\hat{\cf}(s)}\|_{a(s)+I_{\de}(s)}}$ is bounded for all large $s$. This is a contradiction, since we are assuming that $s\leq\frac{\|f_{{\cf}(s)}\|_{I(s)}}{\|f_{{\cf}(s)}\|_{I_{\de}(s)}}$ for all $s\geq 1$.

We proceed to prove \eqref{eq:f_c decays slower that T^n+1}. Observe that for a polynomial $q(x)=c_0+c_1x+...+c_kx^k$ of degree bounded by $k$  and an interval $J\subseteq [0,1]$, it follows by Proposition~\ref{prop:delta goodness of polynomials}, 
\begin{equation} \label{eq:qJ-coeff}
    \|q\|_J\geq (k+1)^{-1}k^{-k} \|q\|_{[0,1]}|J|^k\gg \|(c_0,c_1,...,c_k)\|_1|J|^k.
\end{equation}
Here, we have also used that $\|q\|_{[0,1]}$ and $\|\underline c\|_1$ are two norms on the $(k+1)$-dimensional space of polynomials of degree at most $k$, and we note that the implied constant depends only on $k$. By \eqref{eq:Q-coeff}, $Q_{N,a(s)}$ is a polynomial of degree bounded by $N$ whose coefficients are given by $\hat{\cf}(s)W(a(s))^\top$. We have that $W(a(s))^\top\to W(x_0)^\top$ as $s\to\infty$, and we recall Lemma \ref{lem:Wronskian lemma} which implies that $W(x_0)^\top$ is non-singular. Therefore,  
\[
\norm{\hat{\cf}(s)W(a(s))^\top}\gg \norm{\hat{\cf}(s)}_1\norm{(W(a(s))^\top)^{-1}}^{-1}= \norm{(W(x_0)^\top)^{-1}}^{-1}>0
\]
for all $s\gg 1$. Then, by \eqref{eq:qJ-coeff} applied to $J=[0,\delta l(s)]$, we get
\begin{equation}
    \|Q_{N,a(s)}\|_{I_{\de}(s)}\gg \norm{(W(x_0)^\top)^{-1}}^{-1}\delta^Nl(s)^N.
\end{equation}
Since $\delta$ is fixed, and $l(s)\to0$, we get \eqref{eq:f_c decays slower that T^n+1}.

\vspace{1mm}

\noindent\emph{\textbf{Case 2.2.} $a(s)$ is unbounded.} Here $a(s)\asymp s^\theta$, for $\theta>0$. 
We recall that for each $\mathrm{n}\in\N$, there is a $T_{\mathrm{n}}$  such that all basis functions $F^-_i,F_j^+$ will be continuously differentiable $\mathrm{n}$ times in the ray $[T_{\mathrm{n}},\infty)$ for all $i,j$ (see Theorem \ref{thm:monotonicity thm}). By Lemma \ref{lem:der asymp}, it holds that  $$\left|\frac{d^p}{dt^p}F^-_i(x)\right|\ll_p x^{-\ka_i-p},\text{ and }\left|\frac{d^q}{dt^q}F^+_j(x)\right|\ll_q x^{r_j-q},$$where $\ka_i,r_j$ are as in \eqref{eq:exponents def}, $p,q$ are non-negative integers, and the implied constants depend on $p,q$. 
Then, for all sufficiently large integers $\mathrm{n}$  such that $r_i-(\mathrm{n}+1)<0$ for all $0\leq i\leq m$, it holds that 
$$\sup_{t\in[a(s),a(s)+l(s)]}\left\lvert \frac{d^{\mathrm{n}+1}}{dt^{\mathrm{n}+1}}f_{\hat{\cf}(s)}(t)\right\rvert=O_\mathrm{n}(a(s)^{r_m-(\mathrm{n}+1)}),$$ 
where the implied constant depends on $\mathrm{n}$. By Taylor's theorem, for the Taylor polynomial $P_{\mathrm{n},a(s)}(t)$ of $f_{\hat{\cf}(s)}(t)$ of degree $\mathrm{n}$ centered around $a(s)$, we have
\begin{align}
 \sup_{t\in [a(s),a(s)+l(s)]}|f_{\hat{\cf}(s)}(t)-P_{\mathrm{n},a(s)}(t)|&=O_\mathrm{n}(a(s)^{r_m-(\mathrm{n}+1)}l(s)^{\mathrm{n}+1}) \nonumber\\ 
 &=O_\mathrm{n}\left(a(s)^{r_m}\left(\frac{l(s)}{a(s)}\right)^{\mathrm{n}+1}\right) \nonumber\\
 &=O_\mathrm{n}\left(s^{r_m\theta}s^{-(\mathrm{n}+1)\kappa}\right), \label{eq:f-P}
\end{align}
where the last line follows by the assumption that $a(s)/l(s)\sim s^\ka$ with $\ka>0$. 
Due to Lemma~\ref{lem:linear independence}, the function $\Psi(s):=\|f_{\hat{\cf}(s)}\|_{[\al(s),\al(s)+\de l(s)]}$ is positive. Since $\Psi$ is definable in a polynomially bounded o-minimal structure, there exists an $\eta>0$ such that
\begin{equation} \label{lower bound for decay of sup-norms on subintervals al be}
\|f_{\hat{\cf}(s)}\|_{[\al(s),\al(s)+\de l(s)]}\gg s^{-\eta},~\text{as }s\to\infty.   
\end{equation}We fix $\mathrm{n}$ large enough such that $\nu:=r_m\theta-(\mathrm{n}+1)\kappa<-\eta.$
We denote $J(s)=[a(s),a(s)+l(s)]$ and we let $J_\de(s)=[\al(s),\al(s)+\de l(s)]$. Then:
\begin{align*}
\|f_{\hat{\cf}(s)}\|_{J(s)}
\underbrace{\le}_{\eqref{eq:f-P}} & \  \|P_{\mathrm{n},a(s)}\|_{J(s)}+O(s^{\nu})  \\
\underbrace{\le}_{\text{Prop. } \ref{prop:delta goodness of polynomials}} & M_{\mathrm{n}}(\de)\|P_{\mathrm{n},a(s)}\|_{J_{\de}(s)}+O(s^{\nu}) \\
\underbrace{\le}_{\eqref{eq:f-P}} &\  M_\mathrm{n}(\de)\|f_{\hat{\cf}(s)}\|_{J_{\de}(s)}+O(s^{\nu}) \\
= M&_\mathrm{n}(\de)\|f_{\hat{\cf}(s)}\|_{J_{\de}(s)}\left[1+O\left(\frac{s^{\nu}}{\|f_{\hat{\cf}(s)}\|_{J_{\de}(s)}}\right)\right] \nonumber\\
\underbrace{\ll}_{\eqref{lower bound for decay of sup-norms on subintervals al be}} & M_\mathrm{n}(\de)\|f_{\hat{\cf}(s)}\|_{J_{\de}(s)}\left[1+O\left(\frac{s^{\nu}}{s^{-\eta}}\right)\right] 
\end{align*}
Namely, $\frac{\|f_{\hat{\cf}(s)}\|_{J(s)}}{\|f_{\hat{\cf}(s)}\|_{J_{\de}(s)}}$ is bounded for all large $s$, contradicting the assumption $s<\frac{\|f_{\hat{\cf}(s)}\|_{J(s)}}{\|f_{\hat{\cf}(s)}\|_{J_{\de}(s)}}$.
\end{proof}

\subsection{Relative time near varieties}

The following proposition is an analog of \cite[Proposition 4.2]{Dani1993LimitDO}, and it is key for the linearization technique. 
\begin{proposition}\label{prop:relative time property}
    Let $\psi:[0,\infty)\to \GL(m,\R)$ be a continuous, unbounded curve definable in $\pbomin$ such that $\lim_{t\to\infty}\psi(t)v\neq0,\forall v\in\R^m\setminus \{0\}$. Fix a nonzero polynomial $Q:\R^m\to\R$, and denote 
    $$
    \A:=\{x\in\R^m: Q(x)=0\}.
    $$ 
    Then, there exists a $T_0>0$ such that the following holds: for an $\e>0$ and a compact subset $\emph{\text{K}}_1\subseteq\A$, there exists a compact subset $\emph{\text{K}}_2$ with 
    $$
    \emph{\text{K}}_1\subseteq \emph{\text{K}}_2\subseteq \A
    $$ 
    such that for every compact neighborhood $\Phi$ of $K_2$ in $\R^m$, there exists a compact neighborhood $\Psi$ of $K_1$ in $\R^m$, with $\Psi\subseteq\mathring{\Phi}$, where $\mathring{\Phi}$ denotes the interior of $\Phi$, such that:
    \begin{equation} \label{eq:RelativeTimeProperty}
        |\{t\in[a,b]:\psi(t)v\in\Psi\}|\leq \epsilon (b-a),
    \end{equation}for all  $v\in\R^m$ and $[a,b]\subseteq[T_0,\infty)$, for which: \begin{itemize}
      
        \item $\psi(t)v\in \Phi,~\forall t\in[a,b]$, and 
        \item $\psi(b)v\in \Phi\setminus\mathring{\Phi}$.
    \end{itemize}
\end{proposition}
\begin{remark}
    Compared to \cite[Proposition 4.2]{Dani1993LimitDO}, the statement our proposition is different in the assumption that the curve $\psi(t)v$ remains in  $\Phi$ throughout the whole interval $[a,b]$, and exits $\mathring{\Phi}$ at $t=b$. This is a stronger requirement than the one in \cite[Proposition 4.2]{Dani1993LimitDO} which only asks for $\psi(t_0)v\notin\mathring{\Phi}$ for some $t_0\in [a,b]$.
\end{remark}
\begin{remark}\label{rem:one polynomial determines the variety}
    Notice that for any affine variety $\A\subseteq \R^m$ there exists a polynomial $Q:\R^m\to\R$ such that $\A$ is the zero set of $Q$. Indeed, by the Hilbert basis theorem, $\A$ is the zero set of finitely many polynomials $f_1,...,f_r$, and we let $Q:=f_1^2+\cdots+f_r^2$. 
\end{remark}
Our proof of Proposition~\ref{prop:relative time property} builds on the following versions of the modified $(C,\al)$-good property.

\begin{proposition}\label{prop:composition with polynomial C,alpha goodness }
  Let $\psi:[0,\infty)\to \GL(m,\R)$ be a continuous curve definable in $\pbomin$, and let $Q:\R^m\to\R$ be a polynomial. Then there exist constants $\nu, C,\al,T_0>0$ such that for any $v\in \R^m$ exactly one of the following properties hold:
  \begin{enumerate}
      \item $Q(\psi(t)v)=0,\forall t\geq T_0$;
      \item $\Theta_v(t):=t^\nu Q(\psi(t)v)$ is $(C,\alpha)$-good on $[T_0,\infty)$.
  \end{enumerate}
\end{proposition}
\begin{proof}
    Consider the vector space $$\mathcal V:=\Span_\R\{Q(\psi(t)v),\ t\geq0 : v\in\R^m\}.$$ It is straightforward to verify that $\mathcal V$ is finite-dimensional, and that all functions in $\mathcal V$ are definable. In view of Remark \ref{rem:trivial subspace of eventually zero constant functions}, there is a $T_0\geq 0$ such that if $f\in \mathcal{V}$  is eventually zero, then $f(t)=0,\forall t\in[T_0,\infty)$. We consider $\F$ to be the space of functions of $\V$ restricted to $[T_0,\infty).$ By Lemma \ref{lem:convenient basis for F}, there is a basis $\{f_0,f_1,...,f_n\}$ for $\F$ such that $\deg(f_0)<\deg(f_1)<...<\deg(f_n)$. Let $\nu>0$ such that $\deg(t^\nu f_0)\geq0$. Then, by Theorem \ref{thm:C_alpha_goodness}, we pick $T_1>0$ such that any nonzero function  $$f\in t^\nu\otimes\F=\Span_\R\{t^\nu Q(\psi(t)v),\ t\geq T_1 : v\in\R^m\}$$is $(C,\al)$-good.
\end{proof}  

We now note the following corollary.
\begin{corollary}\label{cor:visiting time near varieties}
Let $\psi:[0,\infty)\to \GL(m,\R)$ be a continuous curve, definable in $\pbomin$. Fix a polynomial $Q:\R^m\to\R$. Then there exist $T_0,C,\alpha>0$ such that for any $v\in \R^m$ and any interval $[x,y]\subseteq [T_0,\infty)$, if $$\sup_{t\in[x,y]} |Q(\psi(t) v)|=|Q(\psi(y)v)|,$$ then
\begin{equation}\label{eq:relative time property near varieties}
   \left|\{t\in [x,y]:|Q(\psi(t) v)|\leq \eta \}\right|\leq C\left(\frac{\eta}{|Q(\psi(y)v)|}\right)^\al(y-x), 
\end{equation}for all $\eta>0$; we may say that $t\mapsto \abs{Q(\psi(t)v)}$ is {\em right-max}-$(C,\alpha)$-good in $[T_0,\infty)$. 

\end{corollary}
\begin{proof}
Fix $\nu,C,\alpha,T_0>0$ as in Proposition \ref{prop:composition with polynomial C,alpha goodness }. Let $[x,y]\subseteq [T_0,\infty)$ and let   $v\in\R^m$ be a vector such that $$|Q(\psi(y)v)|=\delta>0,$$
and such that $$|Q(\psi(t) v)|\leq \delta,\ \forall t\in[x,y].$$ Then,$$\sup_{t\in[x,y]}|t^\nu Q(\psi(t) v)|=y^\nu \delta.$$Now, by the $(C,\al)$-good property of $|t^\nu Q(\psi(t)v)|$, we get:
	\begin{align}
		\left|\{t\in [x,y] : |Q(\psi(t) v)|\leq \eta \}\right|&=\left|\{t\in [x,y] : |t^\nu Q(\psi(t) v)|\leq t^\nu \eta \}\right|\nonumber\\
 & \leq \left|\{t\in [x,y] : |t^\nu Q(\psi(t) v)|\leq y^\nu \eta \}\right|\nonumber\\
		&\leq C\left(\frac{y^\nu\eta}{y^\nu\delta}\right)^\al(y-x)=C\left(\frac{\eta}{\delta}\right)^\al(y-x).
	\end{align}
 
\end{proof}

\begin{proof}[Proof for Proposition \ref{prop:relative time property}] 
Let $\psi:[0,\infty)\to \GL(m,\R)$ be a $\pbomin$-definable, continuous unbounded curve  such that $\lim_{t\to\infty}\psi(t)v\neq0,\forall v\neq0$. Pick $T_0,C,\alpha>0$ such that  the $(C,\al)$-good property holds for the family $$\F_{\text{norm}}:=\{\|\psi(t)v\|,\ t\geq T_0 : v\in \R^m\setminus\{0\}\},$$
see Proposition~\ref{prop:c alpha good ppty of curves and action in reps}, and such that the right-max-$(C,\al)$-property holds for the  functions in $$\F_{\text{variety}}:=\{ \abs{Q(\psi(t)v)},\ t\geq T_0:v\in\R^m\},$$ which satisfy the requirements of Corollary~\ref{cor:visiting time near varieties}.

\noindent  Let $\e>0$. Given a compact set $\text{K}_1\subseteq \A$, let $R_1>0$ be such that $\text{K}_1\subseteq B_{R_1}(0)$. We choose $R_2>R_1$ is such that $C(R_1/R_2)^\alpha\leq\epsilon$, and we denote  $\text{K}_2:=\overline{\text{B}_{R_2}(0)}\cap \A$; see~\eqref{eq:R1R2} for the use of this choice. We define for $\e'>0$:
$$\A^{\e'}:=\{x\in \R^m : |Q(x)|<\e'\}.$$ 
Let $\Phi$ be a compact neighborhood of $\text{K}_2$. Then, we can pick an $\e_1>0$ such that 
\begin{equation} \label{eq:R2e1Phi}
B_{R_2}(0)\cap \A^{\e_1}\subseteq \mathring{\Phi}.
\end{equation}
We denote $\e_2:=\frac{R_1}{R_2}\e_1$, and we note that  $K_1\subseteq B_{R_1}(0)\cap \A^{\e_2}$ and we observe that $C(\e_2/\e_1)^\alpha\leq\epsilon$; see~\eqref{eq:e1e2} for the use of this choice. We consider 
\begin{equation} \label{eq:defPsi}
    \Psi:=\overline{B_{R_1}(0)\cap \A^{\e_2}}.
\end{equation}
Let $v\in\R^m$ and an interval $[a,b]\subseteq[T_0,\infty)$ such that:
    \begin{itemize}
        \item $\psi(t)v\in\Phi,~\forall t\in[a,b]$, and 
        \item $\psi(b)v\in\Phi\setminus\mathring{\Phi}$. 
    \end{itemize}
    Then by \eqref{eq:R2e1Phi}, $\|\psi(b)v\|\geq R_2$ or $|Q(\psi(b)v)|\geq\e_1.$  
    
    Suppose that $\|\psi(b)v\|\geq R_2$. By \eqref{eq:defPsi} and by Proposition \ref{prop:c alpha good ppty of curves and action in reps},  
    \begin{align}
       |\{t\in[a,b]|\psi(t)v\in\Psi\}|\leq &|\{t\in[a,b]|\|\psi(t)v\|\leq R_1\}|\nonumber\\
        \leq & C\left(\frac{R_1}{R_2}\right)^\al (b-a)\leq \e(b-a). \label{eq:R1R2}
    \end{align}
Suppose now that $|Q(\psi(b)v)|\geq\e_1$. Since $t\mapsto |Q(\psi(t)v)|$ is continuous and definable, and since a definable map takes a particular value only finitely many times, there is a decomposition 
$$\{t\in [a, b]: |Q(\psi(t)v)|\leq \e_1\}=[a_1,b_1]\sqcup\cdots\sqcup [a_k,b_k],$$
for some $k\in\N$, such that for all $1\leq i\leq k$:$$|Q(\psi(b_i)v)|=\e_1\text{ and }|Q(\psi(t)v)|\leq \e_1,\forall t\in[a_i,b_i].$$ Since $\Psi\subseteq\A^{\e_2}$, we have \begin{align*}
        \{t\in[a,b]:\psi(t)v\in\Psi\}\subseteq \{t\in[a,b]: |Q(\psi(t)v)|\leq\e_2\},
    \end{align*}and since $\e_2<\e_1$, we have \begin{align*}
       \{t\in[a,b]: |Q(\psi(t)v)|\leq\e_2\}=
       \bigsqcup_{i=1}^k\{t\in[a_i,b_i]: |Q(\psi(t)v)|\leq\e_2\}.
    \end{align*}
    We conclude by \eqref{eq:relative time property near varieties} of Corollary~\ref{cor:visiting time near varieties} that 
    \begin{align}
         |\{t\in[a_i,b_i]: |Q(\psi(t)v)|\leq\e_2\}|
        \leq C\left(\frac{\e_2}{\e_1}\right)^\al(b_i-a_i)\leq \e(b_i-a_i), \label{eq:e1e2}
    \end{align}  for each $1\leq i\leq k$. 
    This proves that $$ |\{t\in[a,b]|\psi(t)v\in\Psi\}|\leq \e(b-a). $$ 
    Thus, \eqref{eq:RelativeTimeProperty} holds in all cases.
    \end{proof}

\section{Non-escape of mass\label{sec:non escape of mass}}
Let $X^*:=X\cup \{\infty\}$ be the one-point compactification of $$X:=\SL(n,\R)/\SL(n,\Z).$$ 

In what follows, $\pbomin$ is a fixed polynomially bounded o-minimal structure. Let $\varphi:[0,\infty)\to \SL(n,\R)$ be a $\pbomin$-definable, continuous and unbounded curve. For $x\in X$ and $T>0$ let $\mu_{T,\varphi,x}$ be the measure defined in \eqref{eq:main definition of measures averaging on curve}. By the Banach-Alaoglu theorem, any weak-star limit $\mu$ as $T\to\infty$ of the measures $\mu_{T,\varphi,x}$ is a probability measure on $X^*$. We now state and prove the non-divergence property in the setup of $\SL(n,\R)/\SL(n,\Z)$. Later in the section, we extend the result to the case of finite volume quotient spaces of Lie groups. In the general case, when the Lie group is of higher rank, the non-divergence result  builds on the non-divergence in $\SL(n,\R)/\SL(n,\Z)$ by using Margulis' arithmeticity theorem see~\cite[page~3]{Zimmer:book84}.

\begin{proposition}\label{prop:non-escape of mass-SLnZ} Let $\varphi:[0,\infty)\to \SL(n,\R)$ be a $\pbomin$-definable, continuous, unbounded, non-contracting curve. Then, there exists $T_0>0$ such that given a compact set $F\subseteq \SL(n,\R)$ and $\e>0$, there exists a compact set $K\subseteq \SL(n,\R)/\SL(n,\Z)$ such that for any $g\in F$ and $T\geq T_0$, we have
\[
\abs{\{t\in[T_0,T]:\varphi(t)g\SL(n,\Z)/\SL(n,\Z)\in K\}}\geq (1-\e)(T-T_0). 
\]
\end{proposition}

We need the following preliminaries for the proof. In the discussion that follows, we will refer to rank-$k$ discrete subgroups of $\R^n$ as \textit{$k$-lattices}. For a $k$-lattice $\La\subseteq \Z^n$, let $\|\La\|$ be the volume of a fundamental domain  $F\subseteq \Span_{\R}\{\La\}$ of $\La$ with respect to the usual measure on the subspace $\Span_{\R}\{\La\}\subseteq\R^n$ (obtained by restricting the usual Euclidean inner product).  Recall that if $\{\vv_1,..., \vv_k\}$ forms a $\Z$-basis for $\La$,
then $$\|\La\|=\|\vv_1 \wedge \cdots \wedge \vv_k\|.$$ 
The norm above is defined by choosing the inner product for which the pure wedges of $k$-positively oriented tuples of the canonical basis vectors $e_i,i=1,...,n$ are orthonormal.
Consider the representation of $\SL(n,\R)$ on $\bigwedge^k\R^n$ defined by \begin{equation}\label{eq:rep of SL(n) on wedges}
    g\cdot(\vv_1\wedge \vv_2...\wedge\vv_k):=g\vv_1\wedge g\vv_2\wedge\cdots\wedge g\vv_k,~g\in\SL(n,\R).
\end{equation}For a fixed $k< n$, the action of $\SL(n,\R)$ is transitive on the space of $k$-lattices of rank $k< n$. We denote $\Z^k:=\Span_\Z\{e_1,...,e_k\}$, and we observe that for a unimodular $n$-lattice $L=g\Z^n\leq \R^n$, where $g\in\SL(n,\R)$, the collection of primitive $k$-sublattices of $L$ is given by 
$$g\ga\Z^k,\ga\in\SL(n,\Z).$$
The following powerful theorem on quantitative non-divergence due to Kleinbock \cite{Kleinbockclay} will be needed. 

\begin{theorem}\label{thm:kleinbock's non-divergence theorem}
	Let $B \subseteq \R$ be an interval, let $C\in\R, \al > 0$, $0 < \rho < 1$ and let $h : B \to \SL(n, \R)$ be a continuous map. Assume that for every $\ga\in\SL(n,\Z)$, and for every $1\leq k< n$  both of the following conditions are satisfied
	\begin{enumerate}
		\item\label{enu:Kleinbocks c,alpha good condition} the function $x \to \|h(x)\ga\cdot e_1\wedge\cdots\wedge e_k\|$ is $(C, \al)$-good on $B$, and
		\item\label{enu:Kleinbocks discreteness condition} $\sup_{x\in B}\|h(x)\ga\cdot e_1\wedge\cdots\wedge e_k\|\ge \rho^k$.
	\end{enumerate}
Then, for every positive number $\rho$ such that $\e<\rho$, we have
\begin{equation}
	|\{x\in B : \la_1(h(x)\Z^n)\le \e\}|\le C  n 2^n \left( \frac{\e}{\rho}\right)^{\al}|B|,
\end{equation}
where $\la_1(\cdot)$ is the function that outputs the length of the shortest nonzero vector of an Euclidean lattice.
\end{theorem}

\begin{proof}[Proof of Proposition~\ref{prop:non-escape of mass-SLnZ}] 
We identify $\SL(n,\R)/\SL(n,\Z)$ with the space of unimodular lattices $$\Lat_n:=\{L:=\Span_\Z\{v_1,...,v_n\}: \det(v_{i,j})=1\}.$$ Consider$$\B_\de:=\{L\in\Lat_n:  (L\setminus \{0\})\cap B_\de(0)\neq\emptyset\}.$$Here  $B_\de(0)\subseteq \R^n$ is the ball of radius $\de$ centered at the origin. By Mahler's criterion, $\Lat_n\setminus \B_\de$ is compact for all $\de>0$ and thus $\B_{\de}$ is a neighborhood of $\infty$. Fix $g\Z^n\in\Lat_n$ for $g\in\SL(n,\R)$, and consider the measures\begin{equation}
    \mu_T(f)=\frac{1}{T} \int_0^Tf(\varphi(t)g\Z^n)dt,\ \forall f\in C_c(\Lat_n).
\end{equation}Let $\mu$ be a weak-* limit of the measures $\mu_T$ as $T\to\infty$.	In order to show that $\mu(\infty)=0$, it is enough to prove that for every $\e>0$, there is $\de$ such that
	\begin{align}
		\limsup_{T\to\infty}\frac{|\{t\in [0,T] : \la_1(\varphi(t)g\Z^n)\le \de\}|}{T}\le \e. 
	\end{align}
This will be concluded by Theorem \ref{thm:kleinbock's non-divergence theorem} 
as follows. We first verify the conditions.
For $k\in\{1,2,..,n-1\}$, consider the representation \eqref{eq:rep of SL(n) on wedges}.  By Proposition \ref{prop:c alpha good ppty of curves and action in reps}, there exists $T_0,C,\al>0$ such that for any fixed $k\in\{1,2,..,n-1\}$, and $\ga\in\SL(n,\Z)$ it holds that\begin{align}
     \Theta(t)=\|\varphi(t)g\ga\cdot (e_1\wedge\cdots\wedge e_k)\|
\end{align}
is $(C,\al)$-good in $[T_0,\infty)$. Let $F\subseteq \SL(n,\R)$ be a given compact subset. Then, the number $\rho$ defined by
 $$\rho:=(\inf\{\|\varphi(T_0)g\ga\cdot( e_1\wedge\cdots\wedge e_k)\|:g\in F,\,\ga\in\SL(n,\Z)\})^{\frac{1}{k}}$$
is positive.
Thus, by Theorem \ref{thm:kleinbock's non-divergence theorem}, for any $g\in F$, 
	\begin{align}
		|\{t\in [T_0,T] : \la_1(\varphi(t)g\Z^n)\le \de\}|\le C \left( \frac{\de}{\rho}\right)^{\al}(T-T_0).
	\end{align}
\end{proof}

\subsection{Non-divergence for homogeneous space of Lie groups}

For the following statement, we recall that for a Lie group $G$, $G_u$ denotes the Lie-subgroup generated by the one-parameter unipotent subgroups of $G$.

\begin{proposition}
    \label{prop:nondiv-Lie}
    Let $G$ be a Lie subgroup of $\SL(n,\R)$ and let $\Gamma$ be a lattice in $G$. Suppose that $\varphi:[0,\infty)\to \SL(n,\R)$ is a $\pbomin$-definable, non-contracting curve. Assume further that $\varphi([0,\infty))\subseteq G_u$. Then, there exists a $T_0>0$ such that given $\e>0$ and a compact set $F\subseteq G$, there exists a compact set $K\subseteq G/\Gamma$ such that for all $g\in F$ and all $T>T_0$, 
    \begin{equation} \label{eq:nondiv-Lie}
        \abs{\{t\in [T_0,T]: \varphi(t)g\Gamma\in K\}}\geq (1-\e) (T-T_0). 
    \end{equation}
\end{proposition}
\begin{remark}
    We note that the assumption that $\varphi([0,\infty))\subseteq G_u$ is non restrictive for our purposes in view of the following. In Proposition~\ref{prop:hull in lie group}, we show that the hull of a non-contracting curve, definable in a polynomially bounded o-minimal structure and taking values in a Lie subgroup $G \subseteq \SL(n, \R)$, is contained in $G_u$. Thus, by multiplying the original curve with a correcting curve, the assumption $\varphi([0,\infty))\subseteq G_u$ will be satisfied.
\end{remark}

\begin{proof}
We follow the arguments as in~\cite[Theorem 6.1]{Dani1993LimitDO} to reduce the problem to the case when $G$ is a semisimple group and $\Gamma$ is an irreducible lattice in $G$. We note that if $G/\Gamma$ is compact, there is nothing to prove. So, we now assume that $G/\Gamma$ is not compact. 

Let $M$ be the smallest closed normal subgroup of $G$ such that $\bar G=G/M$ is semisimple with the trivial center and no compact factors. Then $M/(M\cap\Gamma)$ is compact, see~\cite[Theorem~8.24]{RA72} and \cite[Proof of Theorem~2.2]{Shah1996LimitDO}. Let $q:G\to \bar G$ denote the quotient homomorphism. Then $q(\Gamma)$ is a lattice in $\bar G$, and the natural quotient map $\bar q:G/\Gamma\to \bar G/q(\Gamma)$ is proper. We note that if $\tilde{G}$ and $\tilde M$ denote the Zariski closures of $G$ and $M$ in $\SL(n,\R)$, respectively, then $\bar G=G/M=\tilde{G}^0/\tilde{M}^0$. Therefore, $q|_{G_u}:G_u\to \bar G$ is a rational homomorphism of real algebraic groups. Thus $q\circ \varphi$ is a curve definable in $\pbomin$.
By Lemma \ref{lem:image of non contracting is non contracting}, $q\circ \varphi:[0,\infty)\to \bar G$ is a non-contracting.
Since $\bar q$ is proper, proving the theorem for $\bar G$ in place of $G$ is sufficient. 

Therefore, we assume that $G$ is semisimple with the trivial center and no compact factors. Then, by \cite[Theorem~5.22]{RA72}, there exist closed normal subgroups $G_1,\ldots,G_k$ of $G$ for some $k$ such that $G$ equals the direct product $G_1\times \cdots \times G_k$, such that  $\Gamma_i:=G_i\cap \Gamma$ is an irreducible lattice in $G_i$, for all $i$. Moreover, $\Gamma_1\times\cdots\times\Gamma_k$ is a normal subgroup of finite index in $\Gamma$. Hence, $G/\Gamma$ is finitely covered by $G_1/\Gamma_1\times \cdots \times G_k/\Gamma_k$. Let $q_i:G\to G_i$ denote the natural factor map, and observe that  $q_i\circ \varphi$ is a non-contracting curve in $G_i$ definable in $\pbomin$ by Lemma \ref{lem:image of non contracting is non contracting}. So, proving the theorem for $G_i/\Gamma_i$ separately for each $i$ will imply the theorem for $G/\Gamma$. Therefore, without loss of generality, we may assume that $\Gamma$ is an irreducible lattice in $G$, which is semisimple with the trivial center and no compact factors. 

First, suppose that the real rank of $G$ is $1$.  It is straightforward to adapt the proof of \cite[Proposition~1.2]{Dani:rank-1} to conclude \eqref{eq:nondiv-Lie}. For this purpose, it is enough to use the following property of the map $\varphi$ in place of \cite[Lemmas~2.5 and 2.7]{Dani:rank-1}: Let $V=\wedge^d\mathfrak{g}$, where $d$ is the dimension of the maximal unipotent subgroup of $G$ and $\mathfrak{g}$ the Lie algebra of $G$. Now, $G$ acts on $V$ via $\wedge^d\Ad$. Then there exists $T>0$, $C>0$ and $\alpha>0$ such that for any $g\in G$, the map $t\mapsto \norm{\varphi(t)g\cdot p}$ has the $(C,\alpha)$-good property on $[T_0,\infty)$ by Proposition~\ref{prop:c alpha good ppty of curves and action in reps}. 

Now we suppose that the real rank of $G$ is at least $2$. Then, by the arithmeticity theorem due to G. A. Margulis, $\Gamma$ is arithmetic; that is, see~\cite[page~3]{Zimmer:book84}: There exists an $m\in \N$, an algebraic semisimple group $H\subseteq \SL(m,\R)$  defined over $\Q$ and a surjective algebraic homomorphism $$\rho:H^0\to G,$$such that $\ker\rho$ is compact and $\rho(H^0(\Z))$ is commensurable with $\Gamma$, where $H^0(\Z)=H^0\cap \SL(m,\Z)$. Since $G$ has no compact factors, the map $\rho$ restricted to $H_u$ is a finite cover of $G$. By Lemma \ref{lem:non-contracting under homomorphisms}, we can lift $\varphi$ to a $\pbomin$-definable curve $\psi:[0,\infty)\to H_u\subseteq H^0\subseteq \SL(m,\R)$ which is non-contracting in $H_u$, such that $\rho\circ\psi=\varphi$. By Proposition \ref{prop:non contracting is intrinsic when radical is unipotent}, $\psi$ is non-contracting in $\SL(m,\R)$. Denote $Y:=\SL(m,\R)/\SL(m,\Z)$.  As in Proposition~\ref{prop:non-escape of mass-SLnZ}, let $T_0>0$ be such that for any compact set $F_1\subseteq \SL(m,\R)$, there exists a compact set $K_1\subseteq Y$ such that for any $T\geq T_0$ and $h\in F_1$
\begin{equation} \label{eq:nondiv-SLN}
\abs{\{t\in [T_0,T]:\psi(t)h\SL(m,\Z)\in K_1\}}\geq (1-\e) (T-T_0).
\end{equation}
Therefore, since $H^0/H^0(\Z)\hookrightarrow
 Y$ is a proper $H^0$-equivariant continuous injection, the conclusion of the proposition holds for $H^0$ in place of $G$, $H^0(\Z)$ in place of $\Gamma$, and $\psi$ in place of $\varphi$. And, since $\rho\circ\psi=\varphi$, the conclusion of the proposition holds for $\rho(H^0(\Z))$ in place of $\Gamma$, and hence it holds for $\Gamma$ due to the commensurability. 
\end{proof}

As an immediate consequence of Proposition~\ref{prop:nondiv-Lie} we obtain the following. 

\begin{theorem} \label{thm:non-escape of mass} Let $\varphi:[0,\infty)\to \SL(n,\R)$ be a $\pbomin$-definable, continuous, unbounded, non-contracting curve. Let $G\subseteq\SL(n,\R)$ be a Lie subgroup, $\Gamma$ be a lattice in $G$, and suppose that $\varphi([0,\infty))\subseteq G_u$. Then, any weak-* limit $\mu$ of $\mu_{T,\varphi,x}$ as $T\to\infty$ is a probability measure on $G/\Ga$, for all $x\in G/\Ga$.
\end{theorem}

\section{Unipotent invariance\label{sec:unip invariance} }
We say that a measure $\mu$ on $G/\Ga$ is \textit{invariant} under $g\in G$ if $g_*\mu=\mu$, where $g_*$ denotes the pushforward by the left-translation map by $g$ on $G/\Ga$. We say that a measure is \textit{invariant under a subgroup}, if it is invariant under every element of the subgroup. 

For the rest of the section, $\pbomin$ is a fixed polynomially bounded o-minimal structure.
\begin{proposition}\label{prop:invariance}Let $G\subseteq \GL(n,\R)$ be a closed subgroup. Consider a continuous, unbounded, $\pbomin$-definable curve $\varphi:[0,\infty)\to G$ such that $ \det(\varphi(t))$ is constant in $t$. Suppose further that for all $v\in\R^n\setminus \{0\}$, $\lim_{t\to\infty}\varphi(t)v\neq0$. Then there exists a nontrivial one-parameter unipotent group $P\subseteq G$ such that if  $\Ga\subseteq G$ is a discrete subgroup, $g\in G$ and $\mu$ is weak-$\ast$ limit of the measures $$\mu_{T,\varphi,g\Ga}(f):=\frac{1}{T}\int_0^Tf(\varphi(t)g\Ga)dt, f\in C_c(G/\Ga),$$then $\mu$ is invariant under $P$.
\end{proposition}
The key idea behind the proof is that o-minimal curves exhibit a kind of ``tangency at infinity.'' Specifically, one can construct definable \emph{change-of-speed} functions $h(t)$ that diverge to infinity and such that 
$$\lim_{t \to \infty} \varphi(h(t))\varphi(t)^{-1} = \rho,$$
where $\rho$ is a nontrivial element. Under our assumptions, these limiting elements $\rho$ generate a one-parameter unipotent subgroup, and moreover, it holds that  $h'(t) = 1 + o(t^{-\nu})$. The latter derivative asymptotic ensures that any limiting measure of $\mu_{T,\varphi,g\Gamma}$ as $T \to \infty$ is invariant under the subgroup generated by $\rho$; see Lemma~\ref{lem:invariance for unipotent p.s.} below.

This idea was previously applied by Shah in~\cite{Shah1994LimitDO} in the context of averages along polynomial curves. In the polynomial case, one can construct suitable change-of-speed maps using Taylor expansions. In the o-minimal setting, Peterzil and Steinhorn~\cite{PS99} showed that the collection of all possible limits of the form $\lim_{t \to \infty} \varphi(h(t))\varphi(t)^{-1}$, where $h$ is a definable function diverging to $\infty$, forms a one-dimensional torsion-free group. We refer to this group as the \emph{Peterzil-Steinhorn subgroup}, which we define in a simplified form in Definition~\ref{def:p.s. group} to suit our needs. 

Poulios, in his thesis~\cite{Poulios_thes}, further studied Peterzil-Steinhorn groups and showed that the group associated with an unbounded curve definable in a polynomially bounded structure is either unipotent or $\mathbb{R}$-diagonalizable. Importantly, Poulios establishes a convenient criterion for when the Peterzil-Steinhorn group is unipotent.

\vspace{2mm}

\noindent We now recall the required notions and results from \cite{Poulios_thes}.
For the following, for $r<1$, consider 
\begin{equation}\label{eq:change of speed r<1}
    h_{r,s}(t):=\left(t^{1-r}+(1-r)s\right)^{\frac{1}{1-r}}=t+st^{r}+o(t^r), \text{ where }s\in\R, 
\end{equation}
and let
\begin{equation}\label{eq:change of speed r=1}
    h_{1,s}(t):=st, \text{ where }s>0.
\end{equation}

\begin{proposition}\emph{\cite[Theorem 3.2.4, Theorem 3.3.1, Corollary 3.4.3]{Poulios_thes}}\label{prop:poulios summary} Let $G\subseteq\GL(n,\R)$ be a closed subgroup, and let $\varphi:[0,\infty)\to G$ be a $\pbomin$-definable, unbounded curve. Then, there exists a unique $r\leq 1$ such that \begin{equation}\label{eq:limit in Poulious condition}
	M_\varphi:=\lim_{t \to \infty} t^r \cdot \varphi'(t) \cdot \varphi(t)^{-1}
\end{equation}
with $M_\varphi\in \mathfrak{g}\neq \{0\}$, where  $\mathfrak{g}$ represents the Lie-algebra of $G$. 

\noindent It holds that $M_\varphi$ is nilpotent $\iff$ $r<1$, and $M_\varphi$ is $\R$-diagonalizable $\iff r=1$.  Moreover: 
\begin{enumerate}
    \item  If $r<1$, consider $\rho(s):=\exp(sM_\varphi),~s\in \R$. Then:
\begin{equation} 
    \rho(s):=\lim_{t\to\infty}\varphi(h_{r,s}(t))\varphi(t)^{-1},\ s\in \R. \label{eq:definition of invariance group elements rho}
\end{equation}
\item If $r=1$, let $d(s):=\exp(\log(s)M_\varphi)$, then:
\begin{equation*}
    d(s):=\lim_{t\to\infty}\varphi(h_{1,s}(t))\varphi(t)^{-1},\ s>0.
\end{equation*}
\end{enumerate}
\end{proposition}

\begin{definition}\label{def:p.s. group}
    For an unbounded curve  $\varphi:[0,\infty)\to \GL(n,\R)$ definable in a polynomially  bounded o-minimal structure, we say that $r\leq 1$ is the \emph{P.S. order} of $\varphi$, if $r$ satisfies  \eqref{eq:limit in Poulious condition}. We call the one-parameter subgroup $P$ generated by $\exp(M_\varphi)$ as the \emph{P.S. group} of $\varphi$.
\end{definition}

\begin{remark} \label{rem:phi-unip}
    In the notations of Proposition~\ref{prop:poulios summary}, if $\varphi$ is a restriction of a one-parameter group to $[0,\infty)$, then $r=0$, $h_{r,s}(t)=t+s$, and $\rho(s)=\varphi(s)$ for all $s\in \R$; that is, $\varphi$ is a restriction of a unipotent one-parameter subgroup. 
\end{remark}

\subsection{Unipotent P.S. groups and limiting unipotent invariance} 
We now show that 
when the P.S. group is unipotent,   any limiting measure of measures of the form \eqref{eq:main definition of measures averaging on curve} will be invariant under the P.S. group. This proof is found in \cite{Shah1994LimitDO} and in \cite{Peterzil2018OminimalFO}, and we give it here also for completeness. 
The following elementary calculus lemma will be needed.

\begin{lemma}\label{difference of intergal lemma}
Let $h:(0,\infty) \to \R$ be a differentiable function such that
$\lim_{t\to \infty}h'(t)= 1$. Then, for any bounded continuous function $F:[0,\infty)\to \R$, we have
\begin{equation}
    \lim_{T\to\infty}\frac{1}{T}\int_0^T (F(h(t))-F(t))dt=0
\end{equation}
\end{lemma}

\noindent For the proof below, notice that for $r<1$, it holds that $h'_{s,r}(t)=1+O(t^{r-1})$.

\begin{lemma}\label{lem:invariance for unipotent p.s.}
Let $G\subseteq\GL(n,\R)$ be a closed subgroup and let $\Ga\subseteq G$ be a discrete subgroup. Suppose that $\varphi:[0,\infty)\to G$ is a $\pbomin$-definable, continuous,  unbounded curve such that its P.S. group is unipotent \emph{(}namely, \eqref{eq:limit in Poulious condition} holds for $r< 1$\emph{)}. Let $\rho(s)=\exp(sM_\varphi)$, as in \eqref{eq:definition of invariance group elements rho}.  Then \begin{equation}
\lim_{T\to\infty}\frac{1}{T}\int_{0}^T (f(\rho(s)\varphi(t)x)- f(\varphi(t)x))dt,
\end{equation}
for all $f\in C_c(G/\Ga)$, $x\in G/\Ga$ and $s\in \R$.
\end{lemma}
\begin{proof}
For $f\in C_c(G/\Ga)$ and an $\e>0$, by uniform continuity of $f$, there is an  $T_{\e}$ such for all $t\ge T_{\e}$: \begin{align*}
    |f(\rho(s)\varphi(t)x)-f(\varphi(h_{r,s}(t)x)|&=|f(\rho(s)\varphi(t)x)-f([\varphi(h_{r,s}(t))\varphi(t)^{-1}]\varphi(t)x)|\\&\le \e/2,
\end{align*}
Then,
\begin{align*}
    &\frac{1}{T}\int_{0}^T \left(f(\rho(s)\varphi(t)x)-f(\varphi(t)x)\right)dt\\
    &=O\left(\frac{T_\e}{T}\right)+O(\e)
    +\frac{1}{T}\int_{T_{\e
    }}^T  \left(f(\varphi(h_{r,s}(t)x)-f(\varphi(t)x)\right)dt.
\end{align*}
The last term converges to zero by Lemma~\ref{difference of intergal lemma}.
\end{proof}
If a curve is contained in a unipotent group, as in \cite{Peterzil2018OminimalFO}, then the P.S. group is automatically unipotent. In an overview, to treat our more general setup, we will distinguish between two type of curves to those that are  \textit{essentially diagonal} (see the definition below), and those that are or not. We will show that if a curve is not essentially diagonal, then its P.S. group is unipotent. \begin{definition}\label{def:essentially diagonal}
	We call a curve $\{\varphi(t)\}_{t\geq 0}\subseteq \GL(n,\R)$  \textit{essentially diagonal} if there exists a decomposition \begin{equation}
		\varphi(t)=\sigma(t)b(t)C,
	\end{equation}
	where $b(t)\in \GL(n,\R)$ is a diagonal matrix for all large $t$, $C\in \GL(n,\R)$ is a constant matrix, and $\sigma(t)\in\SL(n,\R)$ is convergent   as $t\to\infty$.
\end{definition}
The following is our key observation. 
\begin{proposition}\label{prop:r is 1 iff essentially diagonal} 	Let $\varphi:[0,\infty)\to \GL(n,\R)$ be a $\pbomin$-definable, unbounded,  continuous curve. Suppose that $\det(\varphi(t))$ is constant for all large enough  $t$. Then, the P.S. group of  $\varphi$ is unipotent if and only if the curve $\varphi$ is not essentially diagonal.
\end{proposition}
We will prove the above Proposition \ref{prop:r is 1 iff essentially diagonal} in the subsection below. Before we proceed, we consider the following lemma which proves Proposition \ref{prop:invariance} by assuming Proposition \ref{prop:r is 1 iff essentially diagonal} in combination with Lemma \ref{lem:invariance for unipotent p.s.}.

\begin{lemma}\label{lem:p.s. group of non-contracting is unip}
   Let $\varphi:[0,\infty)\to \GL(n,\R)$ be a $\pbomin$-definable,  unbounded, continuous curve. Suppose that the determinant $\det(\varphi(t))$ is constant in $t$ and suppose that $\lim_{t\to\infty}\varphi(t)v\neq 0,$ for all  $v\in\R^n\setminus \{0\}$. Then, $\varphi$ is not essentially diagonal and in particular, the P.S. group of  $\varphi$  is unipotent.
\end{lemma}
\begin{proof}
    Suppose for contradiction that the P.S. group is not unipotent. Then,  by Proposition \ref{prop:r is 1 iff essentially diagonal}, $\varphi$ is essentially diagonal. Namely $\varphi(t)=\sigma(t)b(t)C$, where $C$,  $\sigma(t)\in\SL(n,\R)$ are bounded and $b(t)$ is a diagonal matrix with $\det(b(t))=c_0$ for all $t$. Since we assume that $\varphi$ is unbounded, it follows that $b(t)$ is unbounded, which in turn implies that there is an index $1\leq i\leq n$ such that the corresponding entry function $b_{ii}(t)$ on the diagonal of $b(t)$  decays to zero. In particular, for  $v:=C^{-1}e_i$ it holds that $\lim_{t\to\infty}\varphi(t)v=0$, which is a contradiction. 
\end{proof}

\subsubsection{Proving Proposition \ref{prop:r is 1 iff essentially diagonal}} 
We will call the $(i,j)$ index of  an upper-triangular matrix $b\in \SL(n,\R)$ of the form\begin{equation}
    b=\begin{bsmallmatrix}
		f_{1,1} &0 & \cdots & 0 & 0 & \cdots & 0 & 0 & \cdots  & 0 \\
		0 &  \ddots & \ddots & \vdots &\vdots & & \vdots & \vdots & & \vdots \\
  \vdots &  \ddots & \ddots & 0 & 0 & \cdots & 0 & 0 & \cdots & 0\\
		0 & \cdots & 0 &f_{i,i} &0 &\cdots & 0 & \pmb{f_{i,j}}& \cdots & f_{i,n}\\
		\vdots &  & \vdots & \ddots &\ddots & & &\vdots & & \vdots \\
		\vdots &  & \vdots &   & \ddots & \ddots & & \vdots & & \vdots \\
  	\vdots &  & \vdots &   &  & \ddots & \ddots & \vdots & & \vdots \\
    	\vdots &  & \vdots &   &  & & 0 & f_{j,j} & \cdots & f_{j,n}\\
  	\vdots &  & \vdots & & &  & & \ddots & \ddots  & \vdots \\
		0 & \cdots & 0 & \cdots & \cdots & \cdots & \cdots & \cdots & 0 & f_{n,n} 
	\end{bsmallmatrix}\label{eq:first non zero off diag},
\end{equation}
where $f_{i,j}\neq0$,  the \emph{first nonzero off-diagonal entry}. More precisely, $(i,j)$ is the first nonzero entry among the off-diagonal entries according to the following lexicographic order on $\N^2$:
\begin{align}\label{eq:lexicographic order}
    (i,j)\prec (k,l) \iff i<k \text{ or }
     (i=k \text{ and } j<l).
\end{align}Let  $B\leq \SL(n,\R)$ be the subgroup of upper-triangular matrices. Let $b:[0,\infty)\to B$ be a definable curve. Recall that each definable function $f(t)$ is either zero for all large enough $t$ or $|f(t)|>0$  for all large $t$. Thus, for all large enough $t$, there exists a unique first nonzero off-diagonal entry in $b(t)$, or $b(t)$ is diagonal for all large $t$. We will refer to this entry as \textit{the first nonzero off-diagonal entry} of the curve $b(t)$.
\begin{lemma}\label{lem:stable-unstable-constant (SUC) decomposition}


Let $\varphi:[0,\infty)\to \SL(n,\R)$ be a $\pbomin$-definable curve. Then there exists a definable curve $b:[0,\infty)\to B$ and a constant matrix $C \in \SL(n,\R)$ such that
\[
\varphi(t) = \sigma(t) b(t) C,
\]
where $\sigma(t) \in \SL(n,\R)$ converges as $t \to \infty$. Furthermore, for sufficiently large $t$, the curve $b(t)$ satisfies one of the following:
\begin{itemize}
    \item $b(t)$ is diagonal, or
    \item the first nonzero off-diagonal entry $f_{i,j}$ of $b(t)$ satisfies:
    \begin{enumerate}
        \item $\deg f_{i,i} \ne \deg f_{i,j}$, and
        \item $\deg f_{i,j} > \deg f_{j,j}$.
    \end{enumerate}
\end{itemize}
\end{lemma}

\begin{proof}
	Using the KAN decomposition, we first write $\varphi(t)\in \SL(n,\R)$, as  $\varphi(t)=k(t)p(t)$, where $k(t)\in \SO(n,\R)$ and $p(t)\in B$ an upper-triangular matrix. In view of the Gram-Schimdt process,  $k(t)$ and $p(t)$ are definable. Since $k(t)$ is bounded and definable, $\lim_{t\to\infty}k(t)$ exists. For all $t$ large,  $p(t)$ is either diagonal or $p(t)$ takes the form of \eqref{eq:first non zero off diag}.
 
\noindent We employ the following algorithm to achieve the outcome described in the lemma. At each step, we perform either a column or a row operation on $p(t)$:
\begin{enumerate}
\item  All off-diagonal entries are eventually $0$. Then we are done. 
\item \label{enu:initialize} We pick the first index $(i,j)$ (with $i<j$) such that the $(i,j)$-th entry is not eventually $0$, and proceed to the next step.
    \item\label{enu:comapring horiz. in the alg} Suppose that $\deg(f_{i,j})=\deg(f_{i,i})$. Then there is an $c\in\R$ such that $\deg(f_{i,j}-cf_{i,i})<\deg(f_{i,i})$. For this step, subtract from the $j$-th
column the  $i$-th column multiplied by $c$.  This amounts to multiplying $p(t)$ by a constant unipotent matrix from the right. The obtained matrix is the same besides the $i,j$-th entry which is replaced with $f_{i,j}-cf_{i,i}$.  There are two possibilities now:
    \begin{enumerate}
        \item\label{enu:substracting horiz is zero} $f_{i,i}-cf_{i,j}$ is eventually zero. If all off-diagonal entries are eventually $0$, then we are done. Otherwise, the first eventually nonzero off-diagonal entry in the resulting matrix has a strictly larger index (according to the lexicographic order), and we go back to step \ref{enu:initialize}.
        \item $f_{i,i}-cf_{i,j}$ is not eventually zero: Then the first requirement of Lemma~\ref{lem:stable-unstable-constant (SUC) decomposition} is satisfied. We continue then with the following step.
    \end{enumerate}

    \item If $\deg(f_{i,j})>\deg(f_{j,j})$, then the second requirement of Lemma~\ref{lem:stable-unstable-constant (SUC) decomposition} is satisfied, and we are done, otherwise we proceed to the next step.
    
    \item Now suppose $\deg(f_{i,j})\leq \deg(f_{j,j})$. Then subtract from the $i$-th row  the $j$-th row multiplied by $\frac{f_{i,j}}{f_{j,j}}$. This amounts to multiplying from the left by a unipotent matrix, which converges as $t\to\infty$. This is so because
    $$\deg\left(\frac{f_{i,j}}{f_{j,j}}\right)=\deg(f_{i,j})-\deg(f_{j,j})\leq0.$$
    Then, either all the off-diagonal entries in the resulting matrix are eventually $0$, and we are done; or the first eventually nonzero off-diagonal entry in the resulting matrix has a strictly larger index (according to the lexicographic order), and we go back to step \ref{enu:initialize}. 
\end{enumerate}
The algorithm ends with finitely many steps with either a diagonal matrix or a definable curve $b(t)$ satisfying the requirements of the lemma.
\end{proof}

\begin{proof}[Proof for Proposition \ref{prop:r is 1 iff essentially diagonal}]
	First, note that it is enough to prove the statement for $\varphi\subseteq \SL(n,\R)$ since the P.S. group is unchanged if we multiply the curve by a constant invertible matrix from the right. In the notations of Lemma \ref{lem:stable-unstable-constant (SUC) decomposition}:
 $$\varphi(t)=\sigma(t)b(t)C,$$ where $\lim_{t\to\infty}\sigma(t)=g\in\SL(n,\R).$ Let $r\leq 1$ be such that  \eqref{eq:limit in Poulious condition} holds for $\{b(t)\}$. Let $h_{s,r}(t)$ be is as in  \eqref{eq:change of speed r<1}--\eqref{eq:change of speed r=1}. We note that
 \begin{equation} \label{eq:phi-b}
 \lim_{t\to\infty}\varphi(h_{s,r}(t))\varphi(t)^{-1}
 =g\left[\lim_{t\to\infty}b(h_{s,r}(t))b(t)^{-1}\right]g^{-1}.
 \end{equation}
 Therefore, by Proposition \ref{prop:poulios summary}, the $r$ corresponding to \eqref{eq:limit in Poulious condition} for $\varphi(t)$ and $b(t)$ are the same.
Now, $b(t)$ is either eventually diagonal or upper-triangular and satisfies the conditions of the Lemma \ref{lem:stable-unstable-constant (SUC) decomposition} for the first (eventually) nonzero off-diagonal entry. If $b(t)$ is eventually diagonal, then  for all $s>0$, $\lim_{t\to\infty}b(h_{1,s}(t))b^{-1}(t)$ converges to a diagonal matrix. Namely,  the P.S. group is diagonal in this case.
 
\noindent Suppose that $b(t)$ is not eventually diagonal. Let $(i,j)$ be the first  nonzero off-diagonal entry in $b(t)$ (for all large enough $t$). We observe that $(i,j)$-th entry in the matrix $b'(t) b(t)^{-1}$ is
	\begin{align}
		&\frac{-f_{i,i}'f_{i,j}+f_{i,j}'f_{i,i}}{f_{i,i}f_{j,j}}=\left(\frac{f_{i,j}}{f_{i,i}}\right)'\cdot \frac{f_{i,i}}{f_{j,j}}.
	\end{align}

Since $r_{i,i}:=\deg f_{i,i}\ne r_{i,j}:=\deg f_{i,j}$, we have that $$\deg\left(\frac{f_{i,j}}{f_{i,i}}\right)'={r_{i,j}-r_{i,i}-1}.$$ Thus, $$\deg\left(\left(\frac{f_{i,j}}{f_{i,i}}\right)'\cdot \frac{f_{i,i}}{f_{j,j}}\right)=r_{i,j}-r_{j,j}-1.$$ Recall by Lemma \ref{lem:stable-unstable-constant (SUC) decomposition} that $r_{i,j}-r_{j,j}>0$. Thus the $(i,j)$-th entry in  $tb'(t) b(t)^{-1}$ is unbounded.  Hence $r\neq 1$.  
\end{proof}

\subsection{Unipotent invariance in quotients}
For our purposes, it will not suffice only to establish unipotent invariance for our limiting measures, but we will also require unipotent invariance in certain quotients. We will now describe more precisely our motivation for the results in this section and how they play a role in the proof of Theorem \ref{thm: main equidistribution theorem} in Section \ref{sec:linearization}.

Consider a non-contracting curve $\varphi:[0,\infty)\to G:=\mathbf G(\R)$ definable in $\pbomin$, and consider the measures $\mu_{T,\varphi,x}$; see \eqref{eq:main definition of measures averaging on curve}. By Theorem \ref{thm:non-escape of mass}, there exists a subsequence $\mu_{T_i,\varphi,x}$ which converges to a probability measure $\mu$,  as $i\to \infty$. By Proposition \ref{prop:invariance},  $\mu$ is invariant under a nontrivial unipotent one-parameter subgroup. Let $W$ be the subgroup generated by all unipotent one-parameter subgroups that preserve $\mu$. We employ Ratner's measure classification theorem for the $W$-invariant ergodic measures and the linearization technique 
to prove that $\varphi(t)$ is contained in a certain algebraic subgroup $N$, such that $W\trianglelefteq N$ and such that $H_\varphi \subseteq N$, where $H_\varphi$ is the hull of $\varphi$. Our goal will be to prove that $H_\varphi\subseteq W$. To achieve this goal, we would like to show that if $\varphi(t)$ is not bounded modulo $W$, then the image of $\varphi$ in $N/W$, which is definable, is, in fact, non-contracting, and hence its corresponding P.S. group in $N/W$ is unipotent. The latter outcome implies  that $\mu$ is invariant under a unipotent one-parameter subgroup of $N$, which is not contained in $W$. This  contradicts our choice of $W$ as the group generated by \textit{all} unipotents preserving $\mu$.
Before proceeding with the relevant result, we have the following remark which emphasizes that the requirement that $H_\varphi\subseteq N$ is important.

\begin{remark}\label{rem:conter example}
    Consider \(\varphi(t):=\begin{bsmallmatrix}
        t&t^2\\
        0&1/t
    \end{bsmallmatrix}\) for $t\geq 1$. Then $\varphi$ is non-contracting in $\SL(2,\R)$, and its P.S. group is \(U:=\{ u(s):=\begin{bsmallmatrix}
        1&s\\
        0&1
    \end{bsmallmatrix}:s\in \R\}\). Now consider the group $B\leq \SL(2,\R)$ of upper triangular matrices. Then, $B/U$ is naturally identified with the group $A\leq\SL(2,\R)$ of diagonal matrices. In particular, $\varphi(t)$ modulo $U$ is unbounded, but it is a diagonal curve. Due to the next result, this is explained by the fact that $H_\varphi$ is not contained in $B$. 
\end{remark} 

\begin{proposition}\label{prop:Curve modulo H}
    Let $\psi:[0,\infty)\to \SL(n,\R)$ be an unbounded, non-contracting curve definable in $\pbomin$. Suppose that  $\psi([0,\infty))\subseteq H_\psi$, where $H_\psi$ is the hull of $\psi$. Let $N\subseteq \SL(n,\R)$ be a Zariski closed subgroup such that $H_\psi \subseteq N$, and let $H\trianglelefteq N$ be a closed normal subgroup.  
    
    Assume  that the image of $\psi$ in $N/H_u$ is unbounded. Let $q:N\to N/H$ be the natural quotient map. Then, there exists $r<1$ such that
    \begin{equation}\label{eq:convergence with r<1 in quotient}
    \lim_{t\to\infty}q(\psi(h_{s,r}(t)))q(\psi(t))^{-1}=p(s),\,\forall s\in \R,
    \end{equation}
    where $\{p(s):s\in\R\}$ is a nontrivial one-parameter subgroup of $N/H$. Here, $h_{s,r}(t)$ is given by
    \eqref{eq:change of speed r<1}. Importantly, there exists a one-parameter unipotent subgroup $\{u(s):s\in\R\}\subseteq N$ such that $p(s)=q(u(s))$ for all $s\in\R$.
\end{proposition}

\begin{proof}[Proof of Proposition \ref{prop:Curve modulo H}]
Let $\tilde H_u:=\zcl (H_u)$, where $H_u$ is the subgroup generated by all one-parameter unipotent subgroups contained in $H$. Since $H\trianglelefteq N$, it follows that $H_u\trianglelefteq N$, and by Zariski density $\tilde H_u \trianglelefteq N$. Now, $H_u$ is of finite index in $\tilde H_u$ (see \Cref{rem:hull is identity compo of zcl}), therefore, since the image of $\psi$ is unbounded in $N/H_u$ it follows that the image of $\psi$ is unbounded in $N/\tilde H_u$. 

We realize $N/\tilde H_u$ as a real algebraic subgroup of $\SL(m,\R)$ for some $m$. For each $t\in [0,\infty)$, let $\bar\psi(t)$ denote the image of $\psi(t)$ in $\SL(m,\R)$. We claim that $\bar\psi$ is non-contracting for $\SL(m,\R)$.  Indeed, because any finite-dimensional linear representation of $\SL(m,\R)$ is an algebraic representation of $\tilde H_\psi\subseteq N$, and $\psi$ is non-contracting for $\tilde H_\psi$, where $\tilde H_\psi$ is the Zariski closure of the hull of $\psi$, see \Cref{thm:main result on the hull and correcting curve}, (\ref{enu:fixing property after correction-Lie}).

Therefore by Lemma \ref{lem:p.s. group of non-contracting is unip}, there exists $r<1$ such that 
\begin{equation} \label{eq:pbar}
    \lim_{t\to\infty}\bar{\psi}(h_{s,r}(t))\bar{\psi}(t)^{-1}=:\bar p(s),\, \forall s\in \R,
\end{equation}
where $s\mapsto\overline p(s)$ is a nontrivial one-parameter unipotent subgroup of $\SL(m,\R)$ contained in $N/\tilde H_u$. Let $\vartheta:N\to N/H_u$ denote the quotient homomorphism. Since the quotient map $N/H_u\to N/\tilde H_u$ is a homomorphism with a finite kernel, 
for each $s\in \R$, the following limit exists:
\begin{equation}\label{eq:convergence of differences in N modulu Hu}
    \tilde p(s):=\lim_{t\to\infty}\vartheta(\psi(h_{s,r}(t)))\vartheta(\psi(t))^{-1}. 
\end{equation}

In view of Jordan decomposition, there exists a unipotent one-parameter subgroup $\{u(s):s\in\R\}\subseteq N$ such that for all $s\in\R$, the image of $u(s)$ in $N/\tilde H_u$ is $\bar p(s)$, and hence $\tilde p(s)=\vartheta(u(s))$. Since $\{\bar p(s):s\in\R\}$ is a unipotent group, we obtain that $\{u(s):s\in\R\}$ is a nontrivial unipotent one-parameter subgroup of $N$ not contained in  $H_u$, and hence it is not contained in $H$. Thus, $\{q(u(s)):s\in\R\}$ is a nontrivial subgroup of $N/H$. 

Denote $\eta:N/H_u\to N/H$ the quotient homomorphism, and note that  $\eta\circ\vartheta=q$. Then, for all $s\in \R$, $p(s):=\eta(\tilde p(s))=q(u(s))$, and by \eqref{eq:convergence of differences in N modulu Hu},  $$p(s)=\lim_{t\to\infty}\eta\circ\vartheta(\psi(h_{s,r}(t))\psi(t)^{-1})=\lim_{t\to\infty}q(\psi(h_{s,r}(t))) q(\psi(t))^{-1}.$$ 
\end{proof}

\subsection{Inclusion of the hull in a Lie subgroup}\label{sec:inclusion of hull in lie}   

\begin{proposition}\label{prop:hull in lie group}
Let $G\subseteq \SL(n,\R)$ be a connected unimodular closed subgroup, and let $\varphi:[0,\infty)\to \SL(n,\R)$ be a $\pbomin$-definable, non-contracting, continuous curve  such that $\varphi([0,\infty))\subseteq G$. Then, $H_\varphi\subseteq G_u$, and in particular, any correcting curve of $\varphi$ takes values in $G_u$.

\end{proposition}

The result does not follow from \Cref{thm:main result on the hull and correcting curve}, because we do not assume that $G$ is an observable subgroup of $\SL(n,\R)$. 

If $G$ is not algebraic, the hull $H_\varphi$ may be contained in the Zariski closure of $G$, but not in $G$. The unimodularity of $G$ ensures that $H_\varphi\subseteq G$. 

The unimodularity assumption is essential even in the algebraic subgroup case.
For example, in \Cref{rem:conter example}, the curve is non-contracting in $\SL(2,\R)$ and its image is contained in the non-unimodular Borel subgroup $B$ of $\SL(2,\R)$. But its hull is not contained in $B$.

The following is an immediate corollary of \Cref{prop:hull in lie group}.

\begin{corollary}
    Let $G_1\subseteq G_2$ be real points of real algebraic groups. Let $\varphi:[0,\infty)\to G_2$ be a $\pbomin$-definable, non-contracting, continuous curve. Suppose that $\varphi([0,\infty))\subseteq G_1$ and $G_1$ is unimodular. Then $\varphi:[0,\infty)\to G_1$ is non-contracting.
\end{corollary}

The proof will be based on Proposition \ref{prop:Curve modulo H} and the following key lemma.
\begin{lemma}\label{lem:trace of ad restricted to algebra gen. by nilpotents}
    Let $\mathfrak{g}$ be a Lie subalgebra of $\mathfrak{gl}(n,R)$. Let $\mathfrak{g}_u$ be the subalgebra generated by nilpotent matrices in $\mathfrak{g}$. Then for any $X\in\mathfrak{g}$, 
    \[
    \tr(\ad X)=\tr(\ad X|_{\mathfrak{g}_u}),
    \]
    where $\tr(A)$ denotes the trace of $A$.

    So, for the connected Lie subgroup $G$ corresponding to $\mathfrak{g}$, we have 
 \begin{equation} \label{eq:detAd}
 \det(\Ad (g)|_{\mathfrak{g}_u})=\det(\Ad(g)),\ \forall g\in G. 
 \end{equation}
\end{lemma}

\begin{proof}
    We will be using repeatedly the following fact --  If $T$ is a linear map on a vector space $V$ preserving a subspace $W$, and  $T_{V/W}$ is the factored map on $V/W$, then 
    \begin{equation} \label{eq:tr1}
    \tr(T|_V)=\tr(T|_W)+\tr(T_{V/W}),
    \end{equation}
    where $T|_W$ is the restriction of $T$ to $W$.  
    
    Let $X\in \mathfrak{g}$. Then $\ad X$ acts trivially on $\mathfrak{g}/[\mathfrak{g},\mathfrak{g}]$. Hence,
    \begin{equation} \label{eq:tr2}
    \tr(\ad X|_\mathfrak{g})=\tr(\ad X|_{[\mathfrak{g},\mathfrak{g}]}).
    \end{equation}
    Let $\tilde{\mathfrak{g}}$ denote the Lie algebra of the Zariski closure of the Lie group associated with $\mathfrak{g}$ in $\GL(n,\R)$. By \cite[Chapter II, Theorem 13]{Chevalley-LieII}, 
    \[
    [\tilde{\mathfrak{g}},\tilde{\mathfrak{g}}]=[\mathfrak{g},\mathfrak{g}]\subseteq \mathfrak{g}.
    \]
    
    Let $\mathfrak{u}$ denote the radical of the $[\tilde{\mathfrak{g}},\tilde{\mathfrak{g}}]$. Then $\mathfrak{u}$ consists of nilpotent matrices; we can derive this from~\cite[(2.5) A structure theorem]{RA72}. So, $\mathfrak{u}\subseteq \mathfrak{g}_u$.
    
    Since $[\tilde{\mathfrak{g}},\tilde{\mathfrak{g}}]$ and $\mathfrak{u}$ are ideals in $\tilde{\mathfrak{g}}$ and $[\tilde{\mathfrak{g}},\tilde{\mathfrak{g}}]/\mathfrak{u}$ is semisimple, the trace of the derivation on it corresponding to $\ad X$ is $0$. Therefore,
    \begin{equation} \label{eq:tr4}
    \tr(\ad X|_{[\tilde{\mathfrak{g}},\tilde{\mathfrak{g}}]})=\tr(\ad X|_\mathfrak{u}).
    \end{equation}
    
    Let $\mathfrak{v}$ denote the radical of $\mathfrak{g}_u$. 
    Then, $\mathfrak{u}\subseteq \mathfrak{v}$. 
    Since $\mathfrak{g}_u$ is a Lie subalgebra generated by nilpotent matrices, $\mathfrak{v}$ consists of nilpotent matrices. 
    Also, $\mathfrak{g}_u$, and hence $\mathfrak{v}$ are ideals of $\mathfrak{g}$. So, by definition of $\tilde {\mathfrak{g}}$, $\mathfrak{v}$ is an ideal of $\tilde{\mathfrak{g}}$. So, $\mathfrak{v}\subseteq \mathfrak{u}$. Hence $\mathfrak{v}=\mathfrak{u}$. 
    
    As above, since $\mathfrak{g}_u/\mathfrak{v}$ is semisimple, we conclude that
    \begin{equation} \label{eq:tr5}
    \tr(\ad X|_{\mathfrak{g}_u})=\tr(\ad X|_{\mathfrak{v}}).
    \end{equation}
   Now, \eqref{eq:tr1} follows from \eqref{eq:tr2}-\eqref{eq:tr5}.
\end{proof}

\begin{proof}[Proof of Proposition \ref{prop:hull in lie group}]
Consider the exterior product of the Adjoint representation of $\SL(n,\R)$ on $\wedge^{\dim\mathfrak{g}_u} \mathfrak{sl}(n,\R)$, and fix $p\in \wedge^{\dim\mathfrak{g}_u} \mathfrak{g}_u\setminus\{0\}$. Let
\begin{align*}
N&:=\{g\in \SL(n,\R): g\cdot p=p\}\\
&=\{g\in \SL(n,\R):\Ad(g)\mathfrak{g}_u=\mathfrak{g}_u,\, \det(\Ad(g)|_\mathfrak{g_u})=1\}
\end{align*}
So, $N$ normalizes $G_u$. Since $G$ is unimodular, by \eqref{eq:detAd}, $G\subseteq N$. 

We have $\varphi:[0,\infty)\to G$. Let $\be:[0,\infty)\to\SL(n,\R)$ be a correcting curve for $\varphi$, and denote $\psi(t):=\be(t)\varphi(t)$. 
Now $N$ is observable in $\SL(n,\R)$ (\Cref{def:observable gps}), and   $\varphi([0,\infty))\subseteq G\subseteq N$. So $H_\varphi\subseteq N$ by the minimality of the hull, Theorem \ref{thm:main result on the hull and correcting curve},\eqref{enu:hull is minimal-Lie}. Now, by Remark \ref{rem:hull of corrected curve is hull}, we have that $\psi([0,\infty))\subseteq H_\psi=H_\varphi$. 

Assume for contradiction that $\psi$ is not bounded modulo $G_u$. Let $q:N\to N/G_u$ be the quotient map. Then, by Proposition \ref{prop:Curve modulo H}, there exists a one-parameter unipotent subgroup of $\{u(s)\}_{s\in\R}\subseteq N$ such that 
\[
q(u(s))=\lim_{t\to\infty}q(\psi(h_{s,r}(t)))q(\psi(t))^{-1},\,
\forall s\in \R,
\]
and $q\circ u$ is a nontrivial unipotent one-parameter subgroup of $N/G_u$. Let $\be_\infty=\lim_{t\to\infty}\be(t)$. Then, for all $s\in\R$, $$q(\be_\infty^{-1}u(s)\be_\infty)=\lim_{t\to\infty}q(\varphi(h_{s,r}(t)))q(\varphi(t))^{-1},$$and since $\varphi([0,\infty))\subseteq G\subseteq N$, we get that $\{q(\be_\infty^{-1}u(s)\be_\infty):s\in\R\}$ is a nontrivial one-parameter subgroup of $G/G_u$. It then follows that $\{\be_\infty^{-1}u(s)\be_\infty:s\in\R\}$  is a nontrivial one-parameter unipotent subgroup contained in $G$ not contained in $G_u$, a contradiction. 
Finally, since $\varphi$ is bounded modulo $G_u$, and since $\zcl(G_u)$ is observable, we get that $H_\varphi\subseteq G_u$  by the minimality of the hull, Theorem \ref{thm:main result on the hull and correcting curve},\eqref{enu:hull is minimal-Lie}.
\end{proof}

\section{Avoidance of singular sets\label{sec:linearization}}
 Suppose that $\Ga\subseteq \SL(n,\R)$ is a discrete subgroup, and $\mu$ is a probability measure on $X:=\SL(n,\R)/\Ga$ which is invariant under a nontrivial one-parameter unipotent subgroup of $\SL(n,\R)$. Then, by the celebrated Ratner's theorem (see \cite{Ratner91a}), every $W$-ergodic component of $\mu$ is homogeneous. 
 
 The measure $\mu$ that we will consider is a weak-* limit of the measures $\mu_{T,\varphi,x}$ (see \eqref{eq:main definition of measures averaging on curve}),  for a non-contracting, definable polynomially bounded curve $\varphi$. By Theorem~\ref{thm:non-escape of mass} and Proposition~\ref{prop:invariance}, $\mu$ is a probability measure, and it is invariant under a nontrivial unipotent subgroup. 

 As an immediate consequence of Ratner's theorem (see~Theorem~\ref{Ratner's Theorem}) if $\mu$ is not $\SL(n,\R)$-invariant, then it is strictly positive on a Ratner's singular set, which is the image of a certain strictly lower dimensional algebraic subvariety of $\SL(n,\R)$ on $X=\SL(n,\R)/\Ga$. To analyze this situation, we will consider a finite-dimensional representation $V$ of $\SL(n,\R)$ and represent the singular set via an algebraic subvariety, say $\cA$, in $V$. The trajectory $t\mapsto \varphi(t)x$ defining $\mu$ will be represented by countably many simultaneous trajectories $t\mapsto\varphi(t)\gamma\cdot p$ ($\forall\gamma\in\Gamma$) for a well chosen $p\in V$ with $\Gamma\cdot p$ discrete. We will show that $\mu$ cannot be positive on the singular set unless for some $\gamma\in\Gamma$, $\varphi(t)\gamma \cdot p$ converges to a point of $\cA$ as $t\to\infty$. This leads to group-theoretic constraints. 
 
 Studying the dynamics of $\varphi(t)x$-trajectories in thin neighborhoods of a singular set in $X$ via simultaneously analyzing countably many trajectories of $\varphi$ in a thin neighborhood of an algebraic variety in a vector space is known as the {\em linearization technique}.  Showing that an unbounded trajectory of $\varphi$ avoids remaining too close to an algebraic variety for too long is called the {\em avoidance of singular sets} - this is achieved via the $(C,\alpha)$-good property. The remainder of the section is devoted to employing these techniques in our situation. 

\subsection{Description of Ratner's singular sets}

\begin{definition}\label{def:Ratner's Class}
	Let $\mathcal{H}$ be the class of all closed connected subgroups $H$ of $\SL(n,\R)$ such that  $H/H\cap\Ga$ admits an $H$-invariant probability measure and the closed subgroup $H_u\subseteq H$ which is generated by all unipotent one-parameter subgroups of $H$ acts ergodically on $H/H\cap\Ga$ concerning the $H$-invariant probability measure. 
\end{definition}

\begin{theorem} [{\cite[Theorem 1.1]{Ratner91a}}]
	The collection $\mathcal{H}$ is countable.
\end{theorem}

We note that if $\Ga=\SL(n,\Z)$, then $\zcl(H\cap\Gamma)$ is defined over $\Q$ and $H=\zcl(H\cap\Gamma)(\R)^0$  \cite[(2.3) and (3.2)]{Shah1991} for all $H\in\mathcal{H}$, so $\mathcal{H}$ is countable.

\bigskip
Let $H\in\mathcal{H}$. Define \label{definition of N and S} 
\begin{align*}
	N(H,W)=&\{g\in \SL(n,\R): W\subseteq  gHg^{-1}\},\text{ and}\\
	S(H,W)=&\bigcup_{H'\in \mathcal{H},H'\subsetneq H}N(H',W).
\end{align*}
Note that $N(W)N(H,W)N(H)=N(H,W)$, where $N(F)$ denotes the normalizer of $F$ in $\SL(n,\R)$. Also, for all $\gamma\in \Ga$, we have $\gamma H\gamma^{-1}\in\mathcal{H}$, 
\[
N(H,W)\gamma=N(\gamma^{-1} H\gamma,W) \text{ and } S(H,W)\gamma=S(\gamma^{-1} H\gamma,W).
\]
Let $\pi:\SL(n,\R)\to X=\SL(n,\R)/\Ga$ denote the quotient map. Consider the refined singular set
\begin{equation}\label{definition of tube}
T_H(W)=\pi(N(H,W)\setminus S(H,W))=\pi(N(H,W))\setminus \pi(S(H,W)). 
\end{equation}
for the second equality see~\cite[Lemma 2.4]{Mozes1995OnTS}. Also, $T_H(W)=T_{\gamma H\gamma^{-1}}(W)$ for all $\gamma\in\Gamma$.

\begin{theorem}\emph{(Ratner~\cite{Ratner91a}) }\emph{\cite[Theorem 2.2]{Mozes1995OnTS}}\label{Ratner's Theorem}
	 Let $\nu$ be a $W$-invariant Borel probability measure on $X$.  For each $H \in \mathcal{H}$,
	let $\nu_H$ denote the restriction of $\nu$ on $T_H(W)$, which is $W$-invariant. Then $\nu$ decomposes as:
		$$\nu=\sum_{H\in \mathcal H^*} \nu_H,$$
		where $\mathcal H^* \subseteq \mathcal H$ is a countable set of subgroups, each of which is a representative of a
		distinct $\Ga$-conjugacy class. 
\end{theorem}

\subsection{O-minimal curves in representations -- a dichotomy theorem}
For $H\in\mathcal{H}$, consider
\begin{equation}\label{definition of V_H}
  V_H:=\bigwedge^{\dim(H)} \mathfrak{sl}(n,\R)\text{ and } p_H\in \bigwedge^{\dim(H)} \mathfrak{h}\setminus \{0\},
\end{equation}
where $\mathfrak{h}$ denotes the Lie algebra of $H$. Consider the action of $\SL(n,\R)$ on $V_H$ induced by its Adjoint representation its Lie algebra;  that is, $g\cdot(\bigwedge_{i=1}^d x_{i}):=\bigwedge_{i=1}^d \Ad_g(x_{i})$ and extended linearly.   We denote by $\eta_H: \SL(n,\R)\to  V_{H}$  the orbit map \begin{equation}\label{eq:def of orbit map eta H}
       \eta_H(g)=g\cdot p_H,\,\forall g\in \SL(n,\R).
    \end{equation}
 The notation $\eta_H$ is used only for denoting preimages under the orbit map. We note some basic observations. 

\begin{theorem}\label{thm:DM results on A_H and discrteness of Gamma orbit in rep}
	Let $H\in \mathcal{H}$. Then:
\begin{enumerate}

    \item\label{enu:inverse of A_H}\emph{\cite[Proof of Proposition 3.2]{Dani1993LimitDO}} Let $\mathfrak w$ be the Lie algebra of $W$, and consider the linear subspace
    \begin{equation}\label{the variety A_H}
        A_H:=\{v\in V_H : v\wedge w=0,\forall w\in \mathfrak w\}.
    \end{equation}  
    Then $\eta_H^{-1}(A_H)=N(H,W)$.
    
    \item \label{enu:DM discretness Ga pH}\emph{\cite[Theorem 3.4]{Dani1993LimitDO}} 
    \[
   \text{$\Ga\cdot  p_H$ is closed, and hence discrete, in $V_H$.}
   \]   
    \end{enumerate}
\end{theorem}

For example, if $\Ga=\SL(n,\Z)$, then $\mathfrak{h}$ is defined over $\Q$, so we can pick $p_H\in \wedge^d\mathfrak{h}(\Z)\setminus\{0\}$, and hence, $\Ga\cdot p_H\subseteq V_H(\Z)$.

Next, we will provide a version of \cite[Proposition 3.4]{Mozes1995OnTS} for definable curves. Our proof is based on the linearization technique developed by Dani and Margulis~\cite[Proof of Theorem~1]{Dani1993LimitDO}. The result can be considered an analogue of Theorem~\ref{thm:kleinbock's non-divergence theorem}.  In what follows, $\pbomin$ denotes a polynomially bounded o-minimal structure.

\begin{theorem}\label{dichotomy theorem}
Let $\varphi:[0,\infty)\to \SL(n,\R)$ be a $\pbomin$-definable, continuous, unbounded, non-contracting curve. Let $\be:[0,\infty)\to\SL(n,\R)$ be a correcting curve of $\varphi$ and let $H_\varphi$ be the hull of $\varphi$ such that $\beta(t)\varphi(t)\subseteq H_\varphi$ for all $t$, see \Cref{def:hull and correcting curve}. Fix $H\in \mathcal{H}$. Let $\emph{\text{K}}_1 \subseteq A_H$ be a compact set. Then, given $x\in X$, one of the following holds\emph{:}
\begin{enumerate}
		\item\label{enu:first outcome dichotomy} There exists  $w\in \pi^{-1}(x)\cdot p_H\cap \emph{\text{K}}_1$ such that 
	\begin{equation}\label{eq:stabilizer condition}
			H_\varphi\cdot w=\{w\}.
		\end{equation}
    \item \label{enu:second outcome dichotomy} Given $0 < \e < 1$, there exists a closed set $\mathcal S\subseteq \pi(S(H,W))$ such that given a compact set $K \subseteq X\setminus \mathcal S$, there exist a compact neighborhood $\Psi$ of $\emph{\text{K}}_1$ in $ V_H$ and a $T_x>0$ such that for all $T>T_x$,
		\begin{equation}\label{relative time equation}
			|\{t\in [0,T]:\be(t)\varphi(t)x\in K\cap \pi(\eta_H^{-1}(\Psi))\}|\le \e T.
		\end{equation}
	\end{enumerate}
\end{theorem}


We now recall some observations and results that will be used in the proof. Let $N(H)$ be the normalizer of $H$ in $\SL(n,\R)$. Note that, 
\begin{equation} \label{eq:N1H}
g\cdot p_H=\det(\Ad_g|_{\mathfrak h})p_H,\,\forall g\in N(H).
\end{equation}
Set $N_{\Ga}:=N(H)\cap \Ga.$ By \cite[Lemma 3.1]{Dani1993LimitDO}, $N_{\Ga}\subseteq \{\pm p_H\}$. If $N_{\Ga}=\{\pm p_H\}$, let $\tilde{V}_H:=V_H/\{\pm 1\}$, the space $V_H$ modulo the equivalence relation defined by identifying $v\in V_H$ with $-v$. If $N_{\Ga}\cdot p_H=\{p_H\}$, let $\tilde V_H:= V_H$. 

For any $g\in \SL(n,\R)$ and $y=\pi(g)\in X$,  we treat $\pi^{-1}(y)\cdot \tilde p_H=g\Gamma\cdot \tilde p_H$ as the discrete \emph{set of representatives} of $y$ in $\tilde V_H$, where $\tilde p_H $ denotes the image of $p_H$ in $\tilde V_H$. We need the following result on the {\em uniqueness of representatives\/} in thin neighborhoods. 

\begin{proposition}\emph{\cite[Proposition 3.2]{Mozes1995OnTS}}\label{prop: mozes shah statemnt for uniquness of representatives}
	Let $D$ be a compact subset of $A_H$, and consider 
	\begin{equation*}
		S(D)=\{g\in \eta_H^{-1}(D):g\ga \in  \eta_H^{-1}(D) \text{ for some } \ga\in \Ga\setminus N_{\Ga}\}\subseteq \SL(n,\R).
	\end{equation*}
	Then:
	\begin{enumerate}
\item $S(D)\subseteq S(H,W)$.
\item $\pi(S(D))$ is closed in $X$.
\item\label{enu:neighb of inject} Suppose that $K \subseteq X \setminus \pi(S(D))$ is compact. Then, there is a compact neighborhood
$\Phi$ of $D$ in $V_H$  such that any $y\in K$ has at most one representative in $\tilde\Phi$, where $\tilde {{\Phi}}$ is the image of ${\Phi}$ in $\tilde V_H$.

That is, for all $y\in  \pi( \eta_H^{-1}({\Phi})) \cap K$,  the set $ \pi^{-1}(y).\tilde p_H \cap \tilde{{\Phi}}$ consists of a single element. 
	\end{enumerate}
\end{proposition}

\begin{proof}[Proof of Theorem \ref{dichotomy theorem}]
We assume for simplicity that $N_{\Ga}\cdot p_H=\{p_H\}$, namely $\tilde V_H=V_H$. The case for which $N_{\Ga}\cdot p_H=\{\pm p_H\}$ can be treated in essentially the same way, and is left to the readers.  

By choosing a basis for $V_H$,  we identify  $V_H\cong\R^m$ for some $m$. We will
denote by $\psi(t)\in \GL(m,\R)$ the action of $\be(t)\varphi(t)$ on $V_H$. We will apply Proposition \ref{prop:relative time property} for $\psi(t)$. Since $A_H$ is a linear subspace, $A_H=\{Q=0\}$ for a polynomial map $Q:V_H\to\R$, see  Remark~\ref{rem:one polynomial determines the variety}. Let $K_1\subseteq A_H$ and let $0<\e<1$ be given. Let $\text K_2\subseteq A_H$ be a compact set that is given by Proposition \ref{prop:relative time property} with $\text K_1 \subseteq \text K_2$. 

\subsubsection*{Choices of $\mathcal{S}$ and $\Phi$ for injective property}
For applying Proposition \ref{prop: mozes shah statemnt for uniquness of representatives}, we use $\text{K}_2$  in place of $D$, and let $\mathcal{S}:=\pi(S(\text{K}_2))$. Let $K\subseteq X\setminus \mathcal S$ be a compact set, and let $\Phi\subseteq V_H$ be a compact neighborhood of $\text{K}_2$ which satisfies the outcome of Proposition \ref{prop: mozes shah statemnt for uniquness of representatives},\eqref{enu:neighb of inject}. As a consequence, for any $t\geq 0$, $w_1,w_2\in\Ga\cdot p_H$, and $g\in \SL(n,\R)$, if $g\Ga\in K$ and $g\cdot w_i\in \Phi$ for $i=1,2$, then $w_1=w_2$; let us call this the {\em injective property of $K$ in $\Phi$}. 

\subsubsection*{Choice of $\Psi$ using $(C,\alpha)$-good property}
We take $T'_0>0$ to satisfy the outcomes of Propositions \ref{prop:c alpha good ppty of curves and action in reps} and \ref{prop:relative time property} for $\psi$. By Proposition~\ref{prop:relative time property}, we choose a compact neighborhood $\Psi$ of $\text{K}_1$ in $V_H$ such that $\Psi\subseteq \mathring{\Phi}$ and for any $v\in V$ and any closed interval $[a,b]\subseteq [T_0',\infty)$, 
if $\psi([a,b])v\subseteq \Phi$ and $\psi(b)v\notin \mathring{\Phi}$, then
\begin{equation} \label{eq:relative-time}
    \abs{\{s\in [a,b]:\psi(s)v\in\Psi\}}\leq \epsilon (b-a).
\end{equation}

It may be noted that in view of \eqref{eq:R2e1Phi} and \eqref{eq:defPsi}, there exist $R_2>0$ and $\epsilon_1>0$ such that 
\begin{equation} \label{eq:Psi-Phi}
\Psi\subseteq \{v\in V_H:\norm{v}<R_2 \text{ and }\abs{Q(v)}<\epsilon_1\}\subseteq \mathring{\Phi}.
\end{equation}

\subsubsection*{Modifying $\Psi$ for dichotomy}
Let $x\in X$ be given. Since $\pi^{-1}(x)\cdot p_H\cap\Psi$ is finite, by  replacing $\Psi$ with a smaller compact neighborhood of $\text{K}_1$, we may assume that $\pi^{-1}(x)\cdot p_H\cap\Psi=\pi^{-1}(x)\cdot p_H\cap K_1$.

By Theorem \ref{thm:main result on the hull and correcting curve},\eqref{enu:fixing property after correction-Lie}, the following two mutually exclusive cases exhaust all possibilities. 

\begin{itemize}
\item[Case~1] -- \emph{There exists $w\in\pi^{-1}(x)\cdot p_H\cap \Psi$  such that $H_\varphi\cdot w=\{w\}$}.

\item[Case~2] -- \emph{For all $w\in \pi^{-1}(x)\cdot p_H\cap \Psi$, $\lim_{s\to\infty}\psi(s)w=\infty$.}
\end{itemize}
If Case~1 holds, then by our choice of $\Psi$, the first conclusion of the theorem holds. Therefore, to complete the proof, we assume that Case~2 holds. 

Consider
\begin{equation} \label{eq:defJ}
	J=\{t\ge T'_0: \be(t)\varphi(t)x\in \pi( \eta^{-1}_H(\Psi))\cap K \}.
\end{equation}
Then $J$ is closed. Because $\pi( \eta^{-1}_H(\Psi))$ is closed in $X$. To show this, suppose $g_i\to g$ in $\SL(n,\R)$ and $\gamma_i\in\Gamma$ are such that $g_i\gamma_i\cdot p_H\in \Psi$. Since $\Psi$ is compact and $\Gamma\cdot p_H$ is discrete, we can pick $\gamma\in\Gamma$ such that $\gamma_i\cdot p_H=\gamma\cdot p_H$ for infinitely many $i$'s. Therefore $g\gamma\cdot p_H\in\Psi$. 

Then, to prove \eqref{relative time equation}, it is enough to show that there exists $T_x\geq T_0'$ depending on $x$ such that 
\begin{equation} \label{eq:relative time equation-2}
\abs{J\cap [T_0',T]}\leq 2\epsilon T, \ \forall\, T\geq T_x.
\end{equation}
Without loss of generality, we may assume that $J$ is unbounded.

\subsubsection*{Choice of a representative of $x$ in for every $t\in J$}

For every $t\in J$, there exists a representative of $x$, say $w\in \pi^{-1}(x)\cdot p_H$, such that $\psi(t) w\in \Psi \subseteq\Phi$. Since $\beta(t)\varphi(t)x\in K$, by the injective property of $K$ in $\Phi$, $w$ is unique, and we denote it by $w_{t}$.

Let $t\in J$. We note that  
\begin{equation*} \label{eq:unbounded}
\lim_{s\to\infty}\psi(s)w_{t}=\infty,
\end{equation*}
because otherwise $w_{t}$ is fixed by $H_\varphi$ due to (\ref{enu:fixing property after correction-Lie}) of Theorem \ref{thm:main result on the hull and correcting curve}, and since $\psi(t)\subseteq H_\phi$, we have $w_{t}\in \Psi\cap \pi^{-1}(x)\cdot p_H$, contradicting the Case~2.

Therefore, $\{s\geq T_0':\psi(s) w_{t}\in \mathring{\Phi}\}$ is a bounded open subset of $[T_0',\infty)$ containing $t$. Let $U(t)$ be its connected component containing $t$, and denote $$t^+:=\sup U(t)<\infty.$$ Then 
\begin{equation} \label{eq:inPhi}
\psi([t,t^+))w_{t}\subseteq\mathring{\Phi} \text{ and } \psi(t^+)w_{t}\in \Phi\setminus \mathring{\Phi}.
\end{equation}

\subsubsection*{Bound on $|J\cap [t,t^+]|$}
For any $s\in J\cap [t,t^+]$, we have  
$\beta(s)\varphi(s)x\in K$, 
$\beta(s)\varphi(s)\cdot w_{s}\in \Psi\subseteq \mathring{\Phi}$, and $\beta(s)\varphi(s)\cdot w_{t}\in \Phi$, 
and hence by  the injectivity of $K$ in $\Phi$, $w_{s}=w_{t}$. Then,
\begin{equation}
\begin{array}{ll}
J\cap [t,t^+]&=\{s\in [t,t^+]\cap J: \psi(s) w_{s}\in\Psi\} \\
&=\{s\in [t,t^+]\cap J: \psi(s) w_{t}\in\Psi\}.  \label{eq:Omega-Psi}
\end{array}
\end{equation}

By \eqref{eq:relative-time}, \eqref{eq:inPhi},  and \eqref{eq:Omega-Psi}, 
\begin{equation} \label{eq:J-relative}
   \abs{J\cap [t,t^+]}\leq\left|\{s\in [t,t^+]:\psi(s)w_{t}\in \Psi\}\right|\leq \e (t^+-t). 
 \end{equation}

\subsubsection*{Choosing a sequence $t_i\in J$}
For each $i\in\N$, we define $t_i\in J$ inductively as follows: Let $t_1=\inf J\in J$. Let $i\in\N$ be such that $t_i$ is defined. Then we let $t_{i+1}=\inf J\cap (t_{i}^+,\infty)$. Since $J$ is closed, $t_i\in J$.

By our construction, $t_i<t_i^+<t_{i+1}$ for all $i\in\N$. Let $i\in\N$. Since $$J\cap [T_0',t_i^+]=\cup_{j=1}^i J\cap [t_j,t_{j}^+),$$
by \eqref{eq:J-relative},
\begin{equation} \label{eq:uptoi}
\abs{J\cap [T_0',t_i^+]}\leq \epsilon (t_i^+-T_0').
\end{equation}

We will bound $|J\cap[t_{i+1},T]|$ using \Cref{prop:c alpha good ppty of curves and action in reps}. We need the following:

\subsubsection*{Claim.} As $i\to\infty$, $\norm{w_{t_i}}\to\infty$ and $t_i\to\infty$. 

\medskip

For $w\in \pi^{-1}(x)\cdot p_H$, the maps $s\mapsto \norm{\psi(s)w}$ and $s\mapsto Q(\psi(s)w)$ are definable and for $s\geq T_0'$. Let
\[
S_w:=\{s\geq T_0':\norm{\psi(s)w}=R_2 \text{ or } \abs{Q(\psi(s)w)}=\epsilon_1\}.
\]
Since $S_w$  is a definable subset of $\R$, it has only finitely many connected components, see Theorem \ref{thm:uniform bound on fibers}.
For any $i\in\N$, if $w=w_{t_i}$, then  $(t_i,t_i^+)$ contains at least one connected component of $S_w$;
indeed, since $\psi(t_i)w_{t_i}\in \Psi$ and  $\psi(t_i^+)w_{t_i}\not\in\mathring{\Phi}$, by \eqref{eq:Psi-Phi}, $t_i,t_i^+\not\in S_w$ and $(t_i,t_i^+)\cap S_w\neq \emptyset$. 
Therefore, since $\{t_i\}_{i=1}^\infty$ is an infinite sequence, 
the set  $\{i\in\N:w=w_{t_i}\}$ is finite. 

So, $\{w_{t_i}:i\in\N\}$ is infinite. Since $\pi^{-1}(x)\cdot p_H$ is discrete in $V_H$, we get $\lim_{i\to\infty}\|w_{t_i}\|=\infty$. Since $\psi(t_i)w_{t_i}\in \Psi$ for all $i$ and the map $s\mapsto \norm{\psi(s)^{-1}}$ is continuous on $[T_0',\infty)$, we conclude that $t_i\to\infty$. This proves the claim.

By Proposition~\ref{prop:c alpha good ppty of curves and action in reps}, and recalling \eqref{eq:C-alpha-good}, given $\e>0$ we can pick $M_0>0$ such for any $w\in V_H$, if $\norm{\psi(T'_0)w}\geq M_0$, then
 \begin{equation} \label{eq:KM}
 \left|\{s\in [T_0',T]:\psi(s)w\in \Psi\}\right|\leq \e (T-T_0').
 \end{equation}

Due to the above claim, let $i_x\in\N$ be such that $\norm{w_{t_{i}}}\geq \norm{\psi(T_0')^{-1}}\cdot M_0$ for all $i\geq i_x$. Set $T_x=t_{i_x}$. Let $T\geq T_x$. Since $t_j\to\infty$ as $j\to\infty$, 
$i:=\max\{j:t_j\leq T\}<\infty$. Then $i\geq i_x$, so $\norm{\psi(T'_0)w_{t_{i}}}\geq M_0$. By \eqref{eq:KM}, 
\begin{equation} \label{eq:Tlast}
    \abs{J\cap [t_{i},T]} \leq \abs{\{s\in [t_{i},T]:\psi(s)w_{t_{i}}\in \Psi\}}\leq \e(T-T_0'). 
 \end{equation}
By combining \eqref{eq:uptoi} and \eqref{eq:Tlast}, we obtain \eqref{eq:relative time equation-2}.
\end{proof}

\subsection{Proof of Theorem \ref{thm: main equidistribution theorem} }
Let $G:=\mathbf G(\R)$. Suppose that $\G\subseteq G$ is a connected, closed Lie subgroup with a lattice $\Ga\subseteq \G$. Let $\varphi:[0,\infty)\to G$ be a $\pbomin$-definable, continuous, unbounded, non-contracting curve for $G$ such  that $\varphi\left([0,\infty)\right)\subseteq \G$. Let $H_\varphi$ be the hull of $\varphi$ in $G$ and let $\be$ be a definable correcting curve in $G$, so that $\be\varphi\subseteq H_\varphi$, see Definition \ref{def:hull and correcting curve}. Then, by Proposition $\ref{prop:hull in lie group}$, $H_\varphi\subseteq \G$ and $\beta([0,\infty))\subseteq \G$.

We choose an injective algebraic map $\iota:G\hookrightarrow \SL(n,\R)$ for some $n\geq 2$. Note that $\iota(\Ga)$ is a discrete subgroup of $\SL(n,\R)$.  Importantly, $\iota\circ\varphi$ is non-contracting in $\SL(n,\R)$ by Lemma~\ref{lem:image of non contracting is non contracting}, and $\iota(H_\varphi)=H_{\iota\circ\varphi}$ by Lemma~\ref{lem:image of hull under homo}. 
By abuse of notations, in the following we replace $\varphi$ with $\iota\circ \varphi$, replace $\Ga$ with $\iota(\Ga)$ and replace $G$ with $\iota(G)$. 

Let $x_0\in \G/\Ga\hookrightarrow \SL(n,\R)/\Ga$, and fix a limiting measure $\mu$  as $T\to\infty$ of $\mu_{T,\be\varphi,x_0}$  (defined in \eqref{eq:main definition of measures averaging on curve}), with respect to the weak-$\ast$ topology on the space of probability measures on $\SL(n,\R)/\Ga$. We fix  a sequence $\{T_i\}$ such that $T_i\to\infty$ and\begin{equation}\label{eq:limiting seq for mu}
    \text{weak*-}\lim_{i\to\infty}\mu_{T_i,\be\varphi,x_0}=\mu.
\end{equation}
By Theorem~\ref{thm:non-escape of mass}, $\mu$ is a probability measure. 

Let $W$ be the subgroup of $\SL(n,\R)$ generated by all unipotent one-parameter subgroups under which the limiting measure $\mu$ is invariant. By Proposition \ref{prop:invariance}, $\dim(W)>0$. 
Using Theorem \ref{Ratner's Theorem},  there exists $H\in \mathcal{H}$ such that 
\begin{align}\label{eq:tube of lowest dimension with positive measure}
	\mu(\pi(N(H,W)))>0\text{ and }\mu(\pi(S(H,W)))=0.
\end{align}
Here $H\in\mathcal H$ is of smallest dimension such that $\mu(\pi(N(H,W)))>0$.

In view of \eqref{eq:N1H}, we define
\begin{align} \label{definition of $N^1(H)$}
	N^1(H)&:=\eta_H^{-1}(p_H)=\{g\in \SL(n,\R): g\cdot p_H=p_H\}\\
    &=\{g\in N(H):\det((\Ad g)|_\mathfrak{h})=1\}, \notag
\end{align}
where $p_H$ and $\eta_H$ are as in \eqref{eq:def of orbit map eta H}. 
Since $H\cap\Gamma$ is a lattice in $H$, $H$ is unimodular (see \cite[Remark 1.9]{RA72}). So, $H\subseteq N^1(H)$. 

The remainder of the proof proceeds as follows. In \Cref{lem:correcting phi into N1(H)}, we show that there exists an element $g_0 \in \SL(n,\R)$ such that $g_0 \Gamma = x_0$ and 
\[
\{\be(t)\varphi(t)g_0\}_{t \geq 0} \subseteq g_0 N^1(H).
\]
We first note that $\Gamma N^1(H)=\eta_H^{-1}(\Gamma\cdot p_H)$ is closed, as $\Gamma \cdot p_H$ is discrete in $V_H$ (see \Cref{thm:DM results on A_H and discrteness of Gamma orbit in rep}, \eqref{enu:DM discretness Ga pH}). Consequently, the quotient $N^1(H)\Gamma / \Gamma \subseteq X$ is also closed.

By \eqref{eq:limiting seq for mu}, it follows that $\mu$ is supported on $g_0 N^1(H) x_0$. One may view $g_0 N^1(H) \Gamma / \Gamma$ as foliated by orbits of $g_0 H g_0^{-1}$, and we will show that $\mu$ is invariant under $g_0 H g_0^{-1}$. The key step is, to prove that $\mu$ is actually supported on the single orbit of $g_0 H g_0^{-1}$ passing through $x_0$, a fact established in \Cref{cor:key cor establishing main thm}. This will then immediately imply our main result, \Cref{thm: main equidistribution theorem}, as it shows that
\[
\overline{H_\varphi x_0} = g_0 H g_0^{-1} x_0,
\]
and that $\be \varphi$ is equidistributed on this orbit.

\begin{lemma}\label{lem:correcting phi into N1(H)}
  Fix $H\in \mathcal{H}$ such that \eqref{eq:tube of lowest dimension with positive measure} holds. Then, there exists  $g_0\in \SL(n,\R)$ such that $x_0=\pi(g_0):=g_0\Ga$ and \begin{equation}\label{eq:the corrected curve psi in N1(H)}
           \left\{g_0^{-1}\be(t)\varphi(t)g_0\right\}_{t\geq 0}\subseteq N^1(H).
       \end{equation}Importantly, \begin{equation}\label{eq:translate of mu}
           \nu:=(g_0^{-1})_*\mu,
       \end{equation}is supported on $\pi(N^1(H))$, $\nu$ is $  H$-invariant, and the largest group generated by unipotent one-parameter subgroups preserving $\nu$ is $H_u$. 
\end{lemma}

\begin{proof}

In view of \eqref{definition of tube}, due to \eqref{eq:tube of lowest dimension with positive measure}, there is a compact subset  $C \subseteq N(H, W)\setminus S(H, W)$ such that $\mu(\pi(C))=\al$ for some $\al>0$, and $\pi(C)\cap \pi(S(H,W))= \varnothing$.  Consider the compact subset $\text K_1\subseteq A_H$ given by  $\text{K}_1:=C\cdot p_H$. We will now apply Theorem \ref{dichotomy theorem} to $x_0\in X$. 

We  show  that outcome \eqref{enu:second outcome dichotomy} of Theorem \ref{dichotomy theorem} cannot hold. Let $\epsilon:=\frac{\al}{2}$.  Denote by $\mathcal S \subseteq \pi(S(H,W))$  the closed subset given by outcome \eqref{enu:second outcome dichotomy} of Theorem \ref{dichotomy theorem}. Let $K\subseteq X\setminus \mathcal S$ be an open neighborhood of $\pi(C)$, which exists as $\pi(C)$ is compact and $\mathcal S$ is closed. Also, fix an arbitrary compact neighborhood $U$  of $\text{K}_1$ in $V_H$. Then, $K\cap\pi(\eta_H^{-1}(U))$ is an open neighborhood of $\pi(C)$. By \eqref{eq:limiting seq for mu},  for all $i$ large enough, $$|\{t\in [0,T_i]:\be(t)\varphi(t) x_0\in K\cap \pi( \eta_H^{-1}(U))\}|\geq \e T_i,$$ which contradicts outcome  \eqref{enu:second outcome dichotomy} of Theorem \ref{dichotomy theorem}.  

Therefore, the first outcome \eqref{enu:first outcome dichotomy} of Theorem \ref{dichotomy theorem} must hold. That is, we can pick $g_0\in\pi^{-1}(x_0)$ such that $g_0\cdot p_H\in \text{K}_1\subseteq A_H$ and such that:
\begin{equation}\label{eq:consequence of first outcome of dichotomy}
\be(t)\varphi(t)g_0\cdot p_H=g_0\cdot p_H, \text{ for }t\geq 0. 
\end{equation}
Namely, $\be(t)\varphi(t)g_0\in g_0N^1(H)$ for all $t\geq 0$. Since $\pi(N^1(H))$ is closed, we obtain that $\mu$ is supported on a subset of $\pi(g_0N^1(H))$. In fact, we claim that $\mu(\pi(g_0N^1(H)))=1$


To see this, observe that since $g_0\cdot p_H \in A_H$, then $g_0\in N(H,W)$ by Theorem \ref{thm:DM results on A_H and discrteness of Gamma orbit in rep},\eqref{enu:inverse of A_H}. Moreover,  $g_0N^1(H)\subseteq N(H,W)$, since $N(H,W)$ is right-$N^1(H)$ invariant.

In particular, we conclude that $\mu(\pi(g_0N^1(H)\setminus S(H,W)))=1$ (see \eqref{eq:tube of lowest dimension with positive measure}). Therefore, by \cite[Lemma 2.4]{Mozes1995OnTS},  almost every $W$-ergodic component of $\mu$ is invariant under $gHg^{-1}=g_0Hg_0^{-1}$ for some $g\in g_0N^1(H)\setminus S(H,W)$. Therefore $\mu$ is $g_0Hg_0^{-1}$-invariant. Then, $\nu=(g_0^{-1})_*\mu$ is $H$-invariant, and especially,  $H_u$ preserves $\nu$.

Finally, we claim that $H_u=g_0^{-1}Wg_0\subseteq H$. To see this equality, recall that  $g_0\in N(H,W)$, so that $g_0^{-1}Wg_0\subseteq H$, and recall that $W$ is chosen to be the subgroup of $\SL(n,\R)$ generated by all unipotent one-parameter subgroups preserving $\mu$.
\end{proof}

We will use the following notations: $\psi(t):=g_0^{-1}\be(t)\varphi(t)g_0$, for $t\geq 0$, $N:=N^1(H)$, $\De:=N^1(H)\cap\Ga$ and $q:N\to N/H$ denotes the natural quotient map.

Consider $\bar N:=q(N)$, $\bar \De:=q(\De)$. Since $H\De=\De H$ is closed, we have that $\bar \De\leq \bar N$ is discrete (since $\De H$ is closed, then $q(\De)=\De H/H\cong \De/\De\cap H$). We let $\bar q:N/\De \to \bar N/ \bar \De$ be the natural map.
Notice that the fibers of $\bar q$ are  $H$-orbits. Namely, $$\bar q^{-1}(\bar n \bar \De)=nH\De/\De=Hn\De/\De.$$
  Since $\nu$ is a probability measure on $N/\De\cong\pi(N^1(H))$, it follows that $\bar \nu:=\bar q_*\nu$  is a probability measure on $\bar N/ \bar \De$. 
\begin{lemma}\label{lem:invariance of the quotient measure by unip}
     Suppose that $\psi(t)$ is unbounded in $N/H_u$. Then there exists a unipotent one-parameter subgroup $U \leq N$ such that $U$ is not contained in $H$, the measure $\bar{\nu}$ is invariant under $q(U)$, and $q(U)$ is a nontrivial one-parameter subgroup. 

     Moreover, $\nu$ is $U$-invartiant. 
\end{lemma}
\begin{proof}
   We note that the measure push-forward map $q_*$ is continuous, and in particular,$$\bar\nu(f)=\lim_{i\to\infty}\frac{1}{T_i}\int_0^{T_i}f(q(\psi(t))\bar\De) dt,\forall f\in C_c(\bar N/ \bar \De).$$Thus, we may conclude the statement by applying Proposition \ref{prop:Curve modulo H}, and Lemma \ref{difference of intergal lemma} (see the proof of Lemma \ref{lem:invariance for unipotent p.s.}. The only difference here is that $ N/ H$ is not a subgroup of $\GL(n,\R)$).

By the ergodic decomposition for the $H$-action on $N/\De$ (see e.g. \cite[Theorem 8.20]{EW11}), for all $f\in C_c(N/\De)$,
\begin{align}\label{eq:disintegration of mu}
	\int_{N/\De}f(n\De)d\nu(n\De)=
\int_{\bar N/\bar \De}\int_{H\De/\De}f(nh\De)d\mu_{H\De/\De}(h\De)d\bar \nu(q(n)\bar \De),
		\end{align}
 where $\mu_{H\De/\De}$ denotes the   $H$-invariant probability measure on $H\De/\De$.

Therefore, since $\bar\nu$ is $q(U)$-invariant,  $\nu$ is $U$-invariant.
\end{proof}

\begin{corollary}\label{cor:key cor establishing main thm}
    It holds that $\psi(t)\in H_u$ for all $t\geq 0$, and $\nu$ is the $H$-invariant probability on $H\De/\De\cong H\Ga/\Ga$. Moreover, $H$ is the smallest subgroup such that $ g_0^{-1}H_\varphi g_0\subseteq H$ and that the orbit $H\De/\De$ is closed.
\end{corollary}
\begin{proof}
     We first show that $\psi(t)$ is bounded in $N/H_u$.  Suppose $\psi(t)$ is unbounded in $N/H_u$. Then by Lemma  \ref{lem:invariance of the quotient measure by unip}, $\nu$ is invariant by the group generated by $U$ and $H$, which strictly includes $H_u$. This is a contradiction, because $H_u$ is the maximal subgroup generated by unipotent one-parameter subgroups preserving $\nu$. 
     
    Now, since the image of $\psi(t)=g_0^{-1}\be(t)\varphi(t)g_0$ is boundned in $\SL(n,\R)/H_u$, we get that the image of $\be(t)\varphi(t)$ is bounded in $\SL(n,\R)/g_0H_ug_0^{-1}$. Thus,  since $\zcl(H_u)$ is observable (see Lemma \ref{lem:generated by uni implies observ}), by minimality of the hull (see Theorem~\ref{thm:main result on the hull and correcting curve}),  $H_\varphi\subseteq g_0H_ug_0^{-1}$, and in particular $\be(t)\varphi(t)\in g_0H_ug_0^{-1}$ for all $t$. Namely,  $\psi(t)\in g_0^{-1}H_\varphi g_0\subseteq H_u$ for all $t\geq 0$. In particular, $\bar \nu$ is the Dirac measure at the identity coset.  Therefore $\nu=\mu_{H\De/\De}$ by \eqref{eq:disintegration of mu}. In particular $\{\psi(t)\De\}_{t\geq0}$ is dense in $H\De/\De$. In particular $g_0^{-1}H_\varphi g_0\De$ is dense in $H\De/\De$, which finishes the proof.
     
\end{proof}

Theorem \ref{thm: main equidistribution theorem} follows directly from the above corollary.

\section{Curves with a largest possible hull\label{sec:curves with a maximal hull}} \label{sec:7}

The main goal of this section is to prove \Cref{prop:main example for equidistribution} on the equidistribution of the curve \eqref{eq:main example for equidistributing curve}. As an overview, the key idea is that the curve \eqref{eq:main example for equidistributing curve} will be strongly non-contracting, in the sense that it takes nonzero wedges to $\infty$, Lemma \ref{lemma:phi-expands}. The following result, of independent interest, implies the hull of \eqref{eq:main example for equidistributing curve} is semi-simple and acts irreducibly on $\R^n$. Finally, a key classification of \cite[Theorem A.7]{shah2023equidistribution} will allow to prove that the hull equals to the full group $\SL(n,\R)$.

\begin{lemma}\label{lem:hull is semisimple if curve expands}
 Suppose that $\varphi:[0,\infty)\to \SL(n,\R)$ is a continuous curve definable in an  o-minimal structure such that $\varphi$ has the following divergence property:
 \begin{equation}\label{eq:expanding property}
     \lim_{t\to\infty}\varphi(t) \cdot (v_1\wedge \dots \wedge v_k)=\infty,
 \end{equation}for all $1\leq k\leq n-1$ and all linearly independent vectors $v_1,\dots,v_k\in\R^n$. Then, the  hull $H_\varphi$ of $\varphi$ is a connected semi-simple with no compact factors, and the action of $H_\varphi$ on $\R^n$ is irreducible.
\end{lemma}

\begin{proof}
Since $\zcl(H_\varphi)$ is observable, we have  a finite-dimensional rational representation $V$   of $\SL(n,\R)$ and  $v\in V$ such that $\zcl(H_\varphi)$ is the isotropy group of $v$. Note that $\varphi(t)\cdot v,t\geq 0$ is bounded since $\varphi$ is bounded in $\SL(n,\R)/H_\varphi$. Since we assume \eqref{eq:expanding property}, it follows by Lemma \ref{lem:expanding implies orbit closed}  that the orbit $\SL(n,\R)\cdot v$ is Zariski closed. Then, Matsushima criterion \cite{Matsushima_crit} implies that $\zcl(H_\varphi)$ is reductive. Since $H_\varphi$ is connected and generated by one-parameter unipotent subgroups, it follows that  $H_\varphi$ is a connected semi-simple group with no compact factors. Finally, we prove that $H_\varphi$ acts on $\R^n$ irreducibly. Suppose for contradiction that $\R^n=V_1\oplus V_2$ where both $V_1$ and $V_2$ are nontrivial sub-spaces invariant under $H_\varphi$. Let $v_1,\dots,v_k$ be a basis for $V_1$. Here $1\leq k<n$ because $V_2$ is not trivial.   Since semi-simple groups have no nontrivial characters, we obtain that $H_\varphi\cdot(v_1\wedge\cdots\wedge v_k)=v_1\wedge\cdots\wedge v_k$. Since $\varphi([0,\infty))$ is bounded in $\SL(n,\R)/H_\varphi$, we get $\varphi(t)\cdot (v_1\wedge\cdots\wedge v_k)$ is bounded as $t\to\infty$. This contradicts \eqref{eq:expanding property}.
\end{proof}

Let $\varphi$ be the curve given by \eqref{eq:main example for equidistributing curve}. 
We now verify that $\varphi$ satisfies a divergence property, which is stronger than \eqref{eq:expanding property}. 

\begin{lemma}  \label{lemma:phi-expands}
    For any $1\leq k\leq n$, and $0\neq v\in \bigwedge^k\R^{n+1}$, $\lim_{t\to\infty} \varphi(t)\cdot v\to\infty$ as $t\to\infty$.
\end{lemma}

\begin{proof}
We denote by $e_0,e_1,\dots,e_n$ the canonical basis vectors of $\R^{n+1}$, where $e_i$ denotes the vectors whose coordinates equal $0$ besides the $(i+1)$th position in which $1$ appears. Fix an integer $1\leq k\leq n$ and let $\mathcal{I}$ be the collection of all subsets of $\{0,\ldots,n\}$ of cardinality $k$. Let $\varphi$ be the curve given by \eqref{eq:main example for equidistributing curve}. For any $I:=\{i_1,\ldots,i_k\}\in \mathcal I$, where $i_1<i_2<\cdots<i_k$, let $e_I=e_{i_1}\wedge \cdots\wedge e_{i_k}$. There are two cases for $\varphi(t)\cdot e_I$: 

If $i_1=0$, then 
 \begin{align}\varphi(t)\cdot e_{I}
 =&f_0(t)e_0\wedge[f_{i_2}(t)e_0+h_{i_2}(t)^{-1}e_{i_2}]\wedge\cdots\wedge[f_{i_{k}}(t)e_0+h_{i_{k}}(t)^{-1}e_{i_k}]\nonumber\\
 =&f_0(t)\left(\prod_{l=2}^k h_{i_l}(t)^{-1}\right)e_I;\label{eq:action of varphi canonical exterior vecs with e0}
\end{align}
and if $i_1\geq 1$, then 
\begin{align}
          \varphi(t)\cdot e_{I}
          =&[f_{i_1}(t)e_0+h_{i_1}(t)^{-1}e_{i_1}]\wedge\cdots\wedge[f_{i_k}(t)e_0+h_{i_k}(t)^{-1}e_{i_k}]\nonumber\\
          =&\sum_{l=1}^k (-1)^{l-1} f_{i_l}(t)\left(\prod_{j\in\{1,\ldots,k\}\setminus \{l\}}h_{i_j}(t)^{-1}\right)e_{\{0\}\cup I\setminus\{i_l\}} \nonumber\\
          &+\prod_{j\in\{1,\ldots,k\}} h_{i_j}(t)^{-1}e_I.
          \label{eq:action of varphi on exteriors not including e0}
    \end{align}
Let $\mathbf{v}\in\bigwedge^k \R^{n+1}\setminus\{0\}$. Then $\mathbf{v}=\sum_{I\in\mathcal I} c_I e_I$, 
where $c_I\in\R$. Then, \begin{equation} \label{eq:CI}
\varphi(t)\cdot\mathbf{v}=\sum_{I\in \mathcal I} C_I(t)e_I, 
\end{equation}
where $C_I(t)\in \R$. The coefficients of our interest are $C_I(t)$ such that $0\in I$. For any $1\leq i_2<\cdots<i_k\leq n$, and $J=\{i_2,\ldots,i_k\}$,
\begin{align} \label{eq:C*}
C_{\{0\}\cup J}(t)=
    \left(\prod_{j\in J} h_j(t)^{-1}\right)\sum_{l\in\{0,1,\ldots,k\}\setminus J} f_l(t)c_{J\cup\{l\}}(-1)^{j_l},
\end{align}
where $j_l=\min\{j\in \{2,\ldots,k\}:i_{j}>l\}$. 

Let $\mathcal{I}_{\mathbf v}=\{I\in \mathcal{I}:c_I\neq 0\}\neq\varnothing$, as $\mathbf{v}\neq 0$. Let $i_1=\min\cup\mathcal{I}_\mathbf{v}$. We pick $I\subseteq \mathcal{I}_\mathbf{v}$ such that $i_i\in I$. Let $J=I\setminus\{i_1\}$. Then, $c_{J\cup\{l\}}=0$ for every $l<i_1$. Therefore, by \eqref{eq:C*}
\begin{equation} \label{eq:C_J}
C_{\{0\}\cup J}(t)=
    \left(\prod_{j\in J} h_j(t)\right)^{-1}
    \left(f_{i_1}c_{I}+
    \sum_{l\in\{i_1+1,\ldots,k\}\setminus J} f_l(t)c_{J\cup\{l\}}(-1)^{j_l}\right).
\end{equation}

By condition~\eqref{itm:ex1}, $\sum_{i=1}^n \deg h_i=n$ and $\deg h_1\geq \cdots \geq \deg h_n$. Therefore 
\begin{equation} \label{eq:sumhi}
\sum_{i=1}^j \deg h_i\geq j
\end{equation}
for each $1\leq j\leq n$. Since $1\leq k\leq n$ and $\deg h_n>0$, we get
\begin{equation} \label{eq:degh1}
\deg h_{i_2}+\cdots+\deg h_{i_k}\leq \deg h_1+\cdots+\deg h_{k-1}\leq n-\deg h_n<n,
\end{equation}
and if $i_1\geq 1$, then by \eqref{eq:sumhi},
\begin{equation} \label{eq:degh2}
\deg h_{i_2}+\cdots+\deg h_{i_k}\leq  n-(\deg h_1+\cdots+\deg h_{i_1})\leq n-i_1.
\end{equation}

Since $c_I\neq 0$, by our assumptions \eqref{itm:ex2} and \eqref{itm:ex3}, 
\begin{equation} 
\deg\left(c_{I}f_{i_1}+
    \sum_{l\in\{i_1+1,\ldots,k\}\setminus J} (-1)^{j_l}c_{J\cup\{l\}}f_l(t)\right)
    \begin{array}{ll}
    \geq n & \text{if $i_1=0$}\\ 
     > n-{i_1} & \text{if $i_1\geq 1$}.
     \end{array} \label{eq:degf}
\end{equation}
By combining \eqref{eq:C_J}, \eqref{eq:degh1}, \eqref{eq:degh2} and \eqref{eq:degf}, we conclude that $\deg C_{\{0\}\cup J}(t)>0$. Hence, the conclusion of the lemma follows from \eqref{eq:CI}.
\end{proof}

\begin{lemma}\label{lem:ps group of our main example}
   Let $\{\rho(s):s\in\R\}$ be the P.S. group of $\varphi$ (Definition~\ref{def:p.s. group}). Then, there exists $v=(v_1,\ldots,v_n)\in\R^n\setminus\{0\}$, such that 
       \begin{equation} \label{eq:ps group of our main example}
         \rho(s):=\begin{bsmallmatrix}
    1&sv\\
       0& I_n\end{bsmallmatrix}\in\SL(n+1,\R),\, \forall s\in \R.
    \end{equation}
\end{lemma}
\begin{proof}
Due to Lemma~\ref{lemma:phi-expands}, by Lemma~\ref{lem:p.s. group of non-contracting is unip}, $\rho(s)$ is unipotent. 
Moreover, there exists $r<1$ such that  for any $s\in\R$, 
\begin{equation} \label{eq:rhos}
\varphi(h_{r,s}(t))\varphi(t)^{-1}\to \rho(s), 
\end{equation}where $h_{r,s}$ is as in \eqref{eq:change of speed r<1}.
We note that $\varphi(s)$ is contained in the subgroup $F$ of $\SL(n+1,\R)$ whose nonzero entries are only on the diagonal and in the top row. Therefore, by \eqref{eq:rhos}, $\rho(s)$ is a one-parameter unipotent subgroup of $F$. Hence, $\rho(s)$ is of the form given by \eqref{eq:ps group of our main example}.
\end{proof}

\begin{proof}[Proof of Proposition~\texorpdfstring{\ref{prop:main example for equidistribution}}{1.10}]
Let $\{e_0,\ldots,e_n\}$ denote the standard basis of $\R^{n+1}$. Let $(v_1,\ldots,v_n)\in\R^n$ be as in Lemma~\ref{lem:ps group of our main example}. Let $W_1$ denote the orthogonal complement of $v=\sum_{i=1}^n v_ie_i$ in the span of $\{e_1,\ldots,e_n\}$. Then the fixed points space of $\rho(s)$ in $\R^{n+1}$ is $W:=\R e_0+W_1$. 

Now let $\beta$ denote the correcting definable curve such that $\beta(t)\varphi(t)\subseteq H_\varphi$. Since $\beta(t)\to\beta(\infty)\in \SL(n+1,\R)$ as $t\to\infty$, by \eqref{eq:rhos}, we conclude that $\rho_1(s):=\beta(\infty)\rho(s)\beta(\infty)^{-1}\in H_\varphi$ for all $s\in \R$. By Lemma~\ref{lemma:phi-expands} and Lemma~\ref{lem:hull is semisimple if curve expands}, $H_\varphi$ is a semisimple group acting irreducibly on $\R^{n+1}$. Therefore, by Jacobson-Morosov theorem, there exists a homomorphism $\eta:\SL(2,\R)\to H_\varphi$ such that $\eta(u(s))=\rho_1(s)$, where
$u(s):=\begin{bsmallmatrix}
    1 & s \\0 &1
\end{bsmallmatrix}\in\SL(2,\R)$. 
As noted earlier, the fixed points space of $\rho_1$ on $\R^{n+1}$ is $\beta(\infty)W$, and it has dimension $n$. 
Expressing the representation of $\SL(2,\R)$ on $\R^{n+1}$ via $\eta$ as a direct sum of irreducible representations of $\SL_2(\R)$, and noting that the dimension of the $u(s)$-fixed subspace is $n$, we conclude that $\R^{n+1}=V_1\oplus V_2$, where $V_1$ is isomorphic to the standard representation of $\SL(2,\R)$ on $\R^2$, and $\SL(2,\R)$ acts trivially on $V_2$. Therefore, by \cite[Theorem A.7]{shah2023equidistribution}, either $H_\varphi=\SL(n+1,\R)$ or $n+1=2d$ for some $d\geq 2$ and $H_\varphi$ is conjugate to the standard symplectic group $\SP(\R^{2d})$. 

Consider the case when $H_\varphi\neq \SL(n+1,\R)$. Then, there exists a basis $w_1,w'_1,w_2,w'_2,\ldots,w_d,w'_d$ of $\R^{2d}$  such that $H_\varphi$ stabilizes
\[
\mathbf{v}=w_1\wedge w'_1+w_2\wedge w'_2+\cdots+w_n\wedge w'_n\in \wedge^2\R^{2d}.
\]
Now, by Lemma~\ref{lemma:phi-expands}, $\varphi(t)$ cannot be bounded in $\SL(2d,\R)/H_\varphi$, contradictions the assumption that $H_\varphi$ is the hull of $\varphi$. Therefore, $H_\varphi=\SL(n+1,\R)$. This completes the proof of Proposition~\ref{prop:main example for equidistribution}. 
\end{proof}

\begin{proof}[Proof of Corollary~\ref{cor:curve}] If $G$ is a linear Lie group, the result follows immediately from Proposition~\ref{prop:main example for equidistribution} and Theorem~\ref{thm: main equidistribution theorem}. If $G$ is not a linear Lie group, for our proof to work, it is sufficient that after the composition with the Adjoint representation of $G$ on its Lie algebra $\mathfrak{g}$, the curve is definable and non-contracting in $\GL(\mathfrak{g})$. This property is valid in our situation.
\end{proof}

 \appendix
 \section{Kempf's Lemma for real algebraic groups}
We note that \cite[Lemma 1.1(b)]{Kem78} is stated for reductive algebraic groups over an arbitrary field. We need the same result for any algebraic group defined over $\R$. Here, we verify that the proof given by Kempf is valid for real algebraic groups $G$ using the most naive language.

For finite-dimensional real vector spaces $V$ and $W$, we say that $P:V\to W$ is a \textit{polynomial map} if, after choosing bases for $V$ and $W$, the map $P$ in those coordinates is a multi-variable polynomial in each coordinate. We say that the map is $G$-equvairant if $g\cdot P(x)=P(g\cdot x)$ for all $g\in G$ and $x\in V$.

 \begin{lemma}[Kempf]\label{lem:kempf lemma adjusted}
    Let $G$ be the set of $\R$-points of an algebraic group defined over $\R$. Let $V$ be a rational representation of $G$. Suppose that for $v\in V$, it holds that $S:=\zcl(G\cdot v)\setminus G\cdot v$ is non-empty. Then there exists a rational representation $W$ and an equivariant polynomial map $P:V\to W$ such that $P(S)=\{0\}$ and $P(v)\neq 0$.
 \end{lemma}
 
 \begin{proof}
   First, note that since $\zcl (G\cdot v)$ and  $G \cdot v$ are $G$-invariant, we get that $S$ is $G$-invariant. Let $\R[V]$ be the coordinate ring on the vector space $V$ and consider $I\subseteq \R[V]$ the polynomial ideal of the polynomials vanishing on $S$. For $d\in \N$, let $I(d)\subseteq I$ be the subset of polynomials in $I$ with degree bounded by $d$. We may choose $d$ such that the ideal generated by $I(d)$ is $I$. Notice that $I(d)$ is a finite-dimensional vector space. Consider the right action of $G$ on $\R[V]$ defined by\begin{equation}\label{eq:right action on Id}
       (f. g)(x):=f(g\cdot x),~g\in G,~x\in V,~f\in \R[V].
   \end{equation}
   The action preserves the degree of a polynomial. So, it preserves $I(d)$ because $S$ is $G$-invariant. Let $I(d)^*$ be the space of linear functionals on $I(d)$. Using the representation \eqref{eq:right action on Id}, we get a left action on $I(d)^*$ defined by\begin{equation}\label{eq:left action on dual of Id}
       g\cdot l(f):=l(f\cdot g),~g\in G,~l\in I(d)^*,~f\in I(d).
   \end{equation}

 Now consider the map $P:V\to I(d)^*$ defined by\begin{equation}\label{eq:definition of poly equivariant map}P(x):=\text{ev}_x,  
 \end{equation}where $\text{ev}_x$ is the functional mapping $f\in I(d)$ to $f(x)$. It is straightforward to verify that $P$ is an equivariant polynomial map. Finally, we claim that $P(v)=\text{ev}_v$ is a nonzero functional. By the definition of $S$ and $P$, $P(x)=0$ for all $x\in S$. It remains to show that $P(v)\neq 0$. For that, recall the closed orbit lemma  \cite[Section 1.8]{Bo91}, which states that $S$ is Zariski closed. This implies that there exists a $p\in I(d)$ such that $p(S)=0$ and $p(v)\neq 0$. Since $P(v)(p)=p(v)$, $P(v)\neq 0$. 
\end{proof}
\section{\label{sec:Basic-notions-in o-minim}Basic notions in the theory
of o-minimal structures}

In order to keep our paper self contained, we will now recall  the basic definitions and results on o-minimal structures on the field of real numbers. We refer to \cite{van_den_dries_omin_book} and \cite{Geom_cat_and_o_min_str} for more details.

We begin with the notion of a structure \emph{on the ordered field of real numbers}, and later introduce the stronger notion of an o-minimal structure.
\begin{definition}
\label{def:structure-}\emph{A} \emph{structure $\St$ on the ordered field of real numbers} is a sequence of families of sets $$\St_d\subseteq \{\text{subsets of }\R^d\},~d\in\N,$$such that the following requirements are satisfied: 
\begin{enumerate}
\item For each $n\in\N$, $\mathscr S_n$ is a Boolean algebra of sets, and $\R^n\in \mathscr S_n$.
        \item For each $n\in\N$, $\mathscr S_n$ contains the diagonals $\{(x_1,\dots,x_n):x_i=x_j\},$ for all $1\leq i<j\leq n$.
        \item For each $n\in \N$, if $A\in \mathscr S_n$, then $A\times \R\in\mathscr S_{n+1}$ and $\R\times A\in\mathscr S_{n+1}$.
        \item For each $n\in \N$, if $A\in \mathscr S_{n+1}$, then $\pi(A)\in \mathscr S_n$, where $\pi:\R^{n+1}\to\R^n$ is the projection to the first $n$ coordinates.
        \item $\mathscr S_3$ contains the graph of addition $\{(x,y,x+y):x,y\in\R\}$ and the graph of multiplication $\{(x,y,xy):x,y\in\R\}$.
        
\end{enumerate} 
\end{definition}
Given a structure $\St$, we say
that $A\subseteq\R^{d}$ is definable if $A$ belongs to $\St_d$.
We say that a function $f:B\to\R^{m}$ is definable, for $B\subseteq \R^d$, 
if its graph $\{(x,f(x):x\in B\}$ belongs to $\St_{d+m}$.

We note the following basic properties, which are proved directly from Definition \ref{def:structure-}
\begin{lemma}
\label{lem:basic properties of structure}Let $\St$ be a structure on the field of real numbers. Then\emph{:}
\begin{enumerate}
\item The restriction of a definable function to a definable set is definable function on that set. 
\item \label{enu:For-any-structure}If $f:\R^{d}\to\R$ is definable, 
and $A\subseteq \R^d$, is definable, then the subsets: 
\[
\left\{ \mathbf{x}\in A\mid f(\mathbf{x})=0\right\} ,\ \left\{ \mathbf{x}\in A\mid f(\mathbf{x})>0\right\} 
\]
are definable.
\item If $f:A\to B$ and $g:B\to C$ are definable, then
the composition $g\circ f$ is also definable.
\item If $f,g:A\to\R$ are definable, then $f+g,$ $f-g,$ $f\cdot g$
are definable, and $f/g$ on the domain $\{x\in A: g(x)\neq0\}$
is definable.
\end{enumerate}
\end{lemma}

The notion of a structures belongs to model theory, and it is most natural to describe definable sets using first order logic formulae, see \cite[Appendix A]{Geom_cat_and_o_min_str} for a concise introduction. The advantage of the definition above is in allowing to define structures without resorting to notions of model theory. 

The proof of the statement below demonstrates the technique of proving the definability of functions using first-order logic formulae. The claim  is useful for us since the type of the function considered there is used in the proof of Theorem  \ref{thm:delta goodness for linear combinations of o-minimal}.

\begin{lemma} \label{lemma:sup-definable}
    Suppose that $f:\R^2\to \R$ is a bounded definable function, and let $a:\R\to\R$ and $b:\R\to \R$ be definable functions such that $a(s)< b(s),~\forall s$. Then the function $\Psi:\R\to\R$ defined by
    $$\Psi(s):=\sup\left\{|f(s,t)|:t\in \left[a(s),b(s)\right]\right\},\ \forall s\in \R,$$
    is definable.
\end{lemma}
\begin{proof}
    To prove the $\Psi$ is definable, we need to verify that the graph of $\Psi$, $\Ga(\Psi):=\{(s,\Psi(s)):s\in \R\}$ is definable. 
    Consider the definable set:
    $$A_1:=\{(s,y,t):a(s)\leq t\leq b(s)\text{ and }|f(s,t)|>y\},$$
    and consider the projection $\pi_{3,2}:\R^3\to\R^2$ to the first coordinates. Then  $B_1:=(\pi_{2,3}(A_1))^c$ is definable, and it can be alternatively described by:
    $$B_1:=\{(s,y):(\forall t\in[a(s),b(s)])\  |f(s,t)|\leq y\}.$$
    In a similar way (while slightly complicated), the set:
    $$B_2:=\{(s,y):(\forall \e>0)(\exists t\in[a(s), b(s)]) \ y-\epsilon\leq |f(s,t)|\}$$
    is definable.
    The graph of $\Psi$ is then $B_1 \cap B_2$.
    \end{proof}

\subsection{O-minimal structures}
The notion of o-minimal structures was introduced in \cite{generalization_tarski_seind_dries}, as a generalization of the real semi-algebraic geometry.
\begin{definition}(\emph{o-minimal structures})
A structure $\St$  expanding the real field is called o-minimal if $\St_1$ consists of  all possible sets which are finite unions of points and intervals.
\end{definition}

The semi-algebraic sets are definable in every o-minimal structure. Here, the semi-algebraic sets are boolean combinations of sets of the form $\{x\in \R^d: P(x)=0\}$, $\{x\in \R^d: P(x)>0\}$, where $P\in\R[X_1,\dots,X_d]$. By the Tarski-Seidenberg theorem, the semi-algebraic sets form an o-minimal structure.

We collect the following important results, see  \cite{van_den_dries_omin_book} for detailed proofs.  For the following, fix an o-minimal structure.
\begin{theorem}\label{thm:monotonicity thm}\emph{\cite[Ch. 3, Thm. 1.2 and Ch. 7, Prop. 2.5]{van_den_dries_omin_book}} 
 Suppose that $f:(a,b)\to\R$ is definable, 
 where $-\infty\leq a<b\leq\infty$.  
 Then, there exist $a_0,a_1,...,a_{k+1}$ such that
 $a =: a_0 < a_1 < \dots < a_k < a_{k+1} := b$ and such that in each sub-interval $I=(a_i,a_{i+1})$, 
 $f|_I$ is differentiable   and $f|_I$ is either strictly monotonic, or constant. 
\end{theorem}
A direct consequence of the monotonicity theorem is that a definable function $f:\R\to\R$ either converges as $t\to \infty$ or diverges to $\infty$ or to $-\infty$.
\begin{theorem}
\label{thm:uniform bound on fibers}\emph{\cite[Ch.3,  Cor. 3.7]{van_den_dries_omin_book}}. If $A\subseteq \R^m$ is definable, than $A$ has finitely many connected components. Moreover, the following uniformity holds.
For $A\subseteq\R^d\times \R^n$  and $x_0\in\R^d$, consider $$A_{x_0}:=\{(x_0,y):(x_0,y)\in A\}.$$  Then, for a definable $A$, there exists an $N\in\N$ such for all $x\in\R^{m}$,
the set $A_{x}$ has at most $N$ connected components.
\end{theorem}

\begin{theorem}
\label{thm:choice function}\emph{\cite[Ch. 3, Prop. 1.2]{van_den_dries_omin_book}}.
Suppose that $A\subseteq \R^{m+n}$ is definable. Let $\pi_{m+n,m}:\R^{m+n}\to\R^m$ be the projection to the first $m$ coordintates. Then, there is a definable function $f:\pi_{m+n,m}(A)\to\R^{n}$
such that the graph of $f$ is contained in $A$.
\end{theorem}

\printbibliography[title={Bibliography}]
\end{document}